\documentclass[11pt]{article}
\usepackage{amssymb}
\usepackage{amsmath}
\usepackage{mathrsfs}
\usepackage{graphics}
\usepackage{graphicx}
\usepackage{xcolor}
\usepackage{subfigure}
\usepackage[T1]{fontenc}
\usepackage{latexsym,amssymb,amsmath,amsfonts,amsthm}\usepackage{txfonts}
\newcommand{\be}{\begin{eqnarray}}
\newcommand{\ben}{\begin{eqnarray*}}
\newcommand{\en}{\end{eqnarray}}
\newcommand{\enn}{\end{eqnarray*}}

\newtheorem{theorem}{Theorem}[section]
\newtheorem{lemma}{Lemma}[section]
\newtheorem{prp}[theorem]{Proposition}
\newtheorem{thm}[theorem]{Theorem}
\newtheorem{cor}[theorem]{Corollary}
\newtheorem{dfn}{Definition}[section]

\newtheorem{remark}{Remark}

\begin{document}
\renewcommand{\theequation}{\arabic{section}.\arabic{equation}}
\begin{titlepage}
\title{\bf Large deviation principles for first-order scalar conservation laws
with stochastic forcing}
\author{ Zhao Dong$^{1}$,\ Jiang-Lun Wu$^{2}$, Rangrang Zhang$^{3,}\footnote{Corresponding author.}$, Tusheng Zhang$^{4}$\\
{\small $^1$ RCSDS, Academy of Mathematics and Systems Science, Chinese Academy of Sciences, Beijing 100190, China}\\
{\small $^2$  Department of  Mathematics,
Swansea University, Bay Campus, Swansea SA1 8EN, UK}\\
{\small $^3$  School of Mathematics and Statistics,
Beijing Institute of Technology, Beijing 100081, China}\\
{\small $^4$ School of Mathematics, University of Manchester, Oxford Road, Manchester M13 9PL, England, UK}\\
({\small{\sf dzhao@amt.ac.cn},\ {\sf j.l.wu@swansea.ac.uk},\ {\sf rrzhang@amss.ac.cn}, \ {\sf tusheng.zhang@manchester.ac.uk}})}
\date{}
\end{titlepage}
\maketitle

\noindent\textbf{Abstract}:
 In this paper, we established the Freidlin-Wentzell type large deviation principles for first-order scalar conservation laws perturbed by small multiplicative noise. Due to the lack
 of the viscous terms in the stochastic equations, the kinetic solution to the Cauchy problem for these first-order conservation laws is studied. Then, based on the well-posedness
 of the kinetic solutions, we show that the large deviations holds by utilising the weak convergence approach.



\section{Introduction}
This paper concerns the asymptotic behaviour of stochastic scalar conservation laws with small multiplicative noise. The (deterministic) conservation laws (in both scalar and vectorial) are fundamental to our understanding of the space-time evolution laws of interesting physical quantities, in that they {describe} (dynamical) processes {that} can or cannot occur in nature. Mathematically or statistically, such physical laws should incorporate with noise influences, due to the lack of knowledge of certain physical parameters as well as bias or incomplete measurements arising in {experiments} or modeling. More precisely, fix any
$T>0$ and let $(\Omega,\mathcal{F},\mathbb{P},\{\mathcal{F}_t\}_{t\in
[0,T]},(\{\beta_k(t)\}_{t\in[0,T]})_{k\in\mathbb{N}})$ be a stochastic basis. Without loss of generality, here the filtration $\{\mathcal{F}_t\}_{t\in [0,T]}$ is assumed to be complete and $\{\beta_k(t)\}_{t\in[0,T]},k\in\mathbb{N}$, are independent (one-dimensional)  $\{\mathcal{F}_t\}_{t\in [0,T]}-$Wiener processes. We use $\mathbb{E}$ to denote the expectation with respect to $\mathbb{P}$.
Fix any $N\in\mathbb{N}$, let $\mathbb{T}^N\subset\mathbb{R}^N$ denote the $N-$dimensional torus (suppose the periodic length is $1$).
We are concerned with the following scalar conservation law with stochastic forcing
$$du+div(A(u))dt=\Phi(u) dW(t) \quad {\rm{in}} \ \mathbb{T}^N\times[0,T]$$
for a random field $u:(\omega,x,t)\in\Omega\times\mathbb{T}^N\times[0,T]\mapsto u(\omega,x,t):=u(x,t)\in\mathbb{R}$,
that is, the equation is periodic in the space variable $x\in \mathbb{T}^N$,
where the flux function $A:\mathbb{R}\to\mathbb{R}^N$ and the coefficient $\Phi:\mathbb{R}\to\mathbb{R}$ are measurable and fulfill certain conditions specified later,
and $W$ is a cylindrical Wiener process defined on a given (separable) Hilbert space $U$ with
the form $W(t)=\sum_{k\geq 1}\beta_k(t) e_k,t\in[0,T]$, where $(e_k)_{k\geq 1}$ is a complete orthonormal base in the Hilbert space $U$. We consider the following Cauchy problem
\begin{eqnarray}\label{P-19}
\left\{
  \begin{array}{ll}
  du+div(A(u))dt=\Phi(u) dW(t) \quad {\rm{in}}\ \mathbb{T}^N\times (0,T],\\
u(\cdot,0)=u_0(\cdot) \quad {\rm{on}}\ \mathbb{T}^N.
  \end{array}
\right.
\end{eqnarray}

For the deterministic case, i.e., $\Phi\equiv0$, (\ref{P-19}) is well studied in the PDEs literature, see e.g. the monograph \cite{Dafermos} and the most recent reference Ammar,
Wittbold and Carrillo \cite{K-P-J} (and references therein). As well known, the Cauchy problem
for the deterministic first-order PDE (\ref{P-19}) does not admit any (global) smooth solutions, but there exist infinitely many weak solutions to the deterministic Cauchy problem and an additional entropy condition has to be added to get the uniqueness and further to identify the physical weak solution. The notion of entropy solutions for the deterministic problem in the $L^{\infty}$ framework was initiated by Otto in \cite{O}. Moreover, Porretta and Vovelle \cite{P-V} studied the problem in the $L^1$  setting, that is, the solutions are allowed to be unbounded. In order to deal with unbounded solutions, they defined a notion of renormalized entropy solutions which generalizes Otto's original definition of entropy solutions. The kinetic formulation of weak entropy solution of the Cauchy problem for a general multidimensional scalar conservation law, named as the kinetic system, is derived by Lions, Perthame and Tadmor in \cite{L-P-T}. They further discussed the relationship between entropy solutions and the kinetic system.

Having a stochastic forcing term in (\ref{P-19}) is very natural and important for various modeling problems arising in a wide variety of fields, e.g., physics, engineering, biology and so on. The Cauchy problem for the stochastic equation (\ref{P-19}) driven by additive noise has been studied by Kim in \cite{K} wherein the author proposed a method of compensated compactness to prove the existence of a stochastic weak entropy solution via vanishing viscosity approximation. Moreover, a Kruzkov-type method was used there to prove the uniqueness. Furthermore, Vallet and Wittbold \cite{V-W} extended the results of Kim to the multi-dimensional Dirichlet problem with additive noise. By utilising the vanishing viscosity method, Young measure techniques, and Kruzkov doubling variables technique, they managed to show the existence and uniqueness of the stochastic entropy solutions. Concerning the case of the
equation with multiplicative noise, for Cauchy problem over the whole spatial space, Feng and Nualart \cite{F-N} introduced a notion of strong entropy solutions in order to prove the uniqueness of  the entropy solution. Using the vanishing viscosity and compensated compactness arguments, they established the existence of stochastic strong entropy solutions only in the one-dimensional space case. On the other hand, using a kinetic formulation, Debussche and Vovelle \cite{D-V-1} solved the Cauchy problem for (\ref{P-19}) in any dimension. They made use of a notion of kinetic solutions  developed by Lions, Perthame and Tadmor for deterministic, first-order scalar conservation laws in \cite{L-P-T}. In view of the equivalence between kinetic formulation and entropy solution, they obtained the existence and uniqueness of the entropy solutions. The long-time behavior of periodic scalar first-order conservation laws with additive stochastic forcing under an hypothesis of non-degeneracy of the flux function is studied by Debussche and Vovelle in \cite{D-V-2}. For sub-cubic fluxes, they show the existence of an invariant measure. Moreover, for sub-quadratic fluxes, they prove the uniqueness and ergodicity of the invariant measure.

From statistical mechanics point of view, asymptotic analysis for vanishing the noise force is
important and interesting for studying stochastic conservation laws, in which establishing large deviation principles is a core step for finer analysis as well as gaining deeper insight for the described physical evolutions. Due to lack of second order elliptic operators for the space variable, the asymptotic analysis for stochastic conservation laws is really challenging and all those existing approaches for establishing large deviation principles seem unapplicable. To our knowledge, Mariani \cite{Mariani} (see also \cite{Mariani0} for more {details}) is the first work towards large deviations for stochastic conservation laws, wherein the author considered a family of stochastic conservation laws as parabolic SPDEs with additional small viscosity term and small (spatially) regularized (i.e., spatially smoothing) noises. By a very interesting scaling procedure and deep insightful observations from interacting particle {systems}, Mariani has succeeded to establish large deviation principles by vanishing viscosity and noise terms simultaneously in a smart choice of scalings. While, large deviations for the stochastic first-order conservation laws remain open. Due to the fact that the entropy  solutions are living in rather irregular spaces comparing to various type solutions for parabolic SPDEs, it is indeed a challenge to establish large deviation principles for the first-order conservation laws with general noise force.

The purpose of this paper is to prove the Freidlin-Wentzell type large deviation principle (LDP) for the first-order stochastic scalar conservation law in $L^1([0,T];L^1(\mathbb{T}^N))$, which provides the exponential decay of small probabilities associated with the corresponding stochastic dynamical systems with small noise. An important tool for studying the Freidlin-Wentzell's LDP is the weak convergence approach, which is developed by Dupuis and Ellis in \cite{DE}. The key idea of this approach is to prove certain variational representation formula about the Laplace transform of bounded continuous functionals, which then leads to the verification of the equivalence between the LDP and the Laplace principle. In particular, for Brownian functionals, an elegant variational representation formula has been established by Bou\'{e} and Dupuis in \cite{MP} and by Budhiraja and Dupuis in \cite{BD}. Recently, a sufficient condition to verify the large deviation criteria of Budhiraja, Dupuis and Maroulas for functionals of Brownian motions is proposed by Matoussi, Sabbagh and Zhang in \cite{MSZ}, which turns out to be more suitable for SPDEs arising from fluid mechanics. Thus, in the present paper, we adopt this new sufficient condition.

Our proof strategy mainly consists of the following procedures. As an important part of the proof, we need to obtain the global well-posedness of the associated skeleton equations. For showing the uniqueness, we apply the doubling of variables method. For showing the existence result, we first apply the vanishing viscosity method to construct a sequence of approximating equations as in \cite{D-V-1}.  Then, we prove that the family of the solutions of the approximating equations is compact in an appropriate space and that any limit of the approximating solutions gives rise to a solution of the associated skeleton equation.
To complete the proof of the large deviation principle, we also need to study the weak convergence of the small noise perturbations of the problem (\ref{P-19}) in the random directions of the Cameron-Martin space of the driving Brownian motions. To verify the convergence of the randomly perturbed equation to the corresponding unperturbed equation in $L^1([0,T];L^1(\mathbb{T}^N))$, the doubling of variables method plays a key role.

The rest of the paper is organised as follows. The mathematical formulation of stochastic scalar conservation laws is presented in Section 2. In Section 3, we introduce the weak convergence method and state our main result. Section 4 is devoted to the study of the associated skeleton equations. The large deviation principle is proved in Section 5.

\section{Preliminaries}
Let $\mathcal{L}(K_1,K_2)$ (resp. $\mathcal{L}_2(K_1,K_2)$) be the space of bounded (resp. Hilbert-Schmidt) linear operators from a Hilbert space $K_1$ to another Hilbert space $K_2$, whose norm is denoted by $\|\cdot\|_{\mathcal{L}(K_1, K_2)}$(resp. $\|\cdot\|_{\mathcal{L}_2(K_1, K_2)})$. Further, $C_b$ represents the space of bounded, continuous functions and $C^1_b$ stands for the space of bounded, continuously differentiable functions having bounded first order derivative. Let $\|\cdot\|_{L^p}$ denote the norm of Lebesgue space $L^p(\mathbb{T}^N)$ for $p\in (0,\infty]$. In particular, set $H=L^2(\mathbb{T}^N)$ with the corresponding norm $\|\cdot\|_H$. For all $a\geq0$, let
$H^a(\mathbb{T}^N)=W^{a,2}(\mathbb{T}^N)$ be the usual Sobolev space of order $a$ with the norm
\[
\|u\|^2_{H^a}=\sum_{|\alpha|=|(\alpha_1,...,\alpha_N)|=\alpha_1+\cdots+\alpha_N\leq a}\int_{\mathbb{T}^N}|D^{\alpha}u(x)|^2dx.
\]
 $H^{-a}(\mathbb{T}^N)$  stands for the topological dual of $H^a(\mathbb{T}^N)$, whose norm is denoted by $\|\cdot\|_{H^{-a}}$. Moreover, we use the brackets $\langle\cdot,\cdot\rangle$ to denote the duality between $C^{\infty}_c(\mathbb{T}^N\times \mathbb{R})$ and the space of distributions over $\mathbb{T}^N\times \mathbb{R}$.
Similarly, for $1\leq p\leq \infty$ and $q:=\frac{p}{p-1}$, the conjugate exponent of $p$, we denote
\[
\langle F, G \rangle:=\int_{\mathbb{T}^N}\int_{\mathbb{R}}F(x,\xi)G(x,\xi)dxd\xi, \quad F\in L^p(\mathbb{T}^N\times \mathbb{R}), G\in L^q(\mathbb{T}^N\times \mathbb{R}),
\]
and also for a measure $m$ on the Borel measurable space $\mathbb{T}^N\times[0,T]\times \mathbb{R}$
\[
m(\phi):=\langle m, \phi \rangle:=\int_{\mathbb{T}^N\times[0,T]\times \mathbb{R}}\phi(x,t,\xi)dm(x,t,\xi), \quad  \phi\in C_b(\mathbb{T}^N\times[0,T]\times \mathbb{R}).
\]
\subsection{Hypotheses}
For the flux function $A$ and the coefficient $\Phi$ of (\ref{P-19}), we assume
\begin{description}
  \item[\textbf{Hypothesis H}] The flux function $A$ belongs to $C^2(\mathbb{R};\mathbb{R}^N)$ and its derivative $a$ has at most polynomial growth. That is, there exist constants $C>0, p>1$ such that
     \begin{eqnarray}\label{qeq-22}
     |a(\xi)-a(\zeta)|\leq \Gamma(\xi,\zeta)|\xi-\zeta|, \quad \Gamma(\xi,\zeta)=C(1+|\xi|^{p-1}+|\zeta|^{p-1}).
      \end{eqnarray}
      For each $u\in \mathbb{R}$, the map $\Phi(u): U\rightarrow H$ is defined by $\Phi(u) e_k=g_k(\cdot, u)$, where each $g_k(\cdot,u)$ is a regular function on $\mathbb{T}^N$.
      More precisely, we assume that $g_k\in C(\mathbb{T}^N\times \mathbb{R})$ with the following bounds
\begin{eqnarray}\label{equ-28}
G^2(x,u)=\sum_{k\geq 1}|g_k(x,u)|^2&\leq& D_0(1+|u|^2),\\
\label{equ-29}
\sum_{k\geq 1}|g_k(x,u)-g_k(y,v)|^2&\leq& D_1\Big(|x-y|^2+{|u-v|^2}\Big),
\end{eqnarray}
for $x, y\in \mathbb{T}^N, u,v\in \mathbb{R}$.

\end{description}
Since $\|g_k\|_{H}\leq\|g_k\|_{C(\mathbb{T}^N)}$, we deduce that $\Phi(u)\in \mathcal{L}_2(U,H)$, for each $u\in \mathbb{R}$. Moreover, it follows from (\ref{equ-28}) and (\ref{equ-29}) that
\begin{eqnarray}\label{equ-30}
\|\Phi(u)\|^2_{\mathcal{L}_2(U,H)}&\leq& D_0(1+\|u\|^2_H),\\
\label{equ-30-1}
\|\Phi(u)-\Phi(v)\|^2_{\mathcal{L}_2(U,H)}&\leq& D_1\|u-v\|^2_H.
\end{eqnarray}
\subsection{Kinetic solution and generalized kinetic solution}
Let us recall the notion of a solution to equation (\ref{P-19}) from \cite{D-V-1, D-V-2}. Keeping in mind that we are working on the stochastic basis $(\Omega,\mathcal{F},\mathbb{P},\{\mathcal{F}_t\}_{t\in [0,T]},(\beta_k(t))_{k\in\mathbb{N}})$.
\begin{dfn}(Kinetic measure)\label{dfn-3}
 A map $m$ from $\Omega$ to the set of non-negative, finite measures over $\mathbb{T}^N\times [0,T]\times\mathbb{R}$ is said to be a kinetic measure, if
\begin{description}
  \item[1.] $ m $ is measurable, that is, for each $\phi\in C_b(\mathbb{T}^N\times [0,T]\times \mathbb{R}), \langle m, \phi \rangle: \Omega\rightarrow \mathbb{R}$ is measurable,
  \item[2.] $m$ vanishes for large $\xi$, i.e.,
\begin{eqnarray}\label{equ-37}
\lim_{R\rightarrow +\infty}\mathbb{E}[m(\mathbb{T}^N\times [0,T]\times B^c_R)]=0,
\end{eqnarray}
where $B^c_R:=\{\xi\in \mathbb{R}, |\xi|\geq R\}$
  \item[3.] for every $\phi\in C_b(\mathbb{T}^N\times \mathbb{R})$, the process
\[
(\omega,t)\in\Omega\times[0,T]\mapsto \int_{\mathbb{T}^N\times [0,t]\times \mathbb{R}}\phi(x,\xi)dm(x,s,\xi)\in\mathbb{R}
\]
is predictable.
\end{description}
\end{dfn}
Let $\mathcal{M}^+_0(\mathbb{T}^N\times [0,T]\times \mathbb{R})$ be the space of all bounded, nonnegative random measures $m$ satisfying (\ref{equ-37}).

\begin{dfn}(Kinetic solution)\label{dfn-1}
Let $u_0\in L^{\infty}(\mathbb{T}^N)$. A measurable function $u: \mathbb{T}^N\times [0,T]\times\Omega\rightarrow \mathbb{R}$ is called a kinetic solution to (\ref{P-19}) with initial datum $u_0$, if
\begin{description}
  \item[1.] $(u(t))_{t\in[0,T]}$ is predictable,
  \item[2.] for any $p\geq1$, there exists $C_p\geq0$ such that
\[
\mathbb{E}\left(\underset{0\leq t\leq T}{{\rm{ess\sup}}}\ \|u(t)\|^p_{L^p(\mathbb{T}^N)}\right)\leq C_p,
\]
\item[3.] there exists a kinetic measure $m$ such that $f:= I_{u>\xi}$ satisfies: for all $\varphi\in C^1_c(\mathbb{T}^N\times [0,T)\times \mathbb{R})$,
\begin{eqnarray}\notag
&&\int^T_0\langle f(t), \partial_t \varphi(t)\rangle dt+\langle f_0, \varphi(0)\rangle +\int^T_0\langle f(t), a(\xi)\cdot \nabla \varphi (t)\rangle dt\\
\label{P-21}
&=& -\sum_{k\geq 1}\int^T_0\int_{\mathbb{T}^N} g_k(x)\varphi (x,t,u(x,t))dxd\beta_k(t) \\ \notag
&& -\frac{1}{2}\sum_{k\geq 1}\int^T_0\int_{\mathbb{T}^N}\partial_{\xi}\varphi (x,t,u(x,t))G^2(x)dxdt+ m(\partial_{\xi} \varphi), \ a.s. ,
\end{eqnarray}
where $f_0=I_{u_0>\xi}$, $u(t)=u(\cdot,t,\cdot)$, $G^2=\sum^{\infty}_{k=1}|g_k|^2$ and $a(\xi):=A'(\xi)$.
\end{description}
\end{dfn}

In order to prove the existence of a kinetic solution, the generalized kinetic solution was introduced in \cite{D-V-1}.
\begin{dfn}(Young measure)
 Let $(X,\lambda)$ be a finite measure space. Let $\mathcal{P}_1(\mathbb{R})$ denote the set of all (Borel) probability measures on $\mathbb{R}$. A map $\nu:X\to\mathcal{P}_1(\mathbb{R})$ is
said to be a Young measure on $X$, if for each $\phi\in C_b(\mathbb{R})$, the map $z\in X\mapsto \nu_z(\phi)\in\mathbb{R}$ is measurable. Next, we say that a Young measure $\nu$ vanishes at infinity if, for each  $p\geq 1$, the following holds
\begin{eqnarray}\label{equ-26}
\int_X\int_{\mathbb{R}}|\xi|^pd\nu_z(\xi)d\lambda(z)<+\infty.
\end{eqnarray}

\end{dfn}
\begin{dfn}(Kinetic function)
Let $(X,\lambda)$ be a finite measure space. A measurable function $f: X\times \mathbb{R}\rightarrow [0,1]$
is called a kinetic function, if there exists a Young measure $\nu$ on $X$ that vanishes at infinity such that $\forall\xi\in \mathbb{R}$
\[
f(z,\xi)=\nu_z(\xi,+\infty)
\]
holds for $\lambda-a.e.$ $z\in X,$.
We say that $f$ is an equilibrium if there exists a measurable function $u: X\rightarrow \mathbb{R}$ such that $f(z,\xi)=I_{u(z)>\xi}$ a.e., or equivalently, $\nu_z=\delta_{u(z)}$ for $\lambda- a.e. \ z\in X$.
\end{dfn}
 Let $f: X\times \mathbb{R}\rightarrow [0,1]$ be a kinetic function, we use $\bar{f}$ to denote its conjugate function $\bar{f}:=1-f$.

\begin{dfn}(Generalized kinetic solution)\label{dfn-2}
Let $f_0:\Omega\times\mathbb{T}^N\times \mathbb{R}\rightarrow [0,1]$ be a kinetic function with
$(X,\lambda)=(\Omega\times\mathbb{T}^N,\mathbb{P}\otimes dx)$. A measurable function $f:\Omega\times\mathbb{T}^N\times[0,T]\times\mathbb{R}\rightarrow[0,1]$ is said to be a generalized kinetic solution to (\ref{P-19}) with initial datum $f_0$, if
\begin{description}
  \item[1.] $(f(t))_{t\in[0,T]}$ is predictable,
 \item[2.] $f$ is a kinetic function with $(X,\lambda)=(\Omega\times\mathbb{T}^N\times[0,T],\mathbb{P}\otimes dx\otimes dt)$ and for any $p\geq1$, there exists a constant $C_p>0$ such that $\nu:=-\partial_{\xi} f$ fulfills the following
\begin{eqnarray}
\mathbb{E}\left(\underset{0\leq t\leq T}{{\rm{ess\sup}}}\ \int_{\mathbb{T}^N}\int_{\mathbb{R}}|\xi|^pd\nu_{x,t}(\xi)dx\right)\leq C_p,
\end{eqnarray}
\item[3.] there exists a kinetic measure $m$ such that for all $\varphi\in C^1_c(\mathbb{T}^N\times [0,T)\times \mathbb{R})$,
\begin{eqnarray}\notag
&&\int^T_0\langle f(t), \partial_t \varphi(t)\rangle dt+\langle f_0, \varphi(0)\rangle +\int^T_0\langle f(t), a(\xi)\cdot \nabla \varphi (t)\rangle dt\\ \notag
&=& -\sum_{k\geq 1}\int^T_0\int_{\mathbb{T}^N}\int_{\mathbb{R}} g_k(x)\varphi (x,t,\xi)d\nu_{x,t}(\xi)dxd\beta_k(t) \\
\label{P-22}
&& -\frac{1}{2}\int^T_0\int_{\mathbb{T}^N}\int_{\mathbb{R}}\partial_{\xi}\varphi (x,t,\xi)G^2(x)d\nu_{x,t}(\xi)dxdt+ m(\partial_{\xi} \varphi),\  a.s..
\end{eqnarray}
\end{description}
\end{dfn}
Referring to \cite{D-V-1}, almost surely, any generalized solution admits possibly different left and right weak limits at any point $t\in[0,T]$. This property is important for establishing a comparison principle which allows to prove uniqueness. The following result is proved in \cite{D-V-1}.
\begin{prp}(Left and right weak limits)\label{prp-3} Let $f_0$ be a kinetic initial datum and $f$ be a generalized kinetic solution to (\ref{P-19}) with initial $f_0$. Then $f$ admits, almost surely, left and right limits respectively at every point $t\in [0,T]$. More precisely, for any  $t\in [0,T]$, there exist kinetic functions $f^{t\pm}$ on $\Omega\times \mathbb{T}^N\times \mathbb{R}$ such that $\mathbb{P}-$a.s.
\[
\langle f(t-\varepsilon),\varphi\rangle\rightarrow \langle f^{t-},\varphi\rangle
\]
and
\[
\langle f(t+\varepsilon),\varphi\rangle\rightarrow \langle f^{t+},\varphi\rangle
\]
as $\varepsilon\rightarrow 0$ for all $\varphi\in C^1_c(\mathbb{T}^N\times \mathbb{R})$. Moreover, almost surely,
\[
\langle f^{t+}-f^{t-}, \varphi\rangle=-\int_{\mathbb{T}^N\times[0,T]\times \mathbb{R}}\partial_{\xi}\varphi(x,\xi)I_{\{t\}}(s)dm(x,s,\xi).
\]
In particular, almost surely, the set of $t\in [0,T]$ fulfilling that $f^{t+}\neq f^{t-}$ is countable.
\end{prp}
For a generalized kinetic solution $f$, define $f^{\pm}$ by $f^{\pm}(t)=f^{t \pm}$, $t\in [0,T]$. Since we are dealing with the filtration associated to Brownian motion, both $f^{\pm}$ are  clearly predictable as well. Also $f=f^+=f^-$ almost everywhere in time and we can take any of them in an integral with respect to the Lebesgue measure or in a stochastic integral. However, if the integral is with respect to a measure--typically a kinetic measure in this article, the integral is not well defined for $f$ and may differ if one chooses either $f^+$ or $f^-$. In addition, referring to (22) in \cite{D-V-1},
the weak form (\ref{P-22}) satisfied by a generalized kinetic solution $f$ can be strengthened to be weak only respect to $x$ and $\xi$. Concretely, for all $t\in [0,T)$ and $\psi\in C^1_c(\mathbb{T}^N\times \mathbb{R})$,
 \begin{eqnarray}\notag
\langle f^+(t),\psi\rangle&=&\langle f_{0}, \psi\rangle+\int^t_0\langle f(s), a(\xi)\cdot \nabla \psi\rangle ds\\
\notag
&&+\sum_{k\geq 1}\int^t_0\int_{\mathbb{T}^N}\int_{\mathbb{R}}g_k(x,\xi)\psi(x,\xi)d\nu_{x,s}(\xi)dxd\beta_k(s)\\
\label{e-14}
&& +\frac{1}{2}\int^t_0\int_{\mathbb{T}^N}\int_{\mathbb{R}}\partial_{\xi}\psi(x,\xi)G^2(x,\xi)d\nu_{x,s}(\xi)dxds- \langle m,\partial_{\xi} \psi\rangle([0,t]), \quad a.s.,
\end{eqnarray}
 and set $f^+(T)=f(T)$.

At the end of this subsection, as a special example, let us consider the following
stochastic heat equation on $\mathbb{T}^N\times [0, T)$:

\begin{eqnarray}\label{qeq-1}
du-\Delta udt=\Phi(u)dW(t),\quad u(0)=u_0.
\end{eqnarray}
We aim to derive an explicit expression of its kinetic measure $m$. For this,
we have the following kinetic formulation

\begin{prp}\label{prppp-1}
Let $u_0\in L^{\infty}(\mathbb{T}^N)$ and $u$ be the solution to (\ref{qeq-1}). Then $f:=I_{u>\xi}$ satisfies the following

\begin{eqnarray}\notag
&&\int^T_0\langle f(t), \partial_t \varphi(t)\rangle dt+\langle f_0, \varphi(0)\rangle+\int^T_0\langle f(t),\Delta\varphi(t)\rangle dt\\ \notag
&=& -\sum_{k\geq 1}\int^T_0\int_{\mathbb{T}^N}\int_{\mathbb{R}}g_k(x,\xi)\varphi(x,t,\xi)d\nu_{x,t}(\xi)dxd\beta_k(t)\\
\label{qeq-2}
&&\ -\frac{1}{2}\int^T_0\int_{\mathbb{T}^N}\int_{\mathbb{R}}\partial_{\xi}\varphi(x,t,\xi)G^2(x,\xi)d\nu_{x,t}(\xi)dxdt+m(\partial_{\xi}\varphi), a.s.
\end{eqnarray}
for all $\varphi\in C^1_c(\mathbb{T}^N\times [0,T)\times \mathbb{R})$, where $f_0(\xi)=I_{u_0>\xi}$ and for all $\phi\in C_b(\mathbb{T}^N\times [0,T]\times \mathbb{R})$,
\begin{eqnarray*}
d\nu_{x,t}(\xi)=\delta_{u=\xi}d\xi, \quad m(\phi)=\int^T_0\int_{\mathbb{T}^N}\phi(x,t, u(x,t))|\nabla u|^2dxdt.
\end{eqnarray*}
\end{prp}

{
\begin{proof}
By It\^{o} formula, we have for $\theta\in C^2(\mathbb{R})$ with polynomial growth at $\pm \infty$,
\begin{eqnarray*}
d(I_{u>\xi}, \theta')&=&d\int_{\mathbb{R}}I_{u>\xi}\theta'(\xi)d\xi=d\theta(u)\\
&=& \theta'(u)(\Delta udt+\Phi(u)dW(t))+\frac{1}{2}\theta''(u)G^2dt,
\end{eqnarray*}
where $G^2=\sum_{k\geq 1}|g_k|^2$.

The first term can be rewritten as
\begin{eqnarray*}
\theta'(u)\Delta u=\Delta \theta(u)-|\nabla u|^2\theta''(u)=\Delta (I_{u>\xi}, \theta')+(\partial_{\xi}(|\nabla u|^2\delta_{u=\xi}), \theta').
\end{eqnarray*}
Hence, we obtain the following kinetic formulation:
\begin{eqnarray*}
d(I_{u>\xi}, \theta')&=& \Delta (I_{u>\xi}, \theta')dt+(\partial_{\xi}(|\nabla u|^2\delta_{u=\xi}-\frac{1}{2}G^2\delta_{u=\xi}), \theta')dt\\
&&\ +\sum_{k\geq 1}(\delta_{u=\xi}g_k, \theta')d\beta_k.
\end{eqnarray*}
Taking $\theta(\xi)=\int^{\xi}_{-\infty}\chi$, we have
\begin{eqnarray*}
d(I_{u>\xi}, \chi)&=& \Delta (I_{u>\xi}, \chi)dt+(\partial_{\xi}(|\nabla u|^2\delta_{u=\xi}-\frac{1}{2}G^2\delta_{u=\xi}),\chi)dt\\
&&\ +\sum_{k\geq 1}(\delta_{u=\xi}g_k, \chi)d\beta_k.
\end{eqnarray*}
Since the test functions $\varphi(x,\xi)=\alpha(x)\chi(\xi)$ form a dense subset of $C^{\infty}_{c}(\mathbb{T}^N\times \mathbb{R})$, it follows that (\ref{qeq-2}) holds. We complete the proof.
\end{proof}}

{ From above, it is clear that the kinetic measure $m$ has an explicit expression
$$m=|\nabla u|^2\delta_{u=\xi}.$$
}

\subsection{Compactness results}\label{l-1}
Recall the following two compactness results from \cite{D-V-1}, which are important for establishing the existence of generalized kinetic solution of (\ref{P-19}).
\begin{thm}\label{thm-6}
(Compactness of Young measures)\quad Let $(X,\lambda)$ be a finite measure space. Let $(\nu^n)$ be a sequence of Young measures on $X$ satisfying the condition (\ref{equ-26}) for some $p\geq1$, namely,
\begin{eqnarray}\label{equ-19}
\sup_{n\in\mathbb{N}} \int_X\int_{\mathbb{R}}|\xi|^pd\nu^n_{z}(\xi)d\lambda(z)<+\infty.
\end{eqnarray}
Then there exists a Young measure $\nu$ on $X$ and a subsequence which is still denoted by $(v^n)$ such that, for $h\in L^1(X)$ and for $\phi\in C_b(\mathbb{R})$,
\begin{eqnarray}\label{equ-20}
\lim_{n\rightarrow \infty} \int_Xh(z)\int_{\mathbb{R}}\phi(\xi)d\nu^n_{z}(\xi)d\lambda(z)=\int_Xh(z)\int_{\mathbb{R}}\phi(\xi)d\nu_{z}(\xi)d\lambda(z).
\end{eqnarray}
\end{thm}
\begin{cor}\label{cor-2}
(Compactness of Kinetic functions)\quad Let $(X,\lambda)$ be a finite measure space. Let $(f_n)$ be a sequence of kinetic functions on $X\times \mathbb{R}$: $f_n(z,\xi)=\nu^n_z(\xi,\infty)$,  where $\nu^n, n\ge1$, are Young measures on $X$ satisfying (\ref{equ-19}). Then there exists a kinetic function $f$ on $X\times\mathbb{R}$ such that $f_n\rightharpoonup f$ in $L^{\infty}(X\times \mathbb{R})-$ weak $*$, as $n\to\infty$.
\end{cor}

\subsection{Global well-posedness of (\ref{P-19})}
The following result was shown in \cite{D-V-1}.
\begin{thm}\label{thm-4}
(Existence, Uniqueness) Let $u_0\in L^{\infty}(\mathbb{T}^N)$. Assume Hypothesis H holds. Then there is a unique kinetic solution $u$ to equation (\ref{P-19}) with initial datum $u_0$. Besides, if $f$ is a generalized kinetic solution to (\ref{P-19}) with initial datum $I_{u_0>\xi}$, then there exist $u^+$ and $u^{-}$, representatives of $u$
 such that for all $t\in[0,T]$, $f^{\pm}(x,t,\xi)=I_{u^{\pm}(x,t)>\xi}\ a.s.$ for $a.e.\  (x,t,\xi)$.
\end{thm}

\begin{remark}
The kinetic solution $u$ is a strong solution in the probabilistic sense.
\end{remark}

\section{Freidlin-Wentzell large deviations and statement of the main result}
We start with a brief account of notions of large deviations.
Let $\{X^\varepsilon\}_{\varepsilon>0}$ be a family of random variables defined on a given probability space $(\Omega, \mathcal{F}, \mathbb{P})$ taking values in some Polish space $\mathcal{E}$.
\begin{dfn}
(Rate function) A function $I: \mathcal{E}\rightarrow [0,\infty]$ is called a rate function if $I$ is lower semicontinuous. A rate function $I$ is called a good rate function if the level set $\{x\in \mathcal{E}: I(x)\leq M\}$ is compact for each $M<\infty$.
\end{dfn}
\begin{dfn}
(Large deviation principle) The sequence $\{X^\varepsilon\}$ is said to satisfy the large deviation principle with rate function $I$ if for each Borel subset $A$ of $\mathcal{E}$
      \[
      -\inf_{x\in A^o}I(x)\leq \lim \inf_{\varepsilon\rightarrow 0}\varepsilon \log \mathbb{P}(X^\varepsilon\in A)\leq \lim \sup_{\varepsilon\rightarrow 0}\varepsilon \log \mathbb{P}(X^\varepsilon\in A)\leq -\inf_{x\in \bar{A}}I(x),
      \]
      where $A^o$ and $\bar{A}$ denote the interior and closure of $A$ in $\mathcal{E}$, respectively.
\end{dfn}

Suppose $W(t)$ is a cylindrical Wiener process on a Hilbert space $U$ defined on a filtered probability space $(\Omega, \mathcal{F},\{\mathcal{F}_t\}_{t\in [0,T]}, \mathbb{P} )$ ( that is, the paths of $W$ take values in $C([0,T];\mathcal{U})$, where $\mathcal{U}$ is another Hilbert space such that the embedding $U\subset \mathcal{U}$ is Hilbert-Schmidt).
Now we define
\begin{eqnarray*}
&\mathcal{A}:=\{\phi: \phi\ is\ a\ U\text{-}valued\ \{\mathcal{F}_t\}\text{-}predictable\ process\ such\ that \ \int^T_0 |\phi(s)|^2_Uds<\infty\ \mathbb{P}\text{-}a.s.\};\\
&S_M:=\{ h\in L^2([0,T];U): \int^T_0 |h(s)|^2_Uds\leq M\};\\
&\mathcal{A}_M:=\{\phi\in \mathcal{A}: \phi(\omega)\in S_M,\ \mathbb{P}\text{-}a.s.\}.
\end{eqnarray*}
Here and in the sequel {of} this paper, we will always refer to the weak topology on the set $S_M$.

Suppose for each $\varepsilon>0, \mathcal{G}^{\varepsilon}: C([0,T];\mathcal{U})\rightarrow \mathcal{E}$ is a measurable map and let $X^{\varepsilon}:=\mathcal{G}^{\varepsilon}(W)$. Now, we list below sufficient conditions for the large deviation principle of the sequence $X^{\varepsilon}$ as $\varepsilon\rightarrow 0$.
\begin{description}
  \item[\textbf{Condition A} ] There exists a measurable map $\mathcal{G}^0: C([0,T];\mathcal{U})\rightarrow \mathcal{E}$ such that the following conditions hold
\end{description}
\begin{description}
  \item[(a)] For every $M<\infty$, let $\{h^{\varepsilon}: \varepsilon>0\}$ $\subset \mathcal{A}_M$. If $h_{\varepsilon}$ converges to $h$ as $S_M$-valued random elements in distribution, then $\mathcal{G}^{\varepsilon}(W(\cdot)+\frac{1}{\sqrt{\varepsilon}}\int^{\cdot}_{0}h^\varepsilon(s)ds)$ converges in distribution to $\mathcal{G}^0(\int^{\cdot}_{0}h(s)ds)$.
  \item[(b)] For every $M<\infty$, the set $K_M=\{\mathcal{G}^0(\int^{\cdot}_{0}h(s)ds): h\in S_M\}$ is a compact subset of $\mathcal{E}$.
\end{description}

The following result is due to Budhiraja et al. in \cite{BD}.
\begin{thm}\label{thm-7}
If $\{\mathcal{G}^{\varepsilon}\}$ satisfies {condition A}, then $X^{\varepsilon}$ satisfies the large deviation principle on $\mathcal{E}$ with the
following good rate function $I$ defined by
\begin{eqnarray}\label{equ-27-1}
I(f)=\inf_{\{h\in L^2([0,T];U): f= \mathcal{G}^0(\int^{\cdot}_{0}h(s)ds)\}}\Big\{\frac{1}{2}\int^T_0|h(s)|^2_{U}ds\Big\},\ \ \forall f\in\mathcal{E}.
\end{eqnarray}
By convention, $I(f)=\infty$, if  $\Big\{h\in L^2([0,T];U): f= \mathcal{G}^0(\int^{\cdot}_{0}h(s)ds)\Big\}=\emptyset.$
\end{thm}

Recently, a new sufficient condition (Condition B below) to verify the assumptions in {condition A} (hence the large deviation principle) is proposed by Matoussi, Sabagh and Zhang in \cite{MSZ}. It turns out this new sufficient condition is suitable for establishing the large deviation principle for the scalar conservation laws.
\begin{description}
  \item[\textbf{Condition B} ] There exists a measurable map $\mathcal{G}^0: C([0,T];\mathcal{U})\rightarrow \mathcal{E}$ such that the following two items hold
\end{description}
\begin{description}
  \item[(i)] For every $M<+\infty$, and for any family $\{h^{\varepsilon}; \varepsilon>0\}$ $\subset \mathcal{A}_M$ and any $\delta>0$,
      \[
      \lim_{\varepsilon\rightarrow 0}\mathbb{P}\Big(\rho(Y^\varepsilon, Z^\varepsilon)>\delta\Big)=0,
      \]
     where $Y^\varepsilon:=\mathcal{G}^{\varepsilon}\left(W(\cdot)+\frac{1}{\sqrt{\varepsilon}}\int^{\cdot}_{0}h^\varepsilon(s)ds\right)$, $Z^\varepsilon:=\mathcal{G}^0\left(\int^{\cdot}_{0}h^\varepsilon(s)ds\right)$,
     and $\rho(\cdot,\cdot)$ stands for the metric in the space $\mathcal{E}$.
  \item[(ii)] For every $M<+\infty$ and any family $\{h^\varepsilon; \varepsilon>0\}\subset S_M$ that converges to some element $h$ as $\varepsilon\rightarrow 0$,
      $\mathcal{G}^0\left(\int^{\cdot}_{0}h^\varepsilon(s)ds\right)$ converges to $\mathcal{G}^0\left(\int^{\cdot}_{0}h(s)ds\right)$ in the space $\mathcal{E}$.
\end{description}

\subsection{Statement of the main result}
In this paper, we are concerned with the following stochastic conservation law driven by small multiplicative noise
\begin{eqnarray}\label{P-1}
\left\{
  \begin{array}{ll}
  du^\varepsilon+div(A(u^\varepsilon))dt=\sqrt{\varepsilon}\Phi(u^\varepsilon) dW(t),\\
u^\varepsilon(0)=u_0,
  \end{array}
\right.
\end{eqnarray}
for $\varepsilon>0$, where $u_0\in L^{\infty}(\mathbb{T}^N)$. Under Hypothesis H, by Theorem \ref{thm-4}, there exists a unique kinetic solution $u^\varepsilon\in L^1([0,T]; L^1(\mathbb{T}^N))$  a.s..
Therefore, there exists a Borel-measurable function
\[
\mathcal{G}^{\varepsilon}: C([0,T];\mathcal{U})\rightarrow L^1([0,T];L^1(\mathbb{T}^N))
\]
such that $u^{\varepsilon}(\cdot)=\mathcal{G}^{\varepsilon}(W(\cdot))$.

Let $h\in L^2([0,T];U)$, we consider the following skeleton equation
\begin{eqnarray}\label{P-2}
\left\{
  \begin{array}{ll}
    du_h+div(A(u_h))dt=\Phi(u_h) h(t)dt,\\
    u_h(0)=u_0.
  \end{array}
\right.
\end{eqnarray}

The solution $u_h$, whose existence will be proved in next section, defines a measurable mapping $\mathcal{G}^0: C([0,T];\mathcal{U})\rightarrow L^1([0,T];L^1(\mathbb{T}^N))$ so that  $\mathcal{G}^0\Big(\int^{\cdot}_0 h(s)ds\Big):=u_h(\cdot)$.

\smallskip

We are now ready to state our main result of this paper

\begin{thm}\label{thm-3}
Let $u_0\in L^{\infty}(\mathbb{T}^N)$. Assume Hypothesis H holds. Then $u^{\varepsilon}$ satisfies the large deviation principle on $L^1([0,T];L^1(\mathbb{T}^N))$ with the good rate function $I$ given by (\ref{equ-27-1}).
\end{thm}

\section{Skeleton equation}

\subsection{Existence and uniqueness of solutions to the skeleton equation}\label{s-1}
Fix $h\in S_M$, and assume $h(t)=\sum_{k\geq 1}h^k(t)e_k$, where $\{e_k\}_{k\geq 1}$ is an orthonormal basis of $U$. Now, we introduce definitions of solution to the skeleton equation (\ref{P-2}).
\begin{dfn}(Kinetic solution)
Let $u_0\in L^{\infty}(\mathbb{T}^N)$. A measurable function $u_h: \mathbb{T}^N\times [0,T]\rightarrow \mathbb{R}$ is said to be a kinetic solution to (\ref{P-2}), if for any $p\geq 1$, there exists $C_p\geq 0$ such that
\[
\underset{0\leq t\leq T}{{\rm{ess\sup}}}\ \|u_h(t)\|^p_{L^p(\mathbb{T}^N)}\leq C_p,
\]
and if there exists a measure $m_h\in \mathcal{M}^+_0(\mathbb{T}^N\times [0,T]\times \mathbb{R})$ such that $f_h:= I_{u_h>\xi}$ satisfies that for all $\varphi\in C^1_c(\mathbb{T}^N\times [0,T)\times \mathbb{R})$,
\begin{eqnarray}\notag
&&\int^T_0\langle f_h(t), \partial_t \varphi(t)\rangle dt+\langle f_0, \varphi(0)\rangle +\int^T_0\langle f_h(t), a(\xi)\cdot \nabla \varphi (t)\rangle dt\\
\label{P-3}
&=&-\sum_{k\geq 1}\int^T_0\int_{\mathbb{T}^N} g_k(x,u_h(x,t))\varphi (x,t, u_h(x,t))h^k(t)dxdt + m_h(\partial_{\xi} \varphi),
\end{eqnarray}
where $f_0(x,\xi)=I_{u_0(x)>\xi}$.
\end{dfn}

\begin{dfn}(Generalized kinetic solution)
Let $f_0:\mathbb{T}^N\times \mathbb{R}\rightarrow [0,1]$ be a kinetic function. A measurable function $f_h:\mathbb{T}^N\times [0,T]\times \mathbb{R}\rightarrow [0,1]$  is said to be a generalized kinetic solution to (\ref{P-2}) with the initial datum $f_0$, if $(f_h(t))=(f_h(t,\cdot,\cdot))$ is a kinetic function such that for all $p\geq 1$, $\nu^h:=-\partial_{\xi} f_h$ satisfies
\begin{eqnarray}
\underset{0\leq t\leq T}{{\rm{ess\sup}}}\ \int_{\mathbb{T}^N}\int_{\mathbb{R}}|\xi|^pd\nu^h_{x,t}(\xi)dx\leq C_p,
\end{eqnarray}
where $C_p$ is a positive constant and there exists a measure $m_h\in \mathcal{M}^+_0(\mathbb{T}^N\times [0,T]\times \mathbb{R})$ such that for all $\varphi\in C^1_c(\mathbb{T}^N\times [0,T)\times \mathbb{R})$,
\begin{eqnarray}\label{P-4}\notag
&&\int^T_0\langle f_h(t), \partial_t \varphi(t)\rangle dt+\langle f_0, \varphi(0)\rangle +\int^T_0\langle f_h(t), a(\xi)\cdot \nabla \varphi (t)\rangle dt\\
&& \quad=-\sum_{k\geq 1}\int^T_0\int_{\mathbb{T}^N}\int_{\mathbb{R}} g_k(x,\xi)\varphi (x,t,\xi)h^k(t)d\nu^h_{x,t}(\xi)dxdt + m_h(\partial_{\xi} \varphi).
\end{eqnarray}
\end{dfn}

\begin{thm}\label{thm-1}
(Existence)
 Let $u_0\in L^{\infty}(\mathbb{T}^N)$. Assume Hypothesis H holds, then for any $T>0$, (\ref{P-2}) has a generalized kinetic solution $f_h$ on $[0,T]$ with initial datum $f_0=I_{u_0>\xi}$.
\end{thm}
The proof of Theorem \ref{thm-1} is similar to the proof of Theorem \ref{thm-4} which was done in \cite{D-V-1}, we therefore omit it here. Moreover, as stated in Proposition \ref{prp-3}, for the generalized solution $f_h$, we have $f_h=f^+_h=f^-_h$ a.e. $t\in [0,T]$.

\

In order to prove the uniqueness of the skeleton equation (\ref{P-2}), we adopt the doubling of variables method. Similarly to (\ref{e-14}), (\ref{P-4}) can be strengthened to the following form: for all $t\in [0,T)$ and $\varphi\in C^1_c(\mathbb{T}^N\times \mathbb{R})$,
\begin{eqnarray}\notag
&&-\langle f^+_h(t),\varphi\rangle+\langle f_0, \varphi\rangle+\int^t_0\langle f_h(s), a(\xi)\cdot \nabla \varphi\rangle ds\\
\label{P-5}
&&=-\sum_{k\geq 1}\int^t_0\int_{\mathbb{T}^N}\int_{\mathbb{R}} g_k(x,\xi)\varphi (x,\xi)h^k(s)d\nu^h_{x,s}(\xi)dxds + \langle m_h,\partial_{\xi} \varphi\rangle([0,t]), a.s.,
\end{eqnarray}
where $\langle m_h,\partial_{\xi} \varphi\rangle([0,t])=\int_{\mathbb{T}^N\times[0,t]\times \mathbb{R}}\partial_{\xi}\varphi(x,\xi)dm_h(x,s,\xi)$.

\
Now, with the help of (\ref{P-5}), we prove a comparison theorem relating these two generalized kinetic solutions $f_i$, $i=1,2$ of the following equations
\begin{eqnarray}\label{P-6}
\left\{
  \begin{array}{ll}
    du^i_h+div(A(u^i_h))dt=\Phi(u^i_h) h(t)dt,\\
    u^i_h(0)=u_0.
  \end{array}
\right.
\end{eqnarray}
\begin{prp}\label{prp-1}
Under Hypothesis H, let $f_i, i=1,2$ be two generalized solutions to (\ref{P-6}). Then, for all $0< t<T$, and nonnegative test functions $\rho\in C^{\infty}(\mathbb{T}^N), \psi\in C^{\infty}_c(\mathbb{R})$, we have
\begin{eqnarray}\notag
&&\int_{(\mathbb{T}^N)^2}\int_{\mathbb{R}^2}\rho(x-y)\psi(\xi-\zeta)\Big(f^{\pm}_1(x,t,\xi)\bar{f}^{\pm}_2(y,t,\zeta)+\bar{f}^{\pm}_1(x,t,\xi){f}^{\pm}_2(y,t,\zeta)\Big)d\xi d\zeta dxdy\\ \notag
&\leq& \int_{(\mathbb{T}^N)^2}\int_{\mathbb{R}^2}\rho(x-y)\psi(\xi-\zeta)\Big(f_{1,0}(x,\xi)\bar{f}_{2,0}(y,\zeta)+\bar{f}_{1,0}(x,\xi){f}_{2,0}(y,\zeta)\Big)d\xi d\zeta dxdy\\
\label{P-7}
&&\ +K_1+\bar{K}_1+2K_2,
\end{eqnarray}
where
\begin{eqnarray*}
K_1&=&\int^t_0\int_{(\mathbb{T}^N)^2}\int_{\mathbb{R}^2} f_1(x,s,\xi)\bar{f_2}(y,s,\zeta)(a(\xi)-a(\zeta))\cdot\nabla_x \alpha d \xi d \zeta dxdyds,\\
\bar{K}_1&=&\int^t_0\int_{(\mathbb{T}^N)^2}\int_{\mathbb{R}^2} \bar{f}_1(x,s,\xi){f_2}(y,s,\zeta)(a(\xi)-a(\zeta))\cdot\nabla_x \alpha d \xi d \zeta dxdyds,
\end{eqnarray*}
and
\begin{eqnarray*}
K_2=\sum_{k\geq 1}\int^t_0\int_{(\mathbb{T}^N)^2}\rho(x-y)\int_{\mathbb{R}^2}\gamma_1(\xi,\zeta)(g_{k}(x,\xi)-g_{k}(y,\zeta))h^k(s)d \nu^1_{x,s}\otimes \nu^2_{y,s}(\xi,\zeta) dxdyds
\end{eqnarray*}
with $\gamma_1(\xi,\zeta)=\int^{\xi}_{-\infty}\psi(\xi'-\zeta)d\xi'=\int^{\xi-\zeta}_{-\infty}\psi(y)dy$.
\end{prp}
\begin{proof}
Denote $f_1(x,t,\xi)$ and $f_2(y,t,\zeta)$ be two generalized solutions to (\ref{P-6})  with the corresponding kinetic measures $m_1$ and $m_2$.
Let $\varphi_1\in C^{1}_c(\mathbb{T}^N_x\times \mathbb{R}_{\xi})$ and
$\varphi_2\in C^{1}_c(\mathbb{T}^N_y\times \mathbb{R}_{\zeta})$. By (\ref{P-5}), we have
\begin{eqnarray}\label{P-8}\notag
\langle f^{+}_1(t),\varphi_1\rangle&=&\langle f_{1,0}, \varphi_1\rangle+\int^t_0\langle f_1(s), a(\xi)\cdot \nabla_x \varphi_1(s)\rangle ds\\ \notag
&&+\sum_{k\geq 1}\int^t_0\int_{\mathbb{T}^N}\int_{\mathbb{R}} g_{k}(x,\xi)\varphi_1 (x,\xi)h^k(s)d\nu^1_{x,s}(\xi)dxds -\langle m_1,\partial_{\xi} \varphi_1\rangle([0,t]),
\end{eqnarray}
where $f_{1,0}=I_{u_0>\xi}$ and $\nu^1_{x,s}(\xi)=-\partial_{\xi}f_1(s,x,\xi)=\partial_{\xi}\bar{f}_1(s,x,\xi)$.
Similarly,
\begin{eqnarray}\label{P-9}\notag
\langle \bar{f}^{+}_2(t),\varphi_2\rangle&=&\langle \bar{f}_{2,0}, \varphi_2\rangle+\int^t_0\langle \bar{f}_2(s), a(\zeta)\cdot \nabla_y \varphi_2(s)\rangle ds\\ \notag
&&-\sum_{k\geq 1}\int^t_0\int_{\mathbb{T}^N}\int_{\mathbb{R}} g_{k}(y,\zeta)\varphi_2 (y,\zeta)h^k(s)d\nu^2_{y,s}(\zeta)dyds +\langle m_2,\partial_{\zeta} \varphi_2\rangle([0,t]).
\end{eqnarray}
where $f_{2,0}=I_{u_0>\zeta}$ and  $\nu^2_{y,s}(\zeta)=\partial_{\zeta}\bar{f}_2(s,y,\zeta)=-\partial_{\zeta}f_2(s,y,\zeta)$.

Clearly, $t\mapsto\langle  f^{+}_1(t),\varphi_1\rangle$ and $t\mapsto\langle \bar{f}^{+}_2(t),\varphi_2\rangle$ are functions of finite variation.
Denote the duality distribution over $\mathbb{T}^N_x\times \mathbb{R}_\xi\times \mathbb{T}^N_y\times \mathbb{R}_\zeta$ by $\langle\langle\cdot,\cdot\rangle\rangle$. Setting $\alpha(x,\xi,y,\zeta)=\varphi_1(x,\xi)\varphi_2(y,\zeta)$ and using the integration by parts formula (see Proposition (4.5) on P6 in \cite{RY99}), we have
\begin{eqnarray}\notag
\langle\langle f^+_1(t)\bar{f}^+_2(t), \alpha \rangle\rangle
&=& \langle\langle f_{1,0}\bar{f}_{2,0}, \alpha \rangle\rangle+\int^t_0\int_{(\mathbb{T}^N)^2}\int_{\mathbb{R}^2}f_1\bar{f}_2(a(\xi)\cdot \nabla_x+a(\zeta)\cdot \nabla_y) \alpha d\xi d\zeta dxdyds\\ \notag
\quad &&-\sum_{k\geq 1}\int^t_0\int_{(\mathbb{T}^N)^2}\int_{\mathbb{R}^2}f_1(s,x,\xi)\alpha g_{k}(y,\zeta)h^k(s)d\xi d\nu^2_{y,s}(\zeta)dxdyds\\ \notag
\quad &&+\sum_{k\geq 1}\int^t_0\int_{(\mathbb{T}^N)^2}\int_{\mathbb{R}^2}\bar{f}_2(s,y,\zeta)\alpha g_{k}(x,\xi)h^k(s)d\zeta d\nu^1_{x,s}(\xi)dxdyds\\ \notag
\quad &&+\int_{(0,t]}\int_{(\mathbb{T}^N)^2}\int_{\mathbb{R}^2}f^-_1(s,x,\xi)\partial_{\zeta} \alpha dm_2(y,\zeta,s)d\xi dx\\ \notag
\quad &&-\int_{(0,t]}\int_{(\mathbb{T}^N)^2}\int_{\mathbb{R}^2}\bar{f}^+_2(s,y,\zeta)\partial_{\xi} \alpha dm_1(x,\xi,s)d\zeta dy\\
\label{P-10}
&=:& \langle\langle f_{1,0}\bar{f}_{2,0}, \alpha \rangle\rangle+I_1+I_2+I_3+I_4+I_5.
\end{eqnarray}
Similarly, we  have
\begin{eqnarray}\notag
\langle\langle \bar{f}^+_1(t)f^+_2(t), \alpha \rangle\rangle
&=& \langle\langle \bar{f}_{1,0}f_{2,0}, \alpha \rangle\rangle+\int^t_0\int_{(\mathbb{T}^N)^2}\int_{\mathbb{R}^2}\bar{f}_1{f}_2(a(\xi)\cdot \nabla_x+a(\zeta)\cdot \nabla_y) \alpha d\xi d\zeta dxdyds\\ \notag
&& \ +\sum_{k\geq 1}\int^t_0\int_{(\mathbb{T}^N)^2}\int_{\mathbb{R}^2}\bar{f}_1(s,x,\xi)\alpha g_{k}(y,\zeta)h^k(s)d\xi d\nu^2_{y,s}(\zeta)dxdyds\\ \notag
&&\ -\sum_{k\geq 1}\int^t_0\int_{(\mathbb{T}^N)^2}\int_{\mathbb{R}^2}{f}_2(s,y,\zeta)\alpha g_{k}(x,\xi)h^k(s) d\nu^1_{x,s}(\xi)d\zeta dxdyds\\ \notag
&&\ -\int_{(0,t]}\int_{(\mathbb{T}^N)^2}\int_{\mathbb{R}^2}\bar{f}^+_1(s,x,\xi)\partial_{\zeta} \alpha dm_2(y,\zeta,s)d\xi dx\\ \notag
 &&\ +\int_{(0,t]}\int_{(\mathbb{T}^N)^2}\int_{\mathbb{R}^2}{f}^-_2(s,y,\zeta)\partial_{\xi} \alpha dm_1(x,\xi,s)d\zeta dy\\
\label{P-10-1}
&=:& \langle\langle \bar{f}_{1,0}{f}_{2,0}, \alpha \rangle\rangle+\bar{I}_1+\bar{I}_2+\bar{I}_3+\bar{I}_4+\bar{I}_5.
\end{eqnarray}

Following the idea developed by \cite{D-V-1}, we can relax the conditions imposed on $\alpha$. Specifically,
By a density argument, (\ref{P-10}) and (\ref{P-10-1}) remain true for any test function $\alpha\in C^{\infty}_c(\mathbb{T}^N_x\times \mathbb{R}_\xi\times \mathbb{T}^N_y\times \mathbb{R}_\zeta)$. The assumption that $\alpha$ is compactly supported can be relaxed thanks to (\ref{equ-37}) on $m_i$ and (\ref{equ-26}) on $\nu_i$, $i=1,2.$
Using a truncation argument of $\alpha$, it is easy to see that (\ref{P-10}) and (\ref{P-10-1}) remain true if
we take $\alpha \in C^{\infty}_b(\mathbb{T}^N_x\times \mathbb{R}_\xi\times \mathbb{T}^N_y\times \mathbb{R}_\zeta)$, which is compactly supported in a neighbourhood of the diagonal
\[
\Big\{(x,\xi,x,\xi); x\in \mathbb{T}^N, \xi\in \mathbb{R}\Big\}.
\]
Taking $\alpha=\rho(x-y)\psi(\xi-\zeta)$, then we have the following remarkable identities
\begin{eqnarray}\label{P-11}
(\nabla_x+\nabla_y)\alpha=0, \quad (\partial_{\xi}+\partial_{\zeta})\alpha=0.
\end{eqnarray}
Referring to Proposition 13 in \cite{D-V-1}, we know that $I_4, I_5, \bar{I}_4, \bar{I}_5$ in (\ref{P-10}) and (\ref{P-10-1}) are all non-positive. As a result, we have
\begin{eqnarray*}
&&\int_{(\mathbb{T}^N)^2}\int_{\mathbb{R}^2}\rho_\gamma (x-y)\psi_{\delta}(\xi-\zeta)(f^{+}_1(x,t,\xi)\bar{f}^{+}_2(y,t,\zeta)+\bar{f}^{+}_1(x,t,\xi)f^{+}_2(y,t,\zeta))d\xi d\zeta dxdy\\
&\leq & \int_{(\mathbb{T}^N)^2}\int_{\mathbb{R}^2}\rho_\gamma (x-y)\psi_{\delta}(\xi-\zeta)(f_{1,0}(x,\xi)\bar{f}_{2,0}(y,\zeta)+\bar{f}_{1,0}(x,\xi)f_{2,0}(y,\zeta))d\xi d\zeta dxdy\\
&&\ +\sum^{3}_{i=1}(I_i+\bar{I}_i).
\end{eqnarray*}
With the aid of (\ref{P-11}), we deduce that
\begin{eqnarray*}
I_1&=&\int^t_0\int_{(\mathbb{T}^N)^2}\int_{\mathbb{R}^2}f_1\bar{f}_2(a(\xi)-a(\zeta))\cdot\nabla_x\alpha d\xi d\zeta dxdyds,\\
\bar{I}_1&=&\int^t_0\int_{(\mathbb{T}^N)^2}\int_{\mathbb{R}^2}\bar{f}_1f_2(a(\xi)-a(\zeta))\cdot\nabla_x\alpha d\xi d\zeta dxdyds.
\end{eqnarray*}
Let
\[
\gamma_1(\xi,\zeta)=\int^{\xi}_{-\infty}\psi(\xi'-\zeta)d\xi'
\]
for some $\xi, \zeta\in\mathbb{R}$.  Then
\begin{eqnarray}\notag
I_2
&=&-\sum_{k\geq 1}\int^t_0\int_{(\mathbb{T}^N)^2}\int_{\mathbb{R}^2}f_1(s,x,\xi)\rho(x-y)\partial_{\xi} \gamma_1(\xi,\zeta) g_{k}(y,\zeta)h^k(s)d\xi d\nu^2_{y,s}(\zeta)dxdyds\\ \notag
&=&-\sum_{k\geq 1}\int^t_0\int_{(\mathbb{T}^N)^2}\int_{\mathbb{R}}\rho(x-y) g_{k}(y,\zeta)h^k(s)\Big(\int_{\mathbb{R}}f_1(s,x,\xi)\partial_{\xi} \gamma_1(\xi,\zeta)d\xi\Big) d\nu^2_{y,s}(\zeta)dxdyds\\
\label{P-12}
&=&-\sum_{k\geq 1}\int^t_0\int_{(\mathbb{T}^N)^2}\int_{\mathbb{R}^2}\rho(x-y)\gamma_1(\xi,\zeta) g_{k}(y,\zeta)h^k(s)d \nu^1_{x,s}\otimes \nu^2_{y,s}(\xi,\zeta)dxdyds.
\end{eqnarray}
Similarly, for $\xi, \zeta\in\mathbb{R}$, let
\[
\gamma_2(\zeta,\xi)=\int^{\infty}_{\zeta}\psi(\xi-\zeta')d\zeta',
\]
then
\begin{eqnarray}\notag
I_3
&=&-\sum_{k\geq 1}\int^t_0\int_{(\mathbb{T}^N)^2}\int_{\mathbb{R}^2}\bar{f}_2(s,y,\zeta)\rho(x-y)\partial_{\zeta}\gamma_2(\zeta, \xi)g_{k}(x,\xi)h^k(s)d\nu^1_{x,s}(\xi)d\zeta dxdyds\\ \notag
&=&-\sum_{k\geq 1}\int^t_0\int_{(\mathbb{T}^N)^2}\int_{\mathbb{R}}\rho(x-y)g_{k}(x,\xi)h^k(s)\Big(\int_{\mathbb{R}} \bar{f}_2(s,y,\zeta)\partial_{\zeta}\gamma_2(\zeta, \xi)d\zeta\Big) d\nu^1_{x,s}(\xi) dxdyds\\
\label{P-13}
&=& \sum_{k\geq 1}\int^t_0\int_{(\mathbb{T}^N)^2}\int_{\mathbb{R}^2}\gamma_2(\zeta, \xi)\rho(x-y)g_{k}(x,\xi)h^k(s) d \nu^1_{x,s}\otimes \nu^2_{y,s}(\xi,\zeta) dxdyds.
\end{eqnarray}

Note that $\gamma_1(\xi,\zeta)=\gamma_2(\zeta, \xi)=\int^{\xi-\zeta}_{-\infty}\psi(y)dy$. We deduce from (\ref{P-12}) and (\ref{P-13}) that
\begin{eqnarray*}
I_2+I_3=\sum_{k\geq 1}\int^t_0\int_{(\mathbb{T}^N)^2}\rho(x-y)\int_{\mathbb{R}^2}\gamma_1(\xi,\zeta)(g_{k}(x,\xi)-g_{k}(y,\zeta))h^k(s)d \nu^1_{x,s}\otimes \nu^2_{y,s}(\xi,\zeta) dxdyds.
\end{eqnarray*}
Similarly, we have
\begin{eqnarray*}
\bar{I}_2+\bar{I}_3=\sum_{k\geq 1}\int^t_0\int_{(\mathbb{T}^N)^2}\rho(x-y)\int_{\mathbb{R}^2}\gamma_1(\xi,\zeta)(g_{k}(x,\xi)-g_{k}(y,\zeta))h^k(s)d \nu^1_{x,s}\otimes \nu^2_{y,s}(\xi,\zeta) dxdyds.
\end{eqnarray*}
Taking $K_1=I_1, \bar{K}_1=\bar{I}_1$ and $K_2=I_2+I_3=\bar{I}_2+\bar{I}_3$,
the equation (\ref{P-7}) is established for $f^+_i$. To obtain the result for  $f^-_i$, we take $t_n\uparrow t$, write (\ref{P-7}) for $f^+_i(t_n)$ and let $n\rightarrow \infty$.

\end{proof}

Now, we are in a position to establish the uniqueness.
\begin{thm}\label{thm-2}
 Let $u_0\in L^{\infty}(\mathbb{T}^N)$ and assume Hypothesis H holds. Then there exists at most one kinetic solution to (\ref{P-2}) with the initial datum $u_0$. Besides, any generalized solution $f_h$ is actually a kinetic solution, i.e. if $f_h$ is a generalized solution to (\ref{P-2}) with initial datum $I_{u_0>\xi}$, then there exists a kinetic solution $u_h$ to (\ref{P-2}) with initial datum $u_0$ such that $f_h(x,t,\xi)=I_{u_h(x,t)>\xi}$, for a.e. $(x,t,\xi)$.
\end{thm}
\begin{proof}
Let $\rho_{\gamma}, \psi_{\delta}$ be approximations to the identity on $\mathbb{T}^N$ and $\mathbb{R}$, respectively. That is, let $\rho\in C^{\infty}(\mathbb{T}^N)$, $\psi\in C^{\infty}_c(\mathbb{R})$ be symmetric nonnegative functions such as $\int_{\mathbb{T}^N}\rho =1$, $\int_{\mathbb{R}}\psi =1$ and supp$\psi \subset (-1,1)$. We define
\[
\rho_{\gamma}(x)=\frac{1}{\gamma^N}\rho\Big(\frac{x}{\gamma}\Big), \quad \psi_{\delta}(\xi)=\frac{1}{\delta}\psi\Big(\frac{\xi}{\delta}\Big).
\]
Letting $\rho:=\rho_{\gamma}(x-y)$ and $\psi:=\psi_{\delta}(\xi-\zeta)$ in Proposition \ref{prp-1}, we get from (\ref{P-7}) that
\begin{eqnarray}\notag
 &&\int_{(\mathbb{T}^N)^2}\int_{\mathbb{R}^2}\rho_\gamma (x-y)\psi_{\delta}(\xi-\zeta)(f^{\pm}_1(x,t,\xi)\bar{f}^{\pm}_2(y,t,\zeta)+\bar{f}^{\pm}_1(x,t,\xi)f^{\pm}_2(y,t,\zeta))d\xi d\zeta dxdy\\
\notag
&\leq& \int_{(\mathbb{T}^N)^2}\int_{\mathbb{R}^2}\rho_\gamma (x-y)\psi_{\delta}(\xi-\zeta)(f_{1,0}(x,\xi)\bar{f}_{2,0}(y,\zeta)+\bar{f}_{1,0}(x,\xi)f_{2,0}(y,\zeta))d\xi d\zeta dxdy\\
\label{equ-39}
&&\ +\tilde{K}_1+\tilde{\bar{K}}_1+2\tilde{K}_2,
\end{eqnarray}
where $\tilde{K}_1$, $\tilde{\bar{K}}_1$, $\tilde{K}_2$ in (\ref{equ-39}) are the corresponding $K_1$, $\bar{K}_1, K_2$ in the statement of Proposition \ref{prp-1} with $\rho$, $\psi$ replaced by $\rho_{\gamma}$, $\psi_{\delta}$, respectively. Let $\tilde{\gamma}_1(\xi,\zeta)=\int^{\xi}_{-\infty}\psi_{\delta}(\xi'-\zeta)d\xi'$, for simplicity, we  denote $\tilde{\gamma}_1(\xi,\zeta)=\gamma_1(\xi,\zeta)$.
In the following, we devote to making estimates of $\tilde{K}_1, \tilde{\bar{K}}_1$ and  $\tilde{K}_2$.

For $\tilde{K}_1$ and $\tilde{\bar{K}}_1$, by (\ref{qeq-22}) and using the same method as in the proof of Theorem 15 in \cite{D-V-1}, we have
\begin{eqnarray}\label{P-15}
|\tilde{K}_1|+|\tilde{\bar{K}}_1|\leq 2TC_p \delta \gamma^{-1}.
\end{eqnarray}
For $\tilde{K}_2$, by utilizing  (\ref{equ-29}), we deduce that
\begin{eqnarray*}
\tilde{K}_2
&\leq& \int^t_0\int_{(\mathbb{T}^N)^2}\rho_{\gamma}(x-y)\int_{\mathbb{R}^2}{\gamma}_1(\xi,\zeta)\sum_{k\geq 1}|g_{k}(x,\xi)-g_{k}(y,\zeta)||h^k(s)|d \nu^1_{x,s}\otimes \nu^2_{y,s}(\xi,\zeta) dxdyds\\
&\leq& \int^t_0\int_{(\mathbb{T}^N)^2}\rho_{\gamma}(x-y)
\int_{\mathbb{R}^2}{\gamma}_1(\xi,\zeta)\Big(\sum_{k\geq 1}|g_{k}(x,\xi)-g_{k}(y,\zeta)|^2\Big)^{\frac{1}{2}}
\Big(\sum_{k\geq 1}|h^k(s)|^2\Big)^{\frac{1}{2}}
d \nu^1_{x,s}\otimes \nu^2_{y,s}(\xi,\zeta) dxdyds\\
&\leq&\sqrt{D_1}\int^t_0|h(s)|_U\int_{(\mathbb{T}^N)^2}\rho_{\gamma}(x-y)|x-y|\int_{\mathbb{R}^2}{\gamma}_1(\xi,\zeta)d \nu^1_{x,s}\otimes \nu^2_{y,s}(\xi,\zeta) dxdyds\\
&&+\sqrt{D_1}\int^t_0|h(s)|_U\int_{(\mathbb{T}^N)^2}\rho_{\gamma}(x-y)\int_{\mathbb{R}^2}{\gamma}_1(\xi,\zeta)|\xi-\zeta|d \nu^1_{x,s}\otimes \nu^2_{y,s}(\xi,\zeta) dxdyds\\
&=:&\tilde{K}_{2,1}+\tilde{K}_{2,2},
\end{eqnarray*}
Note that
\begin{eqnarray*}
\int_{\mathbb{R}^2}{\gamma}_1(\xi,\zeta)d \nu^1_{x,s}\otimes \nu^2_{y,s}(\xi,\zeta)&\leq& 1,
\\
\int_{(\mathbb{T}^N)^2}\rho_{\gamma}(x-y)|x-y|dxdy&\leq&\gamma,
\end{eqnarray*}
it follows that
\begin{eqnarray}\label{P-14}
\tilde{K}_{2,1}\leq\sqrt{D_1} \gamma\int^t_0|h(s)|_Uds
\leq\sqrt{D_1}\gamma(T+M).
\end{eqnarray}
Moreover, by $\nu^1_{x,s}(\xi)=-\partial_{\xi}f_1(s,x,\xi)=\partial_{\xi}\bar{f}_1(s,x,\xi)$ and $\nu^2_{y,s}(\zeta)=\partial_{\zeta}\bar{f}_2(s,y,\zeta)=-\partial_{\zeta}f_2(s,y,\zeta)$, it follows that
\begin{eqnarray*} \notag
\tilde{K}_{2,2}&\leq& \sqrt{D_1}\int^t_0|h(s)|_U\int_{(\mathbb{T}^N)^2}\int_{\mathbb{R}^2}\rho_{\gamma}(x-y)|\xi-\zeta|d \nu^1_{x,s}\otimes \nu^2_{y,s}(\xi,\zeta) dxdyds\\ \notag
&=& \sqrt{D_1}\int^t_0|h(s)|_U\int_{(\mathbb{T}^N)^2}\int_{\mathbb{R}^2}\rho_{\gamma}(x-y)(\xi-\zeta)^{+}d \nu^1_{x,s}\otimes \nu^2_{y,s}(\xi,\zeta) dxdyds\\ \notag
&& +\sqrt{D_1}\int^t_0|h(s)|_U\int_{(\mathbb{T}^N)^2}\int_{\mathbb{R}^2}\rho_{\gamma}(x-y)(\xi-\zeta)^{-}d \nu^1_{x,s}\otimes \nu^2_{y,s}(\xi,\zeta) dxdyds\\
&=& \sqrt{D_1}\int^t_0|h(s)|_U\int_{(\mathbb{T}^N)^2}\int_{\mathbb{R}}\rho_{\gamma}(x-y)(f^{\pm}_1(x,s,\xi)\bar{f}^{\pm}_2(y,s,\xi)+\bar{f}^{\pm}_1(x,s,\xi){f}^{\pm}_2(y,s,\xi))d\xi dxdyds,
\end{eqnarray*}
where we have used $\delta_{\xi=\zeta}=-\partial_{\xi}\partial_{\zeta}(\xi-\zeta)^{+}=-\partial_{\xi}\partial_{\zeta}(\xi-\zeta)^{-}$.

By utilizing $\int_{\mathbb{R}}\psi_{\delta}(\xi-\zeta)d\zeta=1$ and $\int^{\delta}_0\psi_{\delta}(\zeta')d\zeta'=\int^{0}_{-\delta}\psi_{\delta}(\zeta')d\zeta'=\frac{1}{2}$, we get
\begin{eqnarray}\notag
&&\Big|\int_{(\mathbb{T}^N)^2}\int_{\mathbb{R}}\rho_{\gamma}(x-y)f^{\pm}_1(x,s,\xi)\bar{f}^{\pm}_2(y,s,\xi)d\xi dxdy\\ \notag
&&\ -\int_{(\mathbb{T}^N)^2}\int_{\mathbb{R}^2}f^{\pm}_1(x,s,\xi)\bar{f}^{\pm}_2(y,s,\zeta)\rho_{\gamma}(x-y)\psi_{\delta}(\xi-\zeta)dxdyd\xi d\zeta\Big|\\ \notag
&=&\Big|\int_{(\mathbb{T}^N)^2}\rho_{\gamma}(x-y)\int_{\mathbb{R}}f^{\pm}_1(x,s,\xi)\int_{\mathbb{R}}\psi_{\delta}(\xi-\zeta)(\bar{f}^{\pm}_2(y,s,\xi)-\bar{f}^{\pm}_2(y,s,\zeta))d\zeta d\xi dxdy\Big|\\ \notag
&\leq&\int_{(\mathbb{T}^N)^2}\rho_{\gamma}(x-y)\int_{\mathbb{R}}f^{\pm}_1(x,s,\xi)\int^{\xi}_{\xi-\delta}\psi_{\delta}(\xi-\zeta)(\bar{f}^{\pm}_2(y,s,\xi)-\bar{f}^{\pm}_2(y,s,\zeta)) d\zeta d\xi dxdy\\ \notag
&&\ +\int_{(\mathbb{T}^N)^2}\int_{\mathbb{R}}\rho_{\gamma}(x-y)f^{\pm}_1(x,s,\xi)\int^{\xi+\delta}_{\xi}\psi_{\delta}(\xi-\zeta)(\bar{f}^{\pm}_2(y,s,\zeta)-\bar{f}^{\pm}_2(y,s,\xi)) d\zeta d\xi dxdy\\ \notag
&\leq&\int_{(\mathbb{T}^N)^2}\rho_{\gamma}(x-y)\int^{\delta}_0\psi_{\delta}(\zeta')\int_{\mathbb{R}}f^{\pm}_1(x,s,\xi)(\bar{f}^{\pm}_2(y,s,\xi)-\bar{f}^{\pm}_2(y,s,\xi-\zeta')) d\xi d\zeta' dxdy\\ \notag
&&\ +\int_{(\mathbb{T}^N)^2}\rho_{\gamma}(x-y)\int^{0}_{-\delta}\psi_{\delta}(\zeta')\int_{\mathbb{R}}f^{\pm}_1(x,s,\xi)(\bar{f}^{\pm}_2(y,s,\xi-\zeta')-\bar{f}^{\pm}_2(y,s,\xi)) d\xi d\zeta' dxdy\\ \notag
&\leq&\delta\int_{(\mathbb{T}^N)^2}\rho_{\gamma}(x-y)\Big(\int^{\delta}_0\psi_{\delta}(\zeta')d\zeta'\Big)\Big(\int_{\mathbb{R}}\nu^{2,\pm}_{y,s}(d\xi) \Big) dxdy\\ \notag
&&\ +\delta\int_{(\mathbb{T}^N)^2}\rho_{\gamma}(x-y)\Big(\int^{0}_{-\delta}\psi_{\delta}(\zeta')d\zeta'\Big)\Big(\int_{\mathbb{R}}\nu^{2,\pm}_{y,s}(d\xi) \Big) dxdy\\
\label{qeq-14}
&\leq & \frac{1}{2}\delta+\frac{1}{2}\delta=\delta,
\end{eqnarray}
where we have taken into account the facts that  $\bar{f}^{\pm}_2(y,s,\xi)$ is increasing in $\xi$, $f^{\pm}_1(x,s,\xi)\leq 1$ and $\int_{\mathbb{R}}\nu^{2,\pm}_{y,s}(d\xi)=1$.
Similarly,
\begin{eqnarray}\notag
&&\Big|\int_{(\mathbb{T}^N)^2}\int_{\mathbb{R}}\rho_{\gamma}(x-y)\bar{f}^{\pm}_1(x,s,\xi){f}^{\pm}_2(y,s,\xi)d\xi dxdy\\ \notag
&&\ -\int_{(\mathbb{T}^N)^2}\int_{\mathbb{R}^2}\bar{f}^{\pm}_1(x,s,\xi){f}^{\pm}_2(y,s,\zeta)\rho_{\gamma}(x-y)\psi_{\delta}(\xi-\zeta)dxdyd\xi d\zeta\Big|\\ \notag
&=&\Big|\int_{(\mathbb{T}^N)^2}\rho_{\gamma}(x-y)\int_{\mathbb{R}}\bar{f}^{\pm}_1(x,s,\xi)\int_{\mathbb{R}}\psi_{\delta}(\xi-\zeta)({f}^{\pm}_2(y,s,\xi)-{f}^{\pm}_2(y,s,\zeta))d\zeta d\xi dxdy\Big|\\ \notag
&\leq&\int_{(\mathbb{T}^N)^2}\rho_{\gamma}(x-y)\int_{\mathbb{R}}\bar{f}^{\pm}_1(x,s,\xi)\int^{\xi}_{\xi-\delta}\psi_{\delta}(\xi-\zeta)({f}^{\pm}_2(y,s,\zeta)-{f}^{\pm}_2(y,s,\xi)) d\zeta d\xi dxdy\\ \notag
&&\ +\int_{(\mathbb{T}^N)^2}\int_{\mathbb{R}}\rho_{\gamma}(x-y)\bar{f}^{\pm}_1(x,s,\xi)\int^{\xi+\delta}_{\xi}\psi_{\delta}(\xi-\zeta)({f}^{\pm}_2(y,s,\xi)-{f}^{\pm}_2(y,s,\zeta)) d\zeta d\xi dxdy\\ \notag
&\leq&\int_{(\mathbb{T}^N)^2}\rho_{\gamma}(x-y)\int^{\delta}_0\psi_{\delta}(\zeta')\int_{\mathbb{R}}\bar{f}^{\pm}_1(x,s,\xi)({f}^{\pm}_2(y,s,\xi-\zeta')-{f}^{\pm}_2(y,s,\xi)) d\xi d\zeta' dxdy\\ \notag
&&\ +\int_{(\mathbb{T}^N)^2}\rho_{\gamma}(x-y)\int^{0}_{-\delta}\psi_{\delta}(\zeta')\int_{\mathbb{R}}\bar{f}^{\pm}_1(x,s,\xi)({f}^{\pm}_2(y,s,\xi)-{f}^{\pm}_2(y,s,\xi-\zeta')) d\xi d\zeta' dxdy\\
\label{qeq-14-1}
&\leq & \frac{1}{2}\delta+\frac{1}{2}\delta=\delta.
\end{eqnarray}
Then, we deduce from (\ref{qeq-14}) and (\ref{qeq-14-1}) that
\begin{eqnarray}\notag
\tilde{K}_{2,2}&\leq& 2\delta\sqrt{D_1}\int^t_0|h(s)|_Uds\\ \notag
&& +\sqrt{D_1}\int^t_0|h(s)|_U\int_{(\mathbb{T}^N)^2}\int_{\mathbb{R}^2}(f^{\pm}_1\bar{f}^{\pm}_2+\bar{f}^{\pm}_1{f}^{\pm}_2)\rho_{\gamma}(x-y)\psi_{\delta}(\xi-\zeta)dxdyd\xi d\zeta ds\\
\notag
&\leq& 2\delta\sqrt{D_1}(T+M)\\
\label{qeq-16}
&&+\sqrt{D_1}\int^t_0|h(s)|_U\int_{(\mathbb{T}^N)^2}\int_{\mathbb{R}^2}(f^{\pm}_1\bar{f}^{\pm}_2+\bar{f}^{\pm}_1{f}^{\pm}_2)\rho_{\gamma}(x-y)\psi_{\delta}(\xi-\zeta)dxdyd\xi d\zeta ds.
\end{eqnarray}
Hence, combining (\ref{P-14}) and (\ref{qeq-16}), we get
\begin{eqnarray}\notag
\tilde{K}_2&\leq& \sqrt{D_1}(\gamma+2\delta)(T+M)\\
\label{qeq-17}
&&\ +\sqrt{D_1}\int^t_0|h(s)|_U\int_{(\mathbb{T}^N)^2}\int_{\mathbb{R}^2}(f^{\pm}_1\bar{f}^{\pm}_2+\bar{f}^{\pm}_1{f}^{\pm}_2)\rho_{\gamma}(x-y)\psi_{\delta}(\xi-\zeta)dxdyd\xi d\zeta ds.
\end{eqnarray}
Taking into account (\ref{equ-39}), we deduce that
\begin{eqnarray*}\notag
 &&\int_{(\mathbb{T}^N)^2}\int_{\mathbb{R}^2}\rho_\gamma (x-y)\psi_{\delta}(\xi-\zeta)(f^{\pm}_1(x,t,\xi)\bar{f}^{\pm}_2(y,t,\zeta)+\bar{f}^{\pm}_1(x,t,\xi){f}^{\pm}_2)(y,t,\zeta)d\xi d\zeta dxdy\\ \notag
&\leq& \int_{(\mathbb{T}^N)^2}\int_{\mathbb{R}^2}\rho_\gamma (x-y)\psi_{\delta}(\xi-\zeta)(f_{1,0}(x,\xi)\bar{f}_{2,0}(y,\zeta)+\bar{f}_{1,0}(x,\xi){f}_{2,0}(y,\zeta))d\xi d\zeta dxdy\\
&& +2TC_p \delta \gamma^{-1}+2\sqrt{D_1}(\gamma+2\delta)(T+M)\\
&& +2\sqrt{D_1}\int^t_0|h(s)|_U\int_{(\mathbb{T}^N)^2}\int_{\mathbb{R}^2}(f^{\pm}_1\bar{f}^{\pm}_2+\bar{f}^{\pm}_1{f}^{\pm}_2)\rho_{\gamma}(x-y)\psi_{\delta}(\xi-\zeta)dxdyd\xi d\zeta ds\\
&\leq& \int_{\mathbb{T}^N}\int_{\mathbb{R}}(f_{1,0}\bar{f}_{2,0}+\bar{f}_{1,0}{f}_{2,0})d\xi dx+\mathcal{E}_0(\gamma,\delta)+2TC_p \delta \gamma^{-1}+2\sqrt{D_1}(\gamma+\delta)(T+M)\\
&& +2\sqrt{D_1}\int^t_0|h(s)|_U\int_{(\mathbb{T}^N)^2}\int_{\mathbb{R}^2}(f^{\pm}_1\bar{f}^{\pm}_2+\bar{f}^{\pm}_1{f}^{\pm}_2)\rho_{\gamma}(x-y)\psi_{\delta}(\xi-\zeta)dxdyd\xi d\zeta ds,
\end{eqnarray*}
where $\mathcal{E}_0(\gamma,\delta)\rightarrow 0$, as $\gamma,\delta\rightarrow 0$.

Utilizing Gronwall inequality, we obtain
\begin{eqnarray}\notag
 &&\int_{(\mathbb{T}^N)^2}\int_{\mathbb{R}^2}\rho_\gamma (x-y)\psi_{\delta}(\xi-\zeta)(f^{\pm}_1(x,t,\xi)\bar{f}^{\pm}_2(y,t,\zeta)+\bar{f}^{\pm}_1(x,t,\xi){f}^{\pm}_2(y,t,\zeta))d\xi d\zeta dxdy\\ \notag
&\leq& \Big[\int_{\mathbb{T}^N}\int_{\mathbb{R}}(f_{1,0}\bar{f}_{2,0}+\bar{f}_{1,0}{f}_{2,0})d\xi dx+\mathcal{E}_0(\gamma,\delta)+2TC_p \delta \gamma^{-1}+2\sqrt{D_1}(\gamma+2\delta)(T+M)\Big]\\ \notag
&&\ \times\exp\Big\{2\sqrt{D_1}\int^t_0|h(s)|_Uds\Big\}\\
\notag
&\leq& e^{2\sqrt{D_1}(T+M)}\Big[\int_{\mathbb{T}^N}\int_{\mathbb{R}}(f_{1,0}\bar{f}_{2,0}+\bar{f}_{1,0}{f}_{2,0})d\xi dx+\mathcal{E}_0(\gamma,\delta)\Big]\\
\label{qeq-18}
&&\ +2e^{2\sqrt{D_1}(T+M)}[TC_p \delta \gamma^{-1}+\sqrt{D_1}(\gamma+2\delta)(T+M)].
\end{eqnarray}

%

Combining all the above estimates, it follows that
\begin{eqnarray}\notag
&&\int_{\mathbb{T}^N}\int_{\mathbb{R}}(f^{\pm}_1(x,t,\xi)\bar{f}^{\pm}_2(x,t,\xi)+\bar{f}^{\pm}_1(x,t,\xi){f}^{\pm}_2(x,t,\xi))dxd\xi\\ \notag
&=& \int_{(\mathbb{T}^N)^2}\int_{\mathbb{R}^2}\rho_\gamma (x-y)\psi_{\delta}(\xi-\zeta)(f^{\pm}_1(x,t,\xi)\bar{f}^{\pm}_2(y,t,\zeta)+\bar{f}^{\pm}_1(x,t,\xi){f}^{\pm}_2(y,t,\zeta))d\xi d\zeta dxdy+\mathcal{E}_t(\gamma,\delta)\\ \notag
&\leq& e^{2\sqrt{D_1}(T+M)}[\int_{\mathbb{T}^N}\int_{\mathbb{R}}(f_{1,0}\bar{f}_{2,0}+\bar{f}_{1,0}{f}_{2,0})d\xi dx+2\mathcal{E}_0(\gamma,\delta)]\\
\label{equ-39-1}
&&\ +2e^{2\sqrt{D_1}(T+M)}[TC_p \delta \gamma^{-1}+\sqrt{D_1}(\gamma+2\delta)(T+M)]+\mathcal{E}_t(\gamma,\delta),
\end{eqnarray}
where $\mathcal{E}_t(\gamma,\delta)\rightarrow 0$, as $\gamma,\delta\rightarrow 0$.

Taking $\delta=\gamma^{\frac{4}{3}}$ and letting $\gamma\rightarrow 0$ gives
\begin{eqnarray}\notag
&&\int_{\mathbb{T}^N}\int_{\mathbb{R}}(f^{\pm}_1(x,t,\xi)\bar{f}^{\pm}_2(x,t,\xi)+\bar{f}^{\pm}_1(x,t,\xi){f}^{\pm}_2(x,t,\xi))dxd\xi\\
\label{P-16}
&\leq& e^{2\sqrt{D_1}(T+M)}\int_{\mathbb{T}^N}\int_{\mathbb{R}}(f_{1,0}\bar{f}_{2,0}+\bar{f}_{1,0}{f}_{2,0})dxd\xi.
\end{eqnarray}
The reduction of generalized solutions to kinetic solutions is very similar to the proof of Theorem 15 in \cite{D-V-1}, we therefore omit it here.
Suppose that $u^1_h$ and $u^2_h$ are two kinetic solutions to (\ref{P-6}), using the following identities
\begin{eqnarray}\label{equ-1}
\int_{\mathbb{R}}I_{u^1_h>\xi}\overline{I_{u^2_h>\xi}}d\xi=(u^1_h-u^2_h)^+,
\quad
\int_{\mathbb{R}}\overline{I_{u^1_h>\xi}}I_{u^2_h>\xi}d\xi=(u^1_h-u^2_h)^-,
\end{eqnarray}
we deduce from (\ref{P-16}) with $f_i=I_{u^i_h>\xi}, f_{i,0}=I_{u_0>\xi}$ that
\begin{eqnarray*}\label{P-17}
\|u^1_h(t)-u^2_h(t)\|_{L^1(\mathbb{T}^N)}\leq e^{\sqrt{D_1}(T+M)}\|u_{0}-u_{0}\|_{L^1(\mathbb{T}^N)}=0.
\end{eqnarray*}
This gives the uniqueness.
\end{proof}

In view of Theorem \ref{thm-1} and Theorem \ref{thm-2}, we can define $\mathcal{G}^0: C([0,T];\mathcal{U})\rightarrow L^1([0,T];L^1(\mathbb{T}^N))$ by
\begin{eqnarray}
\mathcal{G}^0(\check{h}):=\left\{
                   \begin{array}{ll}
                      u_{h}, & {\rm{if}}\ \check{h}= \int^{\cdot}_0 h(s)ds, \ {\rm{for\ some}}\ h\in L^2([0,T];U),\\
                    0, & {\rm{otherwise}},
                   \end{array}
                  \right.
\end{eqnarray}
where $u_h$ is the solution of equation (\ref{P-2}).

\subsection{The continuity of the skeleton equations}
{In this part, we aim to prove the continuity of the mapping $\mathcal{G}^0$. Namely, let $u_{h^{\varepsilon}}$ denote the kinetic solution  of (\ref{P-2}) with $h$ replaced by $h^\varepsilon$ and we will show that  $u_{h^{\varepsilon}}$ converges to the kinetic solution $u_h$ of the skeleton equation (\ref{P-2}) in $L^1([0,T];L^1(\mathbb{T}^N))$, if $h^\varepsilon\rightarrow h$ weakly in $L^2([0,T];U)$. For technical reasons, we will introduce two auxiliary approximation processes.

}

For any family $\{h^\varepsilon,\varepsilon>0\}\subset S_M$ and $\eta>0$, let us consider the following parabolic approximation
\begin{eqnarray}\label{eqq-1}
\left\{
  \begin{array}{ll}
    du^{\eta}_{h^\varepsilon}-\eta\Delta u^{\eta}_{h^\varepsilon}dt+div(A(u^{\eta}_{h^\varepsilon}))dt=\Phi(u^{\eta}_{h^\varepsilon}) h^\varepsilon(t)dt,\\
    u^{\eta}_{h^\varepsilon}(0)=u_0.
  \end{array}
\right.
\end{eqnarray}

It is shown by Theorem 2.1 in \cite{GR00} that equation (\ref{eqq-1}) has a unique $L^{\varrho}(\mathbb{T}^N)-$valued solution provided $\varrho$ is large enough and $u_0\in L^{\varrho}(\mathbb{T}^N)$, hence in particular for $u_0\in L^{\infty}(\mathbb{T}^N)$. We denote by $u^{\eta}_{h^\varepsilon}$ the solution of (\ref{eqq-1}).

Furthermore, for any $R\in \mathbb{N}$, we approximate operator $A$ in (\ref{eqq-1}) by Lipschitz continuous operator $A^{R}$ using the method of truncation. Consider the following approximation
\begin{eqnarray}\label{eqq-12}
\left\{
  \begin{array}{ll}
    du^{\eta,R}_{h^\varepsilon}-\eta\Delta u^{\eta,R}_{h^\varepsilon}dt+div(A^{R}(u^{\eta,R}_{h^\varepsilon}))dt=\Phi(u^{\eta,R}_{h^\varepsilon}) h^\varepsilon(t)dt,\\
    u^{\eta,R}_{h^\varepsilon}(0)=u_0,
  \end{array}
\right.
\end{eqnarray}
where $A^{R}$ is Lipschitz continuous hence it has linear growth $|A^{R}(\xi)|\leq C(R)(1+|\xi|)$.

Referring to Proposition 5.1 in \cite{DHV}, we have
\begin{eqnarray}\label{eqq-3}
\sup_{\varepsilon}\Big\{\sup_{t\in[0,T]}\|u^{\eta,R}_{h^\varepsilon}\|^2_H +\int^T_0\|\nabla u^{\eta,R}_{h^\varepsilon}(s)\|^2_Hds\Big\}\leq C(M, \|u_0\|_H),
\end{eqnarray}
where the constant $C$ is independent of $\varepsilon$ and $R$.

Following the same arguments as the proof of Theorem 5.2 in \cite{DHV}, for every $\eta>0$, it can be shown  that
\begin{eqnarray}\label{eqq-13}
\lim_{R\rightarrow +\infty}\sup_{\varepsilon>0}\int^T_0\|u^{\eta,R}_{h^\varepsilon}(t)-u^{\eta}_{h^\varepsilon}(t)\|^2_{H}dt=0.
\end{eqnarray}

{
With the above two approximation processes (\ref{eqq-1}) and (\ref{eqq-12}), for any $\varepsilon, \eta, R>0$, we have
\begin{eqnarray*}
&&\|u_{h^{\varepsilon}}-u_h\|_{L^1([0,T];L^1(\mathbb{T}^N))}\\ \notag
&\leq& \|u^{\eta}_{h^{\varepsilon}}-u_{h^{\varepsilon}}\|_{L^1([0,T];L^1(\mathbb{T}^N))}
+\|u^{\eta}_{h^{\varepsilon}}-u^{\eta,R}_{h^{\varepsilon}}\|_{L^1([0,T];L^1(\mathbb{T}^N))}
+\|u^{\eta,R}_{h^{\varepsilon}}-u^{\eta,R}_{h}\|_{L^1([0,T];L^1(\mathbb{T}^N))}\\
&&\
+\|u^{\eta,R}_{h}-u^{\eta}_{h}\|_{L^1([0,T];L^1(\mathbb{T}^N))}
+\|u^{\eta}_h-u_h\|_{L^1([0,T];L^1(\mathbb{T}^N))}.
\end{eqnarray*}
In order to establish the continuity of the skeleton equations, we need several steps.
}

Firstly, we prove the compactness of $\{u^{\eta,R}_{h^\varepsilon}, \varepsilon> 0\}$. {For simplicity,} { we set $u^{\eta,R}_{\varepsilon}:=u^{\eta,R}_{h^\varepsilon}$.}

As in \cite{FG95}, we introduce the following space. Let $K$ be a separable Banach space with the norm $\|\cdot\|_K$.
Given $p>1, \alpha\in (0,1)$, let $W^{\alpha,p}([0,T]; K)$ be the Sobolev space of all functions $u\in L^p([0,T];K)$ such that
\[
\int^T_0\int^T_0\frac{\|u(t)-u(s)\|_K^{ p}}{|t-s|^{1+\alpha p}}dtds< \infty,
\]
which is then endowed with the norm
\[
\|u\|^p_{W^{\alpha,p}([0,T]; K)}=\int^T_0\|u(t)\|_K^pdt+\int^T_0\int^T_0\frac{\|u(t)-u(s)\|_K^{ p}}{|t-s|^{1+\alpha p}}dtds.
\]
The following result can be found in \cite{FG95}.
\begin{lemma}\label{lem-1-1}
Let $B_0\subset B\subset B_1$ be three Banach spaces. Assume that both $B_0$ and $B_1$ are reflexive, and $B_0$ is compactly embedded in $B$. Let $p\in (1, \infty)$ and $\alpha \in (0, 1)$ be given. Let $\Lambda$ be the space
\[
\Lambda:= L^p([0, T]; B_0)\cap W^{\alpha, p}([0,T]; B_1)
\]
endowed with the natural norm. Then the embedding of $\Lambda$ in $L^p([0,T];B)$ is compact.

\end{lemma}

We then have the following result.

\begin{prp}\label{prpp-1}
For any $\eta, R>0$, $\{u^{\eta,R}_{\varepsilon}, \varepsilon> 0\}$ is compact in $L^2([0,T];H)$.
\end{prp}
\begin{proof}
From (\ref{eqq-12}), $ u^{\eta,R}_{\varepsilon}$ can be written as
\begin{eqnarray*}
 u^{\eta,R}_{\varepsilon}(t)&=&u_0+\eta\int^t_0 \Delta u^{\eta,R}_{\varepsilon} ds -\int^t_0div(A^{R}(u^{\eta,R}_{\varepsilon}(s)))ds+\int^t_0\Phi(u^{\eta,R}_{\varepsilon}) h^{\varepsilon}(s)ds\\
 &=:& I^\varepsilon_1+I^\varepsilon_2+I^\varepsilon_3+I^\varepsilon_4.
\end{eqnarray*}
Clearly, $\|I^\varepsilon_1\|_{H}\leq C_1$. Next,
\begin{eqnarray*}
\| -\Delta u^{\eta,R}_{\varepsilon} \|_{H^{-1}}
&=& \sup_{\|v\|_{H^1}\leq 1}|\langle v,-\Delta u^{\eta,R}_{\varepsilon}\rangle|\\
&=& \sup_{\|v\|_{H^1}\leq 1}|\langle \nabla v, \nabla u^{\eta,R}_{\varepsilon}\rangle|\\
&\leq& C\|\nabla u^{\eta,R}_{\varepsilon}\|_H
\end{eqnarray*}
which then yields the following
\begin{eqnarray*}
\|I^\varepsilon_2(t)-I^\varepsilon_2(s)\|^2_{H^{-1}}
&=&\eta\|\int^t_s-\Delta u^{\eta,R}_{\varepsilon}(l)dl \|^2_{H^{-1}}\\
&\leq& C(t-s)\int^t_s\|-\Delta u^{\eta,R}_{\varepsilon}(l)\|^2_{H^{-1}}dl\\
&\leq& C(t-s)\int^t_s\|\nabla u^{\eta,R}_{\varepsilon}(l)\|^2_Hdl.
\end{eqnarray*}
Hence, by (\ref{eqq-3}), we have for $\alpha\in (0,\frac{1}{2})$,
\begin{eqnarray*}\notag
&&\sup_\varepsilon\|I^\varepsilon_2\|^2_{W^{\alpha,2}([0,T];H^{-1}(\mathbb{T}^N))}\\ \notag
&\leq&\int^T_0 \|I^\varepsilon_2(t)\|^2_{H^{-1}}dt+\int^T_0\int^T_0\frac{\|I^\varepsilon_2(t)-I^\varepsilon_2(s)\|^2_{H^{-1}}}{|t-s|^{1+2\alpha}}dsdt\\
&\leq& C_2(\alpha).
\end{eqnarray*}
By integration by parts formula and the linear growth of $A^{R}$, we have
\begin{eqnarray*}
\|div(A^{R}(u^{\eta,R}_{\varepsilon}(s)))\|_{H^{-1}}&=& \sup_{\|v\|_{H^1}\leq 1}|\langle v,div(A^{R}(u^{\eta,R}_\varepsilon(s)))\rangle|\\
&=& \sup_{\|v\|_{H^1}\leq 1}|\langle \nabla v, A^{R}(u^{\eta,R}_\varepsilon(s))\rangle|\\
&\leq& C(R)\sup_{\|v\|_{H^1}\leq 1}\int_{\mathbb{T}^N} |\nabla v|(1+ |u^{\eta,R}_\varepsilon(s)|)dx\\
&\leq& C(R)(1+\|u^{\eta,R}_\varepsilon(s)\|^2_{H})
\end{eqnarray*}
which gives that
\begin{eqnarray*}
\|I^\varepsilon_3(t)-I^\varepsilon_3(s)\|^2_{H^{-1}}
&=&\|\int^t_sdiv(A^{R}(u^{\eta,R}_\varepsilon(l)))dl \|^2_{H^{-1}}\\
&\leq& C(R)(t-s)\int^t_s\|div(A^{R}(u^{\eta,R}_\varepsilon(l)))\|^2_{H^{-1}}dl\\
&\leq& C(R)(t-s)\int^t_s(1+\|u^{\eta,R}_\varepsilon(l)\|^2_H)dl\\
&\leq& C(R)(t-s)^2[1+\sup_{t\in[0,T]}\|u^{\eta,R}_\varepsilon(t)\|^2_H].
\end{eqnarray*}
Hence, we deduce from (\ref{eqq-3}) that for $\alpha\in (0,\frac{1}{2})$,
\begin{eqnarray*}\notag
&&\sup_\varepsilon\|I^\varepsilon_3\|^2_{W^{\alpha,2}([0,T];H^{-1}(\mathbb{T}^N))}\\ \notag
&\leq&\int^T_0 \|I^\varepsilon_3(t)\|^2_{H^{-1}}dt+\int^T_0\int^T_0\frac{\|I^\varepsilon_3(t)-I^\varepsilon_3(s)\|^2_{H^{-1}}}{|t-s|^{1+2\alpha}}dsdt\\
&\leq& C_3(\alpha).
\end{eqnarray*}
Moreover, by (\ref{equ-30}), it follows that
\begin{eqnarray*}
\|\Phi(u^{\eta,R}_\varepsilon) h^{\varepsilon}(l)\|^2_{H}&\leq& \|\Phi(u^{\eta,R}_\varepsilon)\|^2_{\mathcal{L}_2(U,H)}|h^{\varepsilon}(l)|^2_U\\
&\leq& D_0(1+\|u^{\eta,R}_\varepsilon\|^2_H)|h^{\varepsilon}(l)|^2_U,
\end{eqnarray*}
then, by H\"{o}lder inequality, we get
\begin{eqnarray*}
\|I^\varepsilon_4(t)-I^\varepsilon_4(s)\|^2_{H}
&=&\|\int^t_s\Phi(u^{\eta,R}_\varepsilon) h^{\varepsilon}(l)dl \|^2_{H}\\
&\leq& (t-s)\int^t_s\|\Phi(u^{\eta,R}_\varepsilon) h^{\varepsilon}(l)\|^2_{H}dl\\
&\leq& D_0(t-s)(1+\sup_{t\in[0,T]}\|u^{\eta,R}_\varepsilon(t)\|^2_H)\int^t_s|h^{\varepsilon}(l)|^2_Udl\\
&\leq& D_0M(t-s)(1+\sup_{t\in[0,T]}\|u^{\eta,R}_\varepsilon(t)\|^2_H).
\end{eqnarray*}
Thus, we deduce from (\ref{eqq-3}) that for $\alpha\in (0,\frac{1}{2})$,
\begin{eqnarray*}\notag
&&\sup_\varepsilon\|I^\varepsilon_4\|^2_{W^{\alpha,2}([0,T];H)}\\ \notag
&\leq&\int^T_0 \|I^\varepsilon_4(t)\|^2_{H}dt+\int^T_0\int^T_0\frac{\|I^\varepsilon_4(t)-I^\varepsilon_4(s)\|^2_{H}}{|t-s|^{1+2\alpha}}dsdt\\
&\leq& C_4(\alpha).
\end{eqnarray*}
Collecting the above estimates, we conclude that for $\alpha\in (0,\frac{1}{2})$,
\begin{eqnarray*}\notag
\sup_\varepsilon\|u^{\eta,R}_\varepsilon\|^2_{W^{\alpha,2}([0,T];H^{-1}(\mathbb{T}^N))}\leq C(\alpha).
\end{eqnarray*}
Applying (\ref{eqq-3}) and Lemma \ref{lem-1-1}, we obtain the desired result.
\end{proof}

Furthermore, we apply the doubling of variables method to obtain the uniform convergence of the sequence $\{u^{\eta}_h, \eta>0\}$ to $u_h$ over $S_M$.

\begin{prp}\label{prp-4}
We have
\begin{eqnarray*}
\lim_{\eta\rightarrow 0}\sup_{h\in S_M}\|u^{\eta}_h-u_h\|_{L^1([0,T];L^1(\mathbb{T}^N))}=0.
\end{eqnarray*}

\end{prp}

\begin{proof}
{For any $h\in S_M$, we consider the kinetic solution $f_1(x,t,\xi)=I_{u_h(x,t)>\xi}$ of the skeleton equation (\ref{P-2}) with the corresponding kinetic measure ${m}_1$. As the proof of (\ref{P-5}), for $\varphi_1\in C^{1}_c(\mathbb{T}^N_x\times \mathbb{R}_\xi)$, we have
\begin{eqnarray}\notag
\langle {f}^{+}_1(t), \varphi_1\rangle&=&\langle {f}_{1,0}, \varphi_1\rangle +\int^t_0\langle {f}_1(s), a(\xi)\cdot \nabla_x \varphi_1 (x,\xi)\rangle ds\\
\label{eqq-2}
&& +\sum_{k\geq 1}\int^t_0\int_{\mathbb{T}^N}\int_{\mathbb{R}}\varphi_1(x,\xi) g_{k}(x,\xi)h^k(s)d\nu^1_{x,s}(\xi)dxds-\langle m_1,\partial_{\xi} \varphi_1\rangle([0,t]),
\end{eqnarray}
where $f_{1,0}=I_{u_0>\xi}$ and $\nu^1_{x,s}(\xi)=\partial_{\xi}\bar{f}_1(s,x,\xi)=-\partial_{\xi}{f}_1(s,x,\xi)=\delta_{u_h=\xi}$.
 From Section 2.2 and 2.3 in \cite{DHV}, we know that there exists a kinetic measure $m^{\eta}_2$ such that the kinetic solution $f^{\eta}_2(y,t,\zeta)=I_{u^{\eta}_h(y,t)>\zeta}$ of equation (\ref{eqq-1}) with $h^\varepsilon$ replaced by $h$ satisfies that
for $\varphi_2\in C^{1}_c(\mathbb{T}^N_y\times \mathbb{R}_\zeta)$,
\begin{eqnarray}\notag
\langle \bar{f}^{\eta,+}_2(t),\varphi_2\rangle&=&\langle \bar{f}_{2,0}, \varphi_2\rangle+\int^t_0\langle \bar{f}^{\eta}_2(s), a(\zeta)\cdot \nabla_y \varphi_2+ \eta\Delta_y\varphi_2(s)\rangle ds\\ \notag
&&\ -\sum_{k\geq 1}\int^t_0\int_{\mathbb{T}^N}\int_{\mathbb{R}} g_{k}(y,\zeta)\varphi_2 (y,\zeta)h^k(s)d\nu^{2,\eta}_{y,s}(\zeta)dyds \\
\label{qeq-27}
&& +\int^t_0\int_{\mathbb{T}^N}\int_{\mathbb{R}}\partial_{\zeta} \varphi_2(y,\zeta)dq^{\eta}(y,s,\zeta)+\langle m^{\eta}_2,\partial_{\zeta} \varphi_2\rangle([0,t]),
\end{eqnarray}
where $f_{2,0}=I_{u_0>\zeta}$, $q^{\eta}=\eta|\nabla u^{\eta}_h|^2\delta_{u^{\eta}_h=\zeta}$ and  $\nu^{2,\eta}_{y,s}(\zeta)=-\partial_{\zeta}f^{\eta}_2(s,y,\zeta)=\partial_{\zeta}\bar{f}^{\eta}_2(s,y,\zeta)=\delta_{u^{\eta}_h=\zeta}$.

}

Setting $\alpha(x,\xi,y,\zeta)=\varphi_1(x,\xi)\varphi_2(y,\zeta)$, by the same method as Proposition \ref{prp-1}, we have
\begin{eqnarray*}\notag
\langle\langle f^{+}_1(t)\bar{f}^{\eta,+}_2(t), \alpha \rangle\rangle
&=& \langle\langle f_{1,0}\bar{f}_{2,0}, \alpha \rangle\rangle+\int^t_0\int_{(\mathbb{T}^N)^2}\int_{\mathbb{R}^2}f_1\bar{f}^{\eta}_2(a(\xi)\cdot \nabla_x+a(\zeta)\cdot \nabla_y) \alpha d\xi d\zeta dxdyds\\
&& -\sum_{k\geq 1}\int^t_0\int_{(\mathbb{T}^N)^2}\int_{\mathbb{R}^2}f_1(s,x,\xi)\alpha g_{k}(y,\zeta)h^k(s)d\xi d\nu^{2,\eta}_{y,s}(\zeta)dxdyds\\ \notag
 && +\sum_{k\geq 1}\int^t_0\int_{(\mathbb{T}^N)^2}\int_{\mathbb{R}^2}\bar{f}^{\eta}_2(s,y,\zeta)\alpha g_{k}(x,\xi)h^k(s)d\zeta d\nu^{1}_{x,s}(\xi)dxdyds\\ \notag
 && +\int_{(0,t]}\int_{(\mathbb{T}^N)^2}\int_{\mathbb{R}^2}f^{-}_1(s,x,\xi)\partial_{\zeta} \alpha dm^{\eta}_2(y,\zeta,s)d\xi dx\\ \notag
\quad &&-\int_{(0,t]}\int_{(\mathbb{T}^N)^2}\int_{\mathbb{R}^2}\bar{f}^{\eta,+}_2(s,y,\zeta)\partial_{\xi} \alpha dm_1(x,\xi,s)d\zeta dy\\
&&\ +\eta \int^t_0\int_{(\mathbb{T}^N)^2}\int_{\mathbb{R}^2}f_1\bar{f}^{\eta}_2\Delta_y \alpha d\xi d\zeta dxdyds\\
&& +\int^t_0\int_{(\mathbb{T}^N)^2}\int_{\mathbb{R}^2}f^{-}_1\partial_{\zeta} \alpha dq^{\eta}(y,s,\zeta)d\xi dx\\
&=:&\langle\langle f_{1,0}\bar{f}_{2,0}, \alpha \rangle\rangle+ R_1+R_2+R_3+R_4+R_5+R_6+R_7.
\end{eqnarray*}
Similarly, we have
\begin{eqnarray*}\notag
\langle\langle \bar{f}^{+}_1(t)f^{\eta,+}_2(t), \alpha \rangle\rangle
&=& \langle\langle \bar{f}_{1,0}{f}_{2,0}, \alpha \rangle\rangle+\int^t_0\int_{(\mathbb{T}^N)^2}\int_{\mathbb{R}^2}\bar{f}_1{f}^{\eta}_2(a(\xi)\cdot \nabla_x+a(\zeta)\cdot \nabla_y) \alpha d\xi d\zeta dxdyds\\
 && +\sum_{k\geq 1}\int^t_0\int_{(\mathbb{T}^N)^2}\int_{\mathbb{R}^2}\bar{f}_1(s,x,\xi)\alpha g_{k}(y,\zeta)h^k(s)d\xi d\nu^{2,\eta}_{y,s}(\zeta)dxdyds\\ \notag
 && -\sum_{k\geq 1}\int^t_0\int_{(\mathbb{T}^N)^2}\int_{\mathbb{R}^2}{f}^{\eta}_2(s,y,\zeta)\alpha g_{k}(x,\xi)h^k(s)d\zeta d\nu^{1}_{x,s}(\xi)dxdyds\\ \notag
 && -\int_{(0,t]}\int_{(\mathbb{T}^N)^2}\int_{\mathbb{R}^2}\bar{f}^{+}_1(s,x,\xi)\partial_{\zeta} \alpha dm^{\eta}_2(y,\zeta,s)d\xi dx\\ \notag
 && +\int_{(0,t]}\int_{(\mathbb{T}^N)^2}\int_{\mathbb{R}^2}{f}^{\eta,-}_2(s,y,\zeta)\partial_{\xi} \alpha dm_1(x,\xi,s)d\zeta dy\\
&& +\eta \int^t_0\int_{(\mathbb{T}^N)^2}\int_{\mathbb{R}^2}\bar{f}_1{f}^{\eta}_2\Delta_y \alpha d\xi d\zeta dxdyds\\
&& -\int^t_0\int_{(\mathbb{T}^N)^2}\int_{\mathbb{R}^2}\bar{f}^{+}_1\partial_{\zeta} \alpha dq^{\eta}(y,s,\zeta)d\xi dx\\
&=:&\langle\langle \bar{f}_{1,0}{f}_{2,0}, \alpha \rangle\rangle+ \bar{R}_1+\bar{R}_2+\bar{R}_3+\bar{R}_4+\bar{R}_5+\bar{R}_6+\bar{R}_7.
\end{eqnarray*}
Referring to  Proposition 13 in \cite{D-V-1}, $R_4, R_5, R_7, \bar{R}_4, \bar{R}_5, \bar{R}_7$ are all non-positive. Then,
using the same argument as the proof of Proposition \ref{prp-1} and Theorem \ref{thm-2}, we take $\alpha=\rho_{\gamma}(x-y)\psi_{\delta}(\xi-\zeta)$ with $\rho_{\gamma}$ and $\psi_{\delta}$ being approximations to the identity on $\mathbb{T}^N$ and $\mathbb{R}$, respectively, it yields
\begin{eqnarray}\notag
&& \int_{(\mathbb{T}^N)^2}\int_{\mathbb{R}^2}\rho_\gamma (x-y)\psi_{\delta}(\xi-\zeta)(f^{\pm}_1(x,t,\xi)\bar{f}^{\eta,\pm}_2(y,t,\zeta)+\bar{f}^{\pm}_1(x,t,\xi){f}^{\eta,\pm}_2(y,t,\zeta))d\xi d\zeta dxdy\\ \notag
&\leq& \int_{(\mathbb{T}^N)^2}\int_{\mathbb{R}^2}\rho_\gamma (x-y)\psi_{\delta}(\xi-\zeta)(f_{1,0}(x,\xi)\bar{f}_{2,0}(y,\zeta)+\bar{f}_{1,0}(x,\xi){f}_{2,0}(y,\zeta))d\xi d\zeta dxdy\\
\label{eee-1}
&&\ +\tilde{R}_1+\tilde{\bar{R}}_1+\tilde{R}_2+\tilde{\bar{R}}_2+\tilde{R}_3+\tilde{\bar{R}}_3+\tilde{R}_6+\tilde{\bar{R}}_6.
\end{eqnarray}
where $ \tilde{R}_i, \tilde{\bar{R}}_i$ in (\ref{eee-1}) are the corresponding $R_i, \bar{R}_i$ with $\alpha=\rho_{\gamma}(x-y)\psi_{\delta}(\xi-\zeta)$ for $ i=1,2,3,6$.

From (\ref{P-11}), $\tilde{R}_1$ and $\tilde{\bar{R}}_1$ can be written as
\begin{eqnarray*}
\tilde{R}_1&=& \int^t_0\int_{(\mathbb{T}^N)^2}\int_{\mathbb{R}^2}f_1\bar{f}^{\eta}_2(a(\xi)-a(\zeta))\cdot \nabla_x \rho_{\gamma}(x-y)\psi_{\delta}(\xi-\zeta) d\xi d\zeta dxdyds,\\
\tilde{\bar{R}}_1&=& \int^t_0\int_{(\mathbb{T}^N)^2}\int_{\mathbb{R}^2}\bar{f}_1{f}^{\eta}_2(a(\xi)-a(\zeta))\cdot \nabla_x \rho_{\gamma}(x-y)\psi_{\delta}(\xi-\zeta) d\xi d\zeta dxdyds.
\end{eqnarray*}
Similarly to the proof of Theorem 15 in \cite{D-V-1}, we get
\begin{eqnarray*}
|\tilde{R}_1|\leq TC_p \delta \gamma^{-1}, \quad |\tilde{\bar{R}}_1|\leq TC_p \delta \gamma^{-1}.
\end{eqnarray*}

Moreover, with the aid of $\gamma_1(\xi,\zeta)=\int^{\xi}_{-\infty}\psi_{\delta}(\xi'-\zeta)d\xi', \gamma_2(\xi,\zeta)=\int^{\infty}_{\zeta}\psi_{\delta}(\xi-\zeta')d\zeta'$ and $\gamma_1(\xi,\zeta)=\gamma_2(\xi,\zeta)$, by the same arguments as the proof of Theorem \ref{thm-2}, it follows that
{
\begin{eqnarray*}\notag
\tilde{\bar{R}}_2+\tilde{\bar{R}}_3
&=&\tilde{R}_2+\tilde{R}_3\\
&=&
\sum_{k\geq 1}\int^t_0\int_{(\mathbb{T}^N)^2}\rho_{\gamma}(x-y)\int_{\mathbb{R}^2}\gamma_1(\xi,\zeta)(g_{k}(x,\xi)-g_{k}(y,\zeta))h^k(s)d \nu^{1}_{x,s}\otimes \nu^{2,\eta}_{y,s}(\xi,\zeta) dxdyds\\
&\leq&\int^t_0|h(s)|_U\int_{(\mathbb{T}^N)^2}\rho_{\gamma}(x-y)\int_{\mathbb{R}^2}\gamma_1(\xi,\zeta)\Big(\sum_{k\geq 1}|g_{k}(x,\xi)-g_{k}(y,\zeta)|^2\Big)^{\frac{1}{2}}d \nu^{1}_{x,s}\otimes \nu^{2,\eta}_{y,s}(\xi,\zeta) dxdyds.
\end{eqnarray*}
Applying the same method as the estimate of $\tilde{K}_2$ in Theorem \ref{thm-2}, we deduce that
\begin{eqnarray*}
\tilde{\bar{R}}_2+\tilde{\bar{R}}_3
&=&\tilde{R}_2+\tilde{R}_3\\
&\leq& \sqrt{D_1} \int^t_0|h(s)|_U\int_{(\mathbb{T}^N)^2}\rho_{\gamma}(x-y)\int_{\mathbb{R}^2}\gamma_1(\xi,\zeta)|x-y|d \nu^{1}_{x,s}\otimes \nu^{2,\eta}_{y,s}(\xi,\zeta) dxdyds\\
&&\ +\sqrt{D_1} \int^t_0|h(s)|_U\int_{(\mathbb{T}^N)^2}\rho_{\gamma}(x-y)\int_{\mathbb{R}^2}\gamma_1(\xi,\zeta)|\xi-\zeta|d \nu^{1}_{x,s}\otimes \nu^{2,\eta}_{y,s}(\xi,\zeta) dxdyds\\
 &\leq& (\gamma+2\delta)\sqrt{D_1}\Big(T+\int^T_0|h(s)|^2_Uds\Big)\\
\label{qeq-20}
&&\ +\sqrt{D_1}\int^t_0|h(s)|_U\int_{(\mathbb{T}^N)^2}\int_{\mathbb{R}^2}(f^{\pm}_1\bar{f}^{\eta,\pm}_2+\bar{f}^{\pm}_1{f}^{\eta,\pm}_2)\rho_{\gamma}(x-y)\psi_{\delta}(\xi-\zeta)dxdyd\xi d\zeta ds.
\end{eqnarray*}
}
For the  term $\tilde{R}_6$, it can be estimated as follows:{
\begin{eqnarray*}
 \tilde{R}_6&\leq& \eta \int^t_0\int_{(\mathbb{T}^N)^2}\int_{\mathbb{R}^2}f_1\bar{f}^{\eta}_2\Delta_y \rho_{\gamma}(x-y)\psi_{\delta}(\xi-\zeta) d\xi d\zeta dxdyds\\
 &=& \eta \int^t_0\int_{(\mathbb{T}^N)^2}\Delta_x \rho_{\gamma}(x-y)\Big[\int_{\mathbb{R}^2}f_1\bar{f}^{\eta}_2\psi_{\delta}(\xi-\zeta) d\xi d\zeta\Big] dxdyds\\
 &=& \eta \int^t_0\int_{(\mathbb{T}^N)^2}\Delta_x \rho_{\gamma}(x-y)\Big[\int_{\mathbb{R}^2}l(\xi, \zeta)d\nu^{1}_{x,s}\otimes \nu^{2,\eta}_{y,s}(\xi,\zeta)\Big] dxdyds,
\end{eqnarray*}
where
\begin{eqnarray*}
l(\xi, \zeta)=\int^{\infty}_{\zeta}\int^{\xi}_{-\infty}\psi_{\delta}(\xi'-\zeta')d\xi'd\zeta'.
\end{eqnarray*}
Moreover, let $\xi''=\xi'-\zeta'$, it follows that
\begin{eqnarray*}
l(\xi, \zeta)&\leq & \int^{\infty}_{\zeta}\left(\int_{\{|\xi''|<\delta,\xi''<\xi-\zeta' \}}\psi_{\delta}(\xi'')d\xi''\right)d\zeta'\\
&\leq & C\int^{\xi+\delta}_{\zeta}\delta \|\psi_{\delta}\|_{L^{\infty}}d\zeta'\\
&\leq &C(|\xi|+|\zeta|+\delta).
\end{eqnarray*}
Thus, we have
\begin{eqnarray*}
 \tilde{R}_6 &\leq& \eta C\int^t_0\int_{(\mathbb{T}^N)^2}\int_{\mathbb{R}^2}\Delta_x \rho_{\gamma}(x-y)(|\xi|+|\zeta|+\delta)d\nu^{1}_{x,s}\otimes d\nu^{2,\eta}_{y,s}(\xi,\zeta)dxdyds\\
 &\leq& C\gamma^{-2}\eta  \int^t_0\int_{(\mathbb{T}^N)^2}\int_{\mathbb{R}^2}(|\xi|+|\zeta|+\delta)d\nu^{1}_{x,s}\otimes \nu^{2,\eta}_{y,s}(\xi,\zeta)dxdyds\\
 &\leq& C(1+\delta)\eta T\gamma^{-2},
\end{eqnarray*}
where we have used the property that the measures $\nu^{1}$ and $\nu^{2,\eta}$ vanish at the infinity.
Similarly, using the same method as above, we conclude that $\tilde{\bar{R}}_6$ has the same estimate of $ \tilde{R}_6$.
}

Based on all the above estimates, it follows that
\begin{eqnarray*}\notag
&& \int_{(\mathbb{T}^N)^2}\int_{\mathbb{R}^2}\rho_\gamma (x-y)\psi_{\delta}(\xi-\zeta)(f^{\pm}_1(x,t,\xi)\bar{f}^{\eta,\pm}_2(y,t,\zeta)+\bar{f}^{\pm}_1(x,t,\xi){f}^{\eta,\pm}_2(y,t,\zeta))d\xi d\zeta dxdy\\ \notag
&\leq& \int_{(\mathbb{T}^N)^2}\int_{\mathbb{R}^2}\rho_\gamma (x-y)\psi_{\delta}(\xi-\zeta)(f_{1,0}(x,\xi)\bar{f}_{2,0}(y,\zeta)+\bar{f}_{1,0}(x,\xi){f}_{2,0}(y,\zeta))d\xi d\zeta dxdy\\
&&\ +2TC_p \delta \gamma^{-1}+2C(1+\delta)\eta T\gamma^{-2}+2(\gamma+2\delta)\sqrt{D_1}\Big(T+\int^T_0|h(s)|^2_Uds\Big)\\
&&\ +2\sqrt{D_1}\int^t_0|h(s)|_U\int_{(\mathbb{T}^N)^2}\int_{\mathbb{R}^2}(f^{\pm}_1\bar{f}^{\eta,\pm}_2+\bar{f}^{\pm}_1{f}^{\eta,\pm}_2)\rho_{\gamma}(x-y)\psi_{\delta}(\xi-\zeta)dxdyd\xi d\zeta ds,
\end{eqnarray*}
By Gronwall inequality, we get
\begin{eqnarray}\notag
&& \int_{(\mathbb{T}^N)^2}\int_{\mathbb{R}^2}\rho_\gamma (x-y)\psi_{\delta}(\xi-\zeta)(f^{\pm}_1(x,t,\xi)\bar{f}^{\eta,\pm}_2(y,t,\zeta)+\bar{f}^{\pm}_1(x,t,\xi){f}^{\eta,\pm}_2(y,t,\zeta))d\xi d\zeta dxdy\\ \notag
&\leq& e^{2\sqrt{D_1}(T+\int^T_0|h(s)|^2_Uds)}\Big[\int_{(\mathbb{T}^N)^2}\int_{\mathbb{R}^2}\rho_\gamma (x-y)\psi_{\delta}(\xi-\zeta)(f_{1,0}(x,\xi)\bar{f}_{2,0}(y,\zeta)+\bar{f}_{1,0}(x,\xi){f}_{2,0}(y,\zeta))d\xi d\zeta dxdy\Big]\\ \notag
&& +2e^{2\sqrt{D_1}(T+\int^T_0|h(s)|^2_Uds)}\Big[TC_p \delta \gamma^{-1}+C(1+\delta)\eta T\gamma^{-2}+(\gamma+2\delta)\sqrt{D_1}\Big(T+\int^T_0|h(s)|^2_Uds\Big)\Big]\\ \notag
&\leq& e^{2\sqrt{D_1}(T+\int^T_0|h(s)|^2_Uds)}\Big[\int_{\mathbb{T}^N}\int_{\mathbb{R}}
(f_{1,0}(x,\xi)\bar{f}_{2,0}(x,\xi)+\bar{f}_{1,0}(x,\xi){f}_{2,0}(x,\xi))dxd\xi+\mathcal{E}_0(\gamma,\delta)\Big]\\
\label{qeq-26}
&& +2e^{2\sqrt{D_1}(T+\int^T_0|h(s)|^2_Uds)}\Big[TC_p \delta \gamma^{-1}+C(1+\delta)\eta T\gamma^{-2}+(\gamma+2\delta)\sqrt{D_1}\Big(T+\int^T_0|h(s)|^2_Uds\Big)\Big],
\end{eqnarray}
where $\mathcal{E}_0(\gamma,\delta)$ is independent of $\eta$ and converges to $0$ as $\gamma, \delta\rightarrow 0$.

Let
\begin{eqnarray*}
\mathcal{E}_t(\eta,\gamma,\delta)&:=&\int_{\mathbb{T}^N}\int_{\mathbb{R}}(f^{\pm}_1(x,t,\xi)\bar{f}^{\eta,\pm}_2(x,t,\xi)+\bar{f}^{\pm}_1(x,t,\xi){f}^{\eta,\pm}_2(x,t,\xi))dxd\xi\\
&&\ -\int_{(\mathbb{T}^N)^2}\int_{\mathbb{R}^2}\rho_\gamma (x-y)\psi_{\delta}(\xi-\zeta)(f^{\pm}_1(x,t,\xi)\bar{f}^{\eta,\pm}_2(y,t,\zeta)+\bar{f}^{\pm}_1(x,t,\xi){f}^{\eta,\pm}_2(y,t,\zeta))d\xi d\zeta dxdy.
\end{eqnarray*}
We claim that $\mathcal{E}_t(\eta,\gamma,\delta)$ is independent of $\eta$. Indeed,
\begin{eqnarray*}
\mathcal{E}_t(\eta,\gamma,\delta)&=&\Big[\int_{\mathbb{T}^N}\int_{\mathbb{R}}(f^{\pm}_1(x,t,\xi)\bar{f}^{\eta,\pm}_2(x,t,\xi)+\bar{f}^{\pm}_1(x,t,\xi){f}^{\eta,\pm}_2(x,t,\xi))dxd\xi\\
&& -\int_{(\mathbb{T}^N)^2}\int_{\mathbb{R}}\rho_{\gamma}(x-y)(f^{\pm}_1(x,t,\xi)\bar{f}^{\eta,\pm}_2(y,t,\xi)+\bar{f}^{\pm}_1(x,t,\xi){f}^{\eta,\pm}_2(y,t,\xi))d\xi dxdy\Big]\\
&& +\Big[\int_{(\mathbb{T}^N)^2}\int_{\mathbb{R}}\rho_{\gamma}(x-y)(f^{\pm}_1(x,t,\xi)\bar{f}^{\eta,\pm}_2(y,t,\xi)+\bar{f}^{\pm}_1(x,t,\xi){f}^{\eta,\pm}_2(y,t,\xi))d\xi dxdy\\
&& -\int_{(\mathbb{T}^N)^2}\int_{\mathbb{R}^2}\rho_\gamma (x-y)\psi_{\delta}(\xi-\zeta)(f^{\pm}_1(x,t,\xi)\bar{f}^{\eta,\pm}_2(y,t,\zeta)+\bar{f}^{\pm}_1(x,t,\xi){f}^{\eta,\pm}_2(y,t,\zeta))d\xi d\zeta dxdy\Big]\\
&=:& H_1+H_2,
\end{eqnarray*}
Applying the same method as in the proofs of (\ref{qeq-14}) and (\ref{qeq-14-1}), it follows that
\begin{eqnarray}\label{qeq-25}
|H_2|\leq 2\delta,
\end{eqnarray}
and
\begin{eqnarray}\notag
|H_1|&\leq& \Big|\int_{(\mathbb{T}^N)^2}\rho_{\gamma}(x-y)\int_{\mathbb{R}}I_{u^{\pm}_h(x,t)>\xi}(I_{u^{\eta,\pm}_h(x,t)\leq \xi}-I_{u^{\eta,\pm}_h(y,t)\leq \xi})d\xi dxdy\Big|\\ \notag
&& +\Big|\int_{(\mathbb{T}^N)^2}\rho_{\gamma}(x-y)\int_{\mathbb{R}}I_{u^{\pm}_h(x,t)\leq\xi}(I_{u^{\eta,\pm}_h(x,t)> \xi}-I_{u^{\eta,\pm}_h(y,t)> \xi})d\xi dxdy\Big|\\
\label{qeq-30}
&\leq& 2\int_{(\mathbb{T}^N)^2}\rho_{\gamma}(x-y)|u^{\eta,\pm}_h(x,t)-u^{\eta,\pm}_h(y,t)|dxdy.
\end{eqnarray}
Combining (\ref{qeq-25}) and (\ref{qeq-30}), it yields
\begin{eqnarray}\notag
&&\mathcal{E}_t(\eta,\gamma,\delta)\\ \notag
&\leq& 2\delta+2\int_{(\mathbb{T}^N)^2}\rho_{\gamma}(x-y)|u^{\eta,\pm}_h(x,t)-u^{\eta,\pm}_h(y,t)|dxdy\\ \notag
&=& 2\delta+2\int_{(\mathbb{T}^N)^2}\int_{\mathbb{R}}\rho_{\gamma}(x-y)(f^{\eta,\pm}_2(x,t,\xi)\bar{f}^{\eta,\pm}_2(y,t,\xi)+\bar{f}^{\eta,\pm}_2(x,t,\xi){f}^{\eta,\pm}_2(y,t,\xi))d\xi dxdy\\
\notag
&\leq&  4\delta+2\int_{(\mathbb{T}^N)^2}\int_{\mathbb{R}^2}\rho_{\gamma}(x-y)
\psi_{\delta}(\xi-\zeta)\Big(f^{\eta,\pm}_2(x,t,\xi)\bar{f}^{\eta,\pm}_2(y,t,\zeta)\\
\label{qeq-28}
&&\ +\bar{f}^{\eta,\pm}_2(x,t,\xi){f}^{\eta,\pm}_2(y,t,\zeta)\Big)d\xi d\zeta dxdy.
\end{eqnarray}
Utilizing (\ref{qeq-27}) and applying the similar method as the proof of (\ref{qeq-26}), we obtain
\begin{eqnarray*}\notag
&& \int_{(\mathbb{T}^N)^2}\int_{\mathbb{R}^2}\rho_\gamma (x-y)\psi_{\delta}(\xi-\zeta)(f^{\eta,\pm}_2(x,t,\xi)\bar{f}^{\eta,\pm}_2(y,t,\zeta)+\bar{f}^{\eta,\pm}_2(x,t,\xi){f}^{\eta,\pm}_2(y,t,\zeta))d\xi d\zeta dxdy\\ \notag
&\leq& e^{\sqrt{D_1}(T+\int^T_0|h(s)|^2_Uds)}\Big[\int_{\mathbb{T}^N}\int_{\mathbb{R}}
(f_{2,0}\bar{f}_{2,0}+\bar{f}_{2,0}{f}_{2,0})dxd\xi+|\mathcal{E}_0(\gamma,\delta)|+J^{\sharp}\Big]\\
\notag
&& +2e^{\sqrt{D_1}(T+\int^T_0|h(s)|^2_Uds)}\Big[TC_p \delta \gamma^{-1}+(\gamma+2\delta)\sqrt{D_1}\Big(T+\int^T_0|h(s)|^2_Uds\Big)\Big],
\end{eqnarray*}
where
 \begin{eqnarray*}
J^{\sharp}&:=&2\eta\int^t_0\int_{(\mathbb{T}^N)^2}\int_{\mathbb{R}^2}({f}^{\eta,\pm}_2(x,s,\xi)\bar{f}^{\eta,\pm}_2(y,s,\zeta)+\bar{f}^{\eta,\pm}_2(x,s,\xi){f}^{\eta,\pm}_2(y,s,\zeta))\Delta_x \alpha d\xi d\zeta dxdyds\\
&& -2\int^t_0\int_{(\mathbb{T}^N)^2}\int_{\mathbb{R}^2} \alpha dq^{\eta}(y,s,\zeta) d\nu^1_{x,s}(\xi) dx-2\int^t_0\int_{(\mathbb{T}^N)^2}\int_{\mathbb{R}^2} \alpha dq^{\eta}(x,s,\xi) d\nu^2_{y,s}(\zeta) dy,
\end{eqnarray*}
 and
 $\mathcal{E}_0(\gamma,\delta)$ is different from that in (\ref{qeq-26}) but they both converge to  0, so we do not distinguish them. From Proposition 6.1 in \cite{DHV}, we know that $J^{\sharp}\leq0.$

Hence,
\begin{eqnarray}\notag
&& \int_{(\mathbb{T}^N)^2}\int_{\mathbb{R}^2}\rho_\gamma (x-y)\psi_{\delta}(\xi-\zeta)(f^{\eta,\pm}_2(x,t,\xi)\bar{f}^{\eta,\pm}_2(y,t,\zeta)+\bar{f}^{\eta,\pm}_2(x,t,\xi){f}^{\eta,\pm}_2(y,t,\zeta))d\xi d\zeta dxdy\\
\label{qeq-26-2}
&\leq&e^{2\sqrt{D_1}(T+\int^T_0|h(s)|^2_Uds)}\Big[|\mathcal{E}_0(\gamma,\delta)|+2TC_p \delta \gamma^{-1}+2(\gamma+2\delta)\sqrt{D_1}\Big(T+\int^T_0|h(s)|^2_Uds\Big)\Big],
\end{eqnarray}

Combining (\ref{qeq-28}) and (\ref{qeq-26-2}), we conclude that
\begin{eqnarray}\notag
&&\mathcal{E}_t(\eta,\gamma,\delta)\\
\label{qeq-29}
&\leq& 4\delta+2e^{2\sqrt{D_1}(T+\int^T_0|h(s)|^2_Uds)}\Big[|\mathcal{E}_0(\gamma,\delta)|+2TC_p \delta \gamma^{-1}+2(\gamma+2\delta)\sqrt{D_1}\Big(T+\int^T_0|h(s)|^2_Uds\Big)\Big],
\end{eqnarray}
which implies that $\mathcal{E}_t(\eta,\gamma,\delta)$ is independent of $\eta$, so we denote  that $\mathcal{E}_t(\gamma,\delta):=\mathcal{E}_t(\eta,\gamma,\delta)$.

From (\ref{qeq-26}) and (\ref{qeq-29}), we deduce that
\begin{eqnarray*}
&&\int_{\mathbb{T}^N}\int_{\mathbb{R}}(f^{\pm}_1(x,t,\xi)\bar{f}^{\eta,\pm}_2(x,t,\xi)+\bar{f}^{\pm}_1(x,t,\xi){f}^{\eta,\pm}_2(x,t,\xi))dxd\xi\\ \notag
&\leq& \int_{(\mathbb{T}^N)^2}\int_{\mathbb{R}^2}\rho_\gamma (x-y)\psi_{\delta}(\xi-\zeta)(f^{\pm}_1(x,t,\xi)\bar{f}^{\eta,\pm}_2(y,t,\zeta)+\bar{f}^{\pm}_1(x,t,\xi){f}^{\eta,\pm}_2(y,t,\zeta))d\xi d\zeta dxdy+\mathcal{E}_t(\gamma,\delta)\\ \notag
&\leq& e^{2\sqrt{D_1}(T+\int^T_0|h(s)|^2_Uds)}\Big[\int_{\mathbb{T}^N}\int_{\mathbb{R}}(f_{1,0}\bar{f}_{2,0}+\bar{f}_{1,0}{f}_{2,0})dxd\xi +|\mathcal{E}_0(\gamma,\delta)|+2TC_p \delta \gamma^{-1}\Big]+\mathcal{E}_t(\gamma,\delta)\\
&& +2e^{2\sqrt{D_1}(T+\int^T_0|h(s)|^2_Uds)}\Big[C(1+\delta)\eta T\gamma^{-2}+(\gamma+2\delta)\sqrt{D_1}\Big(T+\int^T_0|h(s)|^2_Uds\Big)\Big]\\
&\leq&e^{2\sqrt{D_1}(T+\int^T_0|h(s)|^2_Uds)}\Big[3|\mathcal{E}_0(\gamma,\delta)|+6TC_p \delta \gamma^{-1}\Big]+4\delta\\
&& +2e^{2\sqrt{D_1}(T+\int^T_0|h(s)|^2_Uds)}\Big[C(1+\delta)\eta T\gamma^{-2}+3(\gamma+2\delta)\sqrt{D_1}\Big(T+\int^T_0|h(s)|^2_Uds\Big)\Big].
\end{eqnarray*}
Then, we obtain
{
\begin{eqnarray*}\notag
&&\sup_{h\in S_M}\int^T_0\int_{\mathbb{T}^N}\int_{\mathbb{R}}(f^{\pm}_1(x,t,\xi)\bar{f}^{\eta,\pm}_2(x,t,\xi)+\bar{f}^{\pm}_1(x,t,\xi){f}^{\eta,\pm}_2(x,t,\xi))dxd\xi dt\\ \notag
&\leq& Te^{2\sqrt{D_1}(T+M)}\Big[3|\mathcal{E}_0(\gamma,\delta)|+6TC_p \delta \gamma^{-1}\Big]+4\delta T\\
&& +2Te^{2\sqrt{D_1}(T+M)}\Big[C(1+\delta)\eta T\gamma^{-2}+3(\gamma+2\delta)\sqrt{D_1}(T+M)\Big].
\end{eqnarray*}
Taking  $\delta=\gamma^{\frac{4}{3}}$ and $\gamma=\eta^{\frac{1}{3}}$, we get
\begin{eqnarray*}
&&\sup_{h\in S_M}\int^T_0\int_{\mathbb{T}^N}\int_{\mathbb{R}}(f^{\pm}_1(x,t,\xi)\bar{f}^{\eta,\pm}_2(x,t,\xi)+\bar{f}^{\pm}_1(x,t,\xi){f}^{\eta,\pm}_2(x,t,\xi))dxd\xi dt\\
&\leq& Te^{2\sqrt{D_1}(T+M)}\Big[3|\mathcal{E}_0(\gamma,\delta)|+6TC_p\eta^{\frac{1}{9}}\Big]+4\eta^{\frac{4}{9}}T\\
&& +2Te^{2\sqrt{D_1}(T+M)}\Big[CT(1+\eta^{\frac{4}{9}})\eta^{\frac{1}{3}} +3(\eta^{\frac{1}{3}}+2\eta^{\frac{4}{9}})\sqrt{D_1}(T+M)\Big].
\end{eqnarray*}

We deduce further from the following identities
\begin{eqnarray*}
\int_{\mathbb{R}}I_{u_h>\xi}\overline{I_{u^{\eta}_h>\xi}}d\xi=(u_h-u^{\eta}_h)^+,
\quad
\int_{\mathbb{R}}\overline{I_{u_h>\xi}}I_{u^{\eta}_h>\xi}d\xi=(u_h-u^{\eta}_h)^-,
\end{eqnarray*}
that
\begin{eqnarray*}
&&\sup_{h\in S_M}\int^T_0\|u^{\eta}_h(t)-u_h(t)\|_{L^1(\mathbb{T}^N)}dt\\
&\leq& Te^{2\sqrt{D_1}(T+M)}[3|\mathcal{E}_0(\gamma,\delta)|+2TC_p \eta^{\frac{1}{9}}+6CT(1+\eta^{\frac{4}{9}})\eta^{\frac{1}{3}}]+4\eta^{\frac{4}{9}}T\\
&&\ +6Te^{2\sqrt{D_1}(T+M)}(\eta^{\frac{1}{3}}+2\eta^{\frac{4}{9}})\sqrt{D_1}(T+M).
\end{eqnarray*}
Therefore, we get
\begin{eqnarray}
\lim_{\eta\rightarrow 0}\sup_{h\in S_M}\|u^{\eta}_h-u_h\|_{L^1([0,T];L^1(\mathbb{T}^N))}=0.
\end{eqnarray}
We complete the proof.
}
\end{proof}
Now, we are in a position to prove the continuity of $\mathcal{G}^0$.
\begin{thm}\label{thmm-1}
Assume $h^\varepsilon\rightarrow h$ weakly in $L^2([0,T];U)$. Then $u_{h^{\varepsilon}}$ converges to $u_h$ in $L^1([0,T];L^1(\mathbb{T}^N))$, where $u_{h^{\varepsilon}}$ is the kinetic solution of (\ref{P-6}) with $h$ replaced by $h^{\varepsilon}$.
\end{thm}
\begin{proof}
Fix any $\eta, R>0$.
For the solution $u^{\eta,R}_{h^{\varepsilon}}$ of (\ref{eqq-12}), we shall firstly prove that when $h^{\varepsilon}\rightarrow h$ weakly in $L^2([0,T];U)$, we have $\lim_{\varepsilon\rightarrow 0}\|u^{\eta,R}_{h^{\varepsilon}}-u^{\eta,R}_h\|_{L^1([0,T];L^1(\mathbb{T}^N))}=0$, where $u^{\eta,R}_{h}$ is the solution of (\ref{eqq-12}) with $h^{\varepsilon}$ replaced by $h$.

In fact, by the chain rule, we have
\begin{eqnarray*}
&&\|u^{\eta,R}_{h^{\varepsilon}}(t)-u^{\eta,R}_h(t)\|^2_{H}+2\eta \int^t_0 \|\nabla (u^{\eta,R}_{h^{\varepsilon}}-u^{\eta,R}_h)\|^2_{H}ds\\
&\leq&  2\int^t_0\langle A^{R}(u^{\eta,R}_{h^{\varepsilon}})-A^{R}(u^{\eta,R}_{h}),\nabla( u^{\eta,R}_{h^{\varepsilon}}-u^{\eta,R}_h)\rangle ds+ 2\int^t_0\langle \Phi(u^{\eta,R}_{h^{\varepsilon}})h^{\varepsilon}(s)-\Phi(u^{\eta,R}_{h})h(s),u^{\eta,R}_{h^{\varepsilon}}-u^{\eta,R}_h\rangle ds\\
&\leq & 2\int^t_0\langle A^{R}(u^{\eta,R}_{h^{\varepsilon}})-A^{R}(u^{\eta,R}_{h}),\nabla( u^{\eta,R}_{h^{\varepsilon}}-u^{\eta,R}_h)\rangle ds+2\int^t_0\langle (\Phi(u^{\eta,R}_{h^{\varepsilon}})-\Phi(u^{\eta,R}_h))h^{\varepsilon}(s),u^{\eta,R}_{h^{\varepsilon}}-u^{\eta,R}_h\rangle ds\\
&&+\ 2\int^t_0\langle \Phi(u^{\eta,R}_{h})(h^{\varepsilon}(s)-h(s)),u^{\eta,R}_{h^{\varepsilon}}-u^{\eta,R}_h\rangle ds\\
&=:& I_1+I_2+2\int^t_0\langle \Phi(u^{\eta,R}_{h})(h^{\varepsilon}(s)-h(s)),u^{\eta,R}_{h^{\varepsilon}}-u^{\eta,R}_h\rangle ds.
\end{eqnarray*}
Using the H\"{o}lder inequality and Lipschitz continuous of $A^{R}$, we get
\begin{eqnarray*}
I_1&\leq& 2R\int^t_0\|\nabla( u^{\eta,R}_{h^{\varepsilon}}-u^{\eta,R}_h)\|_H \|u^{\eta,R}_{h^{\varepsilon}}-u^{\eta,R}_{h}\|_Hds\\
&\leq &\eta \int^t_0 \|\nabla (u^{\eta,R}_{h^{\varepsilon}}-u^{\eta,R}_h)\|^2_{H}ds+C(\eta, R)\int^t_0 \|u^{\eta,R}_{h^{\varepsilon}}-u^{\eta,R}_h\|^2_{H}ds.
\end{eqnarray*}
Using  H\"{o}lder inequality and (\ref{equ-30-1}), we obtain
\begin{eqnarray*}
I_2&\leq& 2\int^t_0 \|\Phi(u^{\eta,R}_{h^{\varepsilon}})-\Phi(u^{\eta,R}_h)\|_{\mathcal{L}_2(U,H)}|h^{\varepsilon}(s)|_U\|u^{\eta,R}_{h^{\varepsilon}}-u^{\eta,R}_h\|_Hds\\
&\leq& \sqrt{D_1}\int^t_0\|u^{\eta,R}_{h^{\varepsilon}}-u^{\eta,R}_h\|^2_H|h^{\varepsilon}(s)|_Uds.
\end{eqnarray*}
Hence, it follows that
\begin{eqnarray*}
&&\sup_{t\in[0,T]}\|u^{\eta,R}_{h^{\varepsilon}}(t)-u^{\eta,R}_h(t)\|^2_{H}+\eta \int^T_0 \|\nabla (u^{\eta,R}_{h^{\varepsilon}}-u^{\eta,R}_h)\|^2_{H}ds\\
&\leq& C(\eta, R)\int^T_0 \|u^{\eta,R}_{h^{\varepsilon}}-u^{\eta,R}_h\|^2_H(1+|h^{\varepsilon}(s)|_U)ds\\
&&\ +2\sup_{t\in[0,T]}\Big|\int^t_0\langle \Phi(u^{\eta,R}_{h})(h^{\varepsilon}(s)-h(s)),u^{\eta,R}_{h^{\varepsilon}}-u^{\eta,R}_h\rangle ds\Big|.
\end{eqnarray*}
By the Gronwall inequality, we obtain
\begin{eqnarray*}
&&\sup_{t\in[0,T]}\|u^{\eta,R}_{h^{\varepsilon}}(t)-u^{\eta,R}_h(t)\|^2_{H}+\eta \int^T_0 \|\nabla (u^{\eta,R}_{h^{\varepsilon}}-u^{\eta,R}_h)\|^2_{H}ds\\
&\leq& 2\sup_{t\in [0,T]}\Big|\int^t_0\langle \Phi(u^{\eta,R}_{h})(h^{\varepsilon}(s)-h(s)),u^{\eta,R}_{h^{\varepsilon}}-u^{\eta,R}_h\rangle ds\Big| \exp\Big\{C(\eta, R)\int^T_0(1+|h^{\varepsilon}(s)|^2_U)ds\Big\}\\
&\leq& C(\eta, R,T,M)\sup_{t\in[0,T]}\Big|\int^t_0\langle \Phi(u^{\eta,R}_{h})(h^{\varepsilon}(s)-h(s)),u^{\eta,R}_{h^{\varepsilon}}-u^{\eta,R}_h\rangle ds\Big|.
\end{eqnarray*}
To show $\lim_{\varepsilon\rightarrow 0}\sup_{t\in[0,T]}\|u^{\eta,R}_{h^{\varepsilon}}(t)-u^{\eta,R}_h(t)\|^2_{H}=0$, it suffices to prove that
\begin{eqnarray*}
\lim_{\varepsilon\rightarrow 0}\sup_{0\leq t\leq T}\Big|\int^t_0\langle \Phi(u^{\eta,R}_{h})(h^{\varepsilon}-h),u^{\eta,R}_{h^{\varepsilon}}-u^{\eta,R}_h\rangle ds\Big|=0.
\end{eqnarray*}
This will be achieved if we show that for any sequence ${\varepsilon}_m\rightarrow 0$, one can find a subsequence ${\varepsilon}_{m_{k}}\rightarrow  0$ such that
\begin{eqnarray}\label{eqq-4}
\lim_{k\rightarrow \infty}\sup_{0\leq t\leq T}\Big|\int^t_0\langle \Phi(u^{\eta,R}_{h})(h^{{\varepsilon}_{m_k}}-h),u^{\eta,R}_{h^{{\varepsilon}_{m_k}}}-u^{\eta,R}_h\rangle ds\Big|=0.
\end{eqnarray}
Now fix a sequence ${\varepsilon}_m\rightarrow 0$. Since $\{u^{\eta,R}_{h^{{\varepsilon}_{m}}}, m\geq 1\}$ is compact in $L^2([0,T];H)$, there exists a subsequence \{$m_k$, $k\geq 1$\} and a mapping $\tilde{u}$  such that $u^{\eta,R}_{h^{{\varepsilon}_{m_k}}}\rightarrow \tilde{u}$ in $L^2([0,T];H)$.  Now, note that
\begin{eqnarray*}
&&\sup_{0\leq t\leq T}\Big|\int^t_0\langle \Phi(u^{\eta,R}_{h})(h^{{\varepsilon}_{m_k}}-h),u^{\eta,R}_{h^{{\varepsilon}_{m_k}}}-u^{\eta,R}_h\rangle ds\Big|\\
&\leq& \sup_{0\leq t\leq T}\Big|\int^t_0\langle \Phi(u^{\eta,R}_{h})(h^{{\varepsilon}_{m_k}}-h),u^{\eta,R}_{h^{{\varepsilon}_{m_k}}}-\tilde{u}\rangle ds\Big|\\
&&\ +\sup_{0\leq t\leq T}\Big|\int^t_0\langle \Phi(u^{\eta,R}_{h})(h^{{\varepsilon}_{m_k}}-h),\tilde{u}-u^{\eta,R}_h\rangle ds\Big|.
\end{eqnarray*}
Since $h^{{\varepsilon}_{m_k}}\rightarrow h$ weakly in $L^2([0,T];U)$, for every $t>0$, it follows that
\begin{eqnarray}\label{eqq-14}
\int^t_0\langle \Phi(u^{\eta,R}_{h})(h^{{\varepsilon}_{m_k}}-h),\tilde{u}-u^{\eta,R}_h\rangle ds\rightarrow0,\ {\rm{as}}\ k\rightarrow \infty.
\end{eqnarray}
On the other hand, by (\ref{equ-30}) and utilizing the assumption on $h$, for $0<t_1<t_2\leq T$, we have
\begin{eqnarray}\notag
&&\Big|\int^{t_2}_{t_1}\langle \Phi(u^{\eta,R}_{h})(h^{{\varepsilon}_{m_k}}-h),\tilde{u}-u^{\eta,R}_h\rangle ds\Big|\\
\notag
&\leq& \int^{t_2}_{t_1} \|\tilde{u}-u^{\eta,R}_h\|_H \|\Phi(u^{\eta,R}_h)\|_{\mathcal{L}_2(U,H)}|h^{{\varepsilon}_{m_k}}-h|_Uds\\
\notag
&\leq& \sqrt{D_0} \int^{t_2}_{t_1} \|\tilde{u}-u^{\eta,R}_h\|_H (1+\|u^{\eta,R}_h\|_H)|h^{{\varepsilon}_{m_k}}-h|_Uds\\ \notag
&\leq& \sqrt{D_0}(2M)^{\frac{1}{2}}\Big(1+\sup_{t\in[0,T]}\|u^{\eta,R}_h\|_H\Big)\Big(\int^{t_2}_{t_1}\|\tilde{u}-u^{\eta,R}_h\|^2_H ds\Big)^{\frac{1}{2}}\\
\label{eqq-15}
&\leq& \sqrt{D_0}C(M,T, \|u_0\|_H)\Big(\int^{t_2}_{t_1}\|\tilde{u}-u^{\eta,R}_h\|^2_H ds\Big)^{\frac{1}{2}}.
\end{eqnarray}
Combining (\ref{eqq-14}) and (\ref{eqq-15}), we deduce that
\begin{eqnarray*}
\lim_{k\rightarrow \infty}\sup_{0\leq t\leq T}\Big|\int^t_0\langle \Phi(u^{\eta,R}_{h})(h^{{\varepsilon}_{m_k}}-h),\tilde{u}-u^{\eta,R}_h\rangle ds\Big|=0.
\end{eqnarray*}
By H\"{o}lder inequality, we have
\begin{eqnarray*}
\sup_{t\in[0,T]}\Big|\int^t_0\langle \Phi(u^{\eta,R}_{h})(h^{{\varepsilon}_{m_k}}-h),u^{\eta,R}_{h^{{\varepsilon}_{m_k}}}-\tilde{u}\rangle ds\Big|
\leq \sqrt{D_0}C(M,T, \|u_0\|_H)\Big(\int^{T}_{0}\|u^{\eta,R}_{h^{{\varepsilon}_{m_k}}}-\tilde{u}\|^2_Hds\Big)^{\frac{1}{2}}.
\end{eqnarray*}
Since
$u^{\eta,R}_{h^{{\varepsilon}_{m_k}}}\rightarrow \tilde{u}$ in $L^2([0,T];H)$, we obtain
\begin{eqnarray*}
\lim_{k\rightarrow \infty}\sup_{t\in[0,T]}\Big|\int^t_0\langle \Phi(u^{\eta,R}_{h})(h^{{\varepsilon}_{m_k}}-h),u^{\eta,R}_{h^{{\varepsilon}_{m_k}}}-\tilde{u}\rangle ds\Big|=0.
\end{eqnarray*}
Collecting the above estimates, we prove (\ref{eqq-4}).
Hence
\begin{eqnarray*}
\lim_{\varepsilon\rightarrow 0}\sup_{t\in[0,T]}\|u^{\eta,R}_{h^{\varepsilon}}(t)-u^{\eta,R}_h(t)\|^2_{H}=0,
\end{eqnarray*}
which further implies that for any $\eta>0$, $R>0$,
\begin{eqnarray}\label{eqq-5}
\lim_{\varepsilon\rightarrow 0}\|u^{\eta,R}_{h^{\varepsilon}}-u^{\eta,R}_h\|_{L^1([0,T];L^1(\mathbb{T}^N))}=0.
\end{eqnarray}
Note that for any $\varepsilon, \eta, R>0$,
\begin{eqnarray}\notag
&&\|u_{h^{\varepsilon}}-u_h\|_{L^1([0,T];L^1(\mathbb{T}^N))}\\ \notag
&\leq& \|u^{\eta}_{h^{\varepsilon}}-u_{h^{\varepsilon}}\|_{L^1([0,T];L^1(\mathbb{T}^N))}
+\|u^{\eta}_{h^{\varepsilon}}-u^{\eta,R}_{h^{\varepsilon}}\|_{L^1([0,T];L^1(\mathbb{T}^N))}
+\|u^{\eta,R}_{h^{\varepsilon}}-u^{\eta,R}_{h}\|_{L^1([0,T];L^1(\mathbb{T}^N))}\\
\label{eqq-6}
&&\
+\|u^{\eta,R}_{h}-u^{\eta}_{h}\|_{L^1([0,T];L^1(\mathbb{T}^N))}
+\|u^{\eta}_h-u_h\|_{L^1([0,T];L^1(\mathbb{T}^N))}.
\end{eqnarray}
For any $\iota>0$, by Proposition \ref{prp-4}, there exists $\eta_{0}$ such that for all $\varepsilon>0$,
$$\|u^{\eta_0}_{h^{\varepsilon}}-u_{h^{\varepsilon}}\|_{L^1([0,T];L^1(\mathbb{T}^N))}\leq \frac{\iota}{4} \,\,\mbox{and}\,\, \|u^{\eta_0}_h-u_h\|_{L^1([0,T];L^1(\mathbb{T}^N))}\leq \frac{\iota}{4}.$$
Letting $\eta=\eta_0$, we deduce from (\ref{eqq-6}) that
\begin{eqnarray}\notag
\|u_{h^{\varepsilon}}-u_h\|_{L^1([0,T];L^1(\mathbb{T}^N))}
&\leq& \frac{\iota}{2}
+\|u^{\eta_0}_{h^{\varepsilon}}-u^{\eta_0,R}_{h^{\varepsilon}}\|_{L^1([0,T];L^1(\mathbb{T}^N))}+\|u^{\eta_0,R}_{h^{\varepsilon}}-u^{\eta_0,R}_{h}\|_{L^1([0,T];L^1(\mathbb{T}^N))}\\
\label{eqq-16}
&&\
+\|u^{\eta_0,R}_{h}-u^{\eta_0}_{h}\|_{L^1([0,T];L^1(\mathbb{T}^N))}.
\end{eqnarray}
Using (\ref{eqq-13}), there exists $R_0$ large enough such that for all $\varepsilon>0$, $$\|u^{\eta_0}_{h^{\varepsilon}}-u^{\eta_0,R_0}_{h^{\varepsilon}}\|_{L^1([0,T];L^1(\mathbb{T}^N))}\leq \frac{\iota}{4}\,\, \mbox{and} \,\, \|u^{\eta_0,R_0}_{h}-u^{\eta_0}_{h}\|_{L^1([0,T];L^1(\mathbb{T}^N))}\leq \frac{\iota}{4}.$$
Replacing $R$ by $R_0$ in (\ref{eqq-16}), we get
\begin{eqnarray*}
\|u_{h^{\varepsilon}}-u_h\|_{L^1([0,T];L^1(\mathbb{T}^N))}\leq \iota+\|u^{\eta_0,R_0}_{h^{\varepsilon}}-u^{\eta_0,R_0}_{h}\|_{L^1([0,T];L^1(\mathbb{T}^N))}.
\end{eqnarray*}
Using (\ref{eqq-5}), we conclude that
\begin{eqnarray*}
\lim_{\varepsilon\rightarrow 0}\|u_{h^{\varepsilon}}-u_h\|_{L^1([0,T];L^1(\mathbb{T}^N))}\leq \iota.
\end{eqnarray*}
Since the constant $\iota$ is arbitrary, we obtain the desired result.
\end{proof}

\section{Large deviations}

For any family $\{h^{\varepsilon}; 0<\varepsilon<1\}\subset \mathcal{A}_M$ with $h^{\varepsilon}=\sum_{k\geq 1}h^{\varepsilon,k}e_k$, we consider the following equation
\begin{eqnarray}\label{P-20}
\left\{
  \begin{array}{ll}
  d\bar{u}^{\varepsilon}+div(A(\bar{u}^{\varepsilon}))dt=\Phi(\bar{u}^{\varepsilon}) h^{\varepsilon}(t)dt+\sqrt{\varepsilon}\Phi(\bar{u}^{\varepsilon}) dW(t),\\
\bar{u}^{\varepsilon}(0)=u_0.
  \end{array}
\right.
\end{eqnarray}
{Combining Theorem \ref{thm-4} and Theorem \ref{thm-2},} we conclude that there exists a unique kinetic solution $\bar{u}^{\varepsilon}$ with initial data $u_0\in L^{\infty}(\mathbb{T}^N)$ satisfying that for any $p\geq 1$
\[
\mathbb{E}\Big(\underset{0\leq t\leq T}{{\rm{ess\sup}}}\ \|\bar{u}^{\varepsilon}(t)\|^p_{L^p(\mathbb{T}^N)}\Big)\leq C_p,
\]
and there exists a kinetic measure $m^\varepsilon\in\mathcal{M}^+_0(\mathbb{T}^N\times [0,T]\times \mathbb{R})$ such that $f^\varepsilon:=I_{\bar{u}^{\varepsilon}>\xi}$ fulfills that for all $\varphi\in C_c^1(\mathbb{T}^N\times [0,T)\times\mathbb{R})$,
\begin{eqnarray}\notag
&&\int^T_0\langle f^\varepsilon(t), \partial_t \varphi(t)\rangle dt+\langle f_0, \varphi(0)\rangle +\int^T_0\langle f^\varepsilon(t), a(\xi)\cdot \nabla \varphi (t)\rangle dt\\ \notag
&=& -\sqrt{\varepsilon}\sum_{k\geq 1}\int^T_0\int_{\mathbb{T}^N} g_k(x,\bar{u}^\varepsilon(x,t))\varphi(x,t,\bar{u}^\varepsilon(x,t))dxd\beta_k(t)\\ \notag
&&\ -\frac{\varepsilon}{2}\int^T_0\int_{\mathbb{T}^N}\partial_{\xi}\varphi(x,t,\bar{u}^\varepsilon(x,t))G^2(x,\bar{u}^\varepsilon(x,t)) dxdt\\
\label{equ-3}
&&\ -\sum_{k\geq 1}\int^T_0\int_{\mathbb{T}^N}\varphi(x,t,\bar{u}^\varepsilon(x,t)) g_k(x,\bar{u}^\varepsilon(x,t))h^{\varepsilon,k}(t)dxdt + m^\varepsilon(\partial_{\xi} \varphi),\quad a.s.
\end{eqnarray}
 where $G^2:=\sum_{k\geq1}|g_k|^2$. According to the definition of $\mathcal{G}^{\varepsilon}$, it is clear  that $\mathcal{G}^{\varepsilon}\Big(W(\cdot)+\frac{1}{\sqrt{\varepsilon}}\int^{\cdot}_0h^{\varepsilon}(s)ds\Big)=\bar{u}^{\varepsilon}(\cdot)$.

According to Theorem \ref{thm-7} (the sufficient condition B) and Theorem \ref{thmm-1}, we only need to prove the following result to establish the main result.
\begin{thm}\label{thm-5}
For every $M<\infty$, let $\{h^{\varepsilon}: \varepsilon>0\}$ $\subset \mathcal{A}_M$. Then
\[
\Big\|\mathcal{G}^{\varepsilon}\Big(W(\cdot)+\frac{1}{\sqrt{\varepsilon}}\int^{\cdot}_0h^\varepsilon(s)ds\Big)-\mathcal{G}^{0}\Big(\int^{\cdot}_0h^\varepsilon(s)ds\Big)\Big\|_{L^1([0,T];L^1(\mathbb{T}^N))}\rightarrow 0\quad {\rm{in \ probability}}.
\]

\end{thm}
\begin{proof}
Recall that $\bar{u}^\varepsilon=\mathcal{G}^{\varepsilon}\Big(W(\cdot)+\frac{1}{\sqrt{\varepsilon}}\int^{\cdot}_0h^\varepsilon(s)ds\Big)$ is the kinetic solution to (\ref{P-20}) with the corresponding kinetic measure ${m}^{\varepsilon}_1$. Moreover, $v^{\varepsilon}:=\mathcal{G}^{0}\Big(\int^{\cdot}_0h^\varepsilon(s)ds\Big)$ is the kinetic solution to the skeleton equation (\ref{P-2}) with $h$ replaced by $h^\varepsilon$ and the corresponding kinetic measure is denoted by $\bar{m}^{\varepsilon}_2$.

Denote $f_1(x,t,\xi):=I_{\bar{u}^\varepsilon(x,t)>\xi}$ and $f_2(y,t,\zeta):=I_{v^{\varepsilon}(y,t)>\zeta}$. Applying the same procedure as (\ref{e-14}), for any $\varphi_1(x,\xi)\in C^{\infty}_c(\mathbb{T}^N_x\times \mathbb{R}_{\xi})$, we have
\begin{eqnarray*}
\langle {f}^{+}_1(t), \varphi_1\rangle&=&\langle {f}_{1,0}, \varphi_1\rangle +\int^t_0\langle {f}_1(s), a(\xi)\cdot \nabla_x \varphi_1 (x,\xi)\rangle ds\\
&& +\sqrt{\varepsilon}\sum_{k\geq 1}\int^t_0\int_{\mathbb{T}^N}\int_{\mathbb{R}} g_{k}(x,\xi)\varphi_1(x,\xi)d{\nu}^{1,\varepsilon}_{x,s}(\xi)dxd\beta_k(s)\\ \notag
&& +\frac{\varepsilon}{2}\int^t_0\int_{\mathbb{T}^N}\int_{\mathbb{R}}\partial_{\xi}\varphi_1(x,\xi)G^2(x,\xi)d{\nu}^{1,\varepsilon}_{x,s}(\xi)dxds\\ \notag
&&+\sum_{k\geq 1}\int^t_0\int_{\mathbb{T}^N}\int_{\mathbb{R}}\varphi_1(x,\xi) g_{k}(x,\xi)h^{\varepsilon,k}(s)d{\nu}^{1,\varepsilon}_{x,s}(\xi)dxds-\langle {m}^\varepsilon_1,\partial_{\xi} \varphi_1\rangle([0,t]), \quad a.s.,
\end{eqnarray*}
where ${f}_{1,0}=I_{u_0>\xi}$ and ${\nu}^{1,\varepsilon}_{x,s}(\xi)=-\partial_{\xi}{f}^{+}_1(s,x,\xi)=\partial_{\xi}\bar{f}_1(s,x,\xi)=\delta_{\bar{u}^{\varepsilon}(x,t)=\xi}$.
 Similarly, in view of (\ref{equ-3}), for all $\varphi_2(y,\zeta)\in C^{\infty}_c(\mathbb{T}^N_y\times \mathbb{R}_\zeta)$, we have
\begin{eqnarray*}\notag
\langle \bar{f}^{+}_2(t),\varphi_2\rangle&=&\langle \bar{f}_{2,0}, \varphi_2\rangle+\int^t_0\langle \bar{f}_2(s), a(\zeta)\cdot \nabla_y \varphi_2(y,\zeta)\rangle ds\\ \notag
&&-\sum_{k\geq 1}\int^t_0\int_{\mathbb{T}^N}\int_{\mathbb{R}} g_{k}(y,\zeta)\varphi_2 (y,\zeta)h^{\varepsilon,k}(s)d\bar{\nu}^{2,\varepsilon}_{y,s}(\zeta)dyds +\langle \bar{m}^\varepsilon_2,\partial_{\zeta} \varphi_2\rangle([0,t]),\quad a.s.,
\end{eqnarray*}
where $f_{2,0}=I_{u_0>\zeta}$ and $\bar{\nu}^{2,\varepsilon}_{y,s}(\zeta)=\partial_{\zeta}\bar{f}_2(s,y,\zeta)=-\partial_{\zeta}{f}_2(s,y,\zeta)=\delta_{v^{\varepsilon}(y,t)=\zeta}$.
Setting $\alpha(x,\xi,y,\zeta)=\varphi_1(x,\xi)\varphi_2(y,\zeta)$, using integration by parts formula and It\^{o} formula, we deduce that
\begin{eqnarray*}\notag
&& \langle\langle{f}^{+}_1(t)\bar{f}^{+}_2(t), \alpha \rangle\rangle\\ \notag
&=&\langle\langle f_{1,0} \bar{f}_{2,0}, \alpha \rangle\rangle
+\int^t_0\int_{(\mathbb{T}^N)^2}\int_{\mathbb{R}^2}f_1\bar{f}_2(a(\xi)-a(\zeta))\cdot \nabla_x \alpha d\xi d\zeta dxdyds\\ \notag
&&+\frac{\varepsilon}{2}\int^t_0\int_{(\mathbb{T}^N)^2}\int_{\mathbb{R}^2}\partial_{\xi} \alpha \bar{f}_2(s,y,\zeta) G^2(x,\xi)d{\nu}^{1,\varepsilon}_{x,s}(\xi)d\zeta dxdyds\\ \notag
&&+\sum_{k\geq 1}\int^t_0\int_{(\mathbb{T}^N)^2}\int_{\mathbb{R}^2}\bar{f}_2(s,y,\zeta)\alpha g_{k}(x,\xi)h^{\varepsilon,k}(s)d\zeta d{\nu}^{1,\varepsilon}_{x,s}(\xi)dxdyds\\ \notag
\quad &&-\sum_{k\geq 1}\int^t_0\int_{(\mathbb{T}^N)^2}\int_{\mathbb{R}^2}{f}_1(s,x,\xi)\alpha g_{k}(y,\zeta)h^{\varepsilon,k}(s)d\xi d\bar{\nu}^{2,\varepsilon}_{y,s}(\zeta)dxdyds\\ \notag
\quad &&-\int_{(0,t]}\int_{(\mathbb{T}^N)^2}\int_{\mathbb{R}^2}\bar{f}^{+}_2(s,y,\zeta)\partial_{\xi} \alpha d{m}^\varepsilon_1(x,\xi,s)d\zeta dy\\ \notag
\quad &&+\int_{(0,t]}\int_{(\mathbb{T}^N)^2}\int_{\mathbb{R}^2}{f}^{-}_1(s,x,\xi)\partial_{\zeta} \alpha d\bar{m}^\varepsilon_2(y,\zeta,s)d\xi dx\\ \notag
&&+\sqrt{\varepsilon}\sum_{k\geq 1}\int^t_0\int_{(\mathbb{T}^N)^2}\int_{\mathbb{R}^2}\bar{f}_2(s,y,\zeta)g_{k}(x,\xi)\alpha d\zeta d{\nu}^{1,\varepsilon}_{x,s}(\xi)dxdyd\beta_k(s)\\
&=:&\langle\langle {f}_{1,0}\bar{f}_{2,0}, \alpha \rangle\rangle+J_1+J_2+J_3+J_4+J_5+J_6+J_7, \quad a.s..
\end{eqnarray*}
Similarly, we get
\begin{eqnarray*}\notag
&& \langle\langle\bar{f}^{+}_1(t){f}^{+}_2(t), \alpha \rangle\rangle\\ \notag
&=&\langle\langle \bar{f}_{1,0} {f}_{2,0}, \alpha \rangle\rangle
+\int^t_0\int_{(\mathbb{T}^N)^2}\int_{\mathbb{R}^2}\bar{f}_1{f}_2(a(\xi)-a(\zeta))\cdot \nabla_x \alpha d\xi d\zeta dxdyds\\ \notag
&& -\frac{\varepsilon}{2}\int^t_0\int_{(\mathbb{T}^N)^2}\int_{\mathbb{R}^2}\partial_{\xi} \alpha {f}_2(s,y,\zeta) G^2(x,\xi)d{\nu}^{1,\varepsilon}_{x,s}(\xi)d\zeta dxdyds\\ \notag
&& -\sum_{k\geq 1}\int^t_0\int_{(\mathbb{T}^N)^2}\int_{\mathbb{R}^2}{f}_2(s,y,\zeta)\alpha g_{k}(x,\xi)h^{\varepsilon,k}(s)d\zeta d{\nu}^{1,\varepsilon}_{x,s}(\xi)dxdyds\\ \notag
&& +\sum_{k\geq 1}\int^t_0\int_{(\mathbb{T}^N)^2}\int_{\mathbb{R}^2}\bar{f}_1(s,x,\xi)\alpha g_{k}(y,\zeta)h^{\varepsilon,k}(s)d\xi d\bar{\nu}^{2,\varepsilon}_{y,s}(\zeta)dxdyds\\ \notag
 && +\int_{(0,t]}\int_{(\mathbb{T}^N)^2}\int_{\mathbb{R}^2}{f}^{-}_2(s,y,\zeta)\partial_{\xi} \alpha d{m}^\varepsilon_1(x,\xi,s)d\zeta dy\\ \notag
 && -\int_{(0,t]}\int_{(\mathbb{T}^N)^2}\int_{\mathbb{R}^2}\bar{f}^{+}_1(s,x,\xi)\partial_{\zeta} \alpha d\bar{m}^\varepsilon_2(y,\zeta,s)d\xi dx\\ \notag
&& -\sqrt{\varepsilon}\sum_{k\geq 1}\int^t_0\int_{(\mathbb{T}^N)^2}\int_{\mathbb{R}^2}{f}_2(s,y,\zeta)g_{k}(x,\xi)\alpha d\zeta d{\nu}^{1,\varepsilon}_{x,s}(\xi)dxdyd\beta_k(s)\\
&=:&\langle\langle \bar{f}_{1,0}{f}_{2,0}, \alpha \rangle\rangle+\bar{J}_1+\bar{J}_2+\bar{J}_3+\bar{J}_4+\bar{J}_5+\bar{J}_6+\bar{J}_7,\quad a.s..
\end{eqnarray*}
By the same method as the proof of Theorem 15 in \cite{D-V-1}, we have
\begin{eqnarray*}
\sup_{t\in [0,T]}(J_5(t)+J_6(t))\leq 0,\quad a.s., \quad \sup_{t\in [0,T]}(\bar{J}_5(t)+\bar{J}_6(t))\leq 0 , \quad a.s..
\end{eqnarray*}
Taking $\alpha(x,y,\xi,\zeta)=\rho_{\gamma}(x-y)\psi_{\delta}(\xi-\zeta)$, where $\rho_{\gamma}$ and $\psi_{\delta}$ are approximations to the identity on $\mathbb{T}^N$ and $\mathbb{R}$, respectively. Then, we have
\begin{eqnarray}\notag
&& \int_{(\mathbb{T}^N)^2}\int_{\mathbb{R}^2}\rho_\gamma (x-y)\psi_{\delta}(\xi-\zeta)(f^{\pm}_1(x,t,\xi)\bar{f}^{\pm}_2(y,t,\zeta)+\bar{f}^{\pm}_1(x,t,\xi){f}^{\pm}_2(y,t,\zeta))d\xi d\zeta dxdy\\ \notag
&\leq& \int_{(\mathbb{T}^N)^2}\int_{\mathbb{R}^2}\rho_\gamma (x-y)\psi_{\delta}(\xi-\zeta)(f_{1,0}(x,\xi)\bar{f}_{2,0}(y,\zeta)+\bar{f}_{1,0}(x,\xi){f}_{2,0}(y,\zeta))d\xi d\zeta dxdy\\
\label{eee-2}
&& \ +\tilde{J}_1(t)+\tilde{\bar{J}}_1(t)+\tilde{J}_2(t)+\tilde{\bar{J}}_2(t)
+\tilde{J}_3(t)+\tilde{\bar{J}}_3(t)+\tilde{J}_4(t)+\tilde{\bar{J}}_4(t)+\tilde{J}_7(t)+\tilde{\bar{J}}_7(t), \quad a.s.,
\end{eqnarray}
where $\tilde{J}_i, \tilde{\bar{J}}_i$  in (\ref{eee-2}) are the corresponding $J_i, \bar{J}_i$ with $\rho, \psi$ replaced by $\rho_{\gamma}(x-y) $ and $\psi_{\delta}(\xi-\zeta)$, respectively, for $i=1,2,3,4,7$.

Utilizing the same method as the proof of Theorem 15 in \cite{D-V-1}, it gives
\begin{eqnarray*}
\mathbb{E}\sup_{t\in [0,T]}|\tilde{J}_1(t)|\leq TC_p \delta \gamma^{-1}, \quad
 \mathbb{E}\sup_{t\in [0,T]}|\tilde{\bar{J}}_1(t)|\leq TC_p \delta \gamma^{-1}.
\end{eqnarray*}
%
{ With the aid of $\gamma_1(\xi,\zeta)=\gamma_2(\xi,\zeta)$ and by using (\ref{equ-28}), we have
\begin{eqnarray*}\notag
\tilde{\bar{J}}_2&=&\tilde{J}_2\\
&=&\frac{\varepsilon}{2}\int^t_0\int_{(\mathbb{T}^N)^2}\int_{\mathbb{R}^2}\alpha  G^2(x,\xi)d\nu^{1,\varepsilon}_{x,s}\otimes \bar{\nu}^{2,\varepsilon}_{y,s}(\xi,\zeta)dxdyds\\ \notag
&\leq& \frac{\varepsilon}{2}D_0\int^t_0\int_{(\mathbb{T}^N)^2}\int_{\mathbb{R}^2}\alpha(1+|\xi|^2)  d\nu^{1,\varepsilon}_{x,s}\otimes \bar{\nu}^{2,\varepsilon}_{y,s}(\xi,\zeta)dxdyds\\ \notag
&\leq&\frac{\varepsilon}{2}D_0\int^t_0\int_{(\mathbb{T}^N)^2}\int_{\mathbb{R}^2}\alpha  d\nu^{1,\varepsilon}_{x,s}\otimes \bar{\nu}^{2,\varepsilon}_{y,s}(\xi,\zeta)dxdyds\\ \notag
&& +\frac{\varepsilon}{2}D_0\int^t_0\int_{(\mathbb{T}^N)^2}\int_{\mathbb{R}^2}\alpha  |\xi|^2d\nu^{1,\varepsilon}_{x,s}\otimes \bar{\nu}^{2,\varepsilon}_{y,s}(\xi,\zeta)dxdyds.
\end{eqnarray*}
Clearly, it holds that
\begin{eqnarray}\notag
&&\mathbb{E}\int_{(\mathbb{T}^N)^2}\int_{\mathbb{R}^2}\alpha  d\nu^{1,\varepsilon}_{x,s}\otimes \bar{\nu}^{2,\varepsilon}_{y,s}(\xi,\zeta)dxdy\\ \notag
&\leq& \mathbb{E} \|\psi_{\delta}\|_{L^{\infty}}\int_{(\mathbb{T}^N)^2}\int_{\mathbb{R}^2}\rho_\gamma(x-y)d\nu^{1,\varepsilon}_{x,s}\otimes \bar{\nu}^{2,\varepsilon}_{y,s}(\xi,\zeta)dxdy\\ \notag
&\leq & \|\psi_{\delta}\|_{L^{\infty}}\int_{(\mathbb{T}^N)^2}\rho_\gamma(x-y)dxdy\\
\label{q-1}
&\leq&\delta^{-1}.
\end{eqnarray}
Moreover, by utilizing the property that measures $\nu^{1,\varepsilon}_{x,s}$ and $\bar{\nu}^{2,\varepsilon}_{y,s}$ vanish at infinity, it follows that
\begin{eqnarray}\notag
&&\mathbb{E}\int_{(\mathbb{T}^N)^2}\int_{\mathbb{R}^2}\alpha  |\xi|^2d\nu^{1,\varepsilon}_{x,s}\otimes \bar{\nu}^{2,\varepsilon}_{y,s}(\xi,\zeta)dxdy\\ \notag
&\leq&\mathbb{E}\int_{(\mathbb{T}^N)^2}\rho_\gamma(x-y)
\int_{\mathbb{R}^2}\psi_{\delta}(\xi-\zeta)|\xi|^2d\nu^{1,\varepsilon}_{x,s}\otimes \bar{\nu}^{2,\varepsilon}_{y,s}(\xi,\zeta)dxdy\\ \notag
&\leq& \mathbb{E}\|\psi_{\delta}\|_{L^{\infty}} \int_{(\mathbb{T}^N)^2}\rho_\gamma(x-y)\int_{\mathbb{R}^2}|\xi|^2d\nu^{1,\varepsilon}_{x,s}\otimes \bar{\nu}^{2,\varepsilon}_{y,s}(\xi,\zeta)dxdy\\ \notag
&\leq& C\delta^{-1}\int_{(\mathbb{T}^N)^2}\rho_\gamma(x-y)dxdy\\
\label{q-2}
&\leq& C\delta^{-1}.
\end{eqnarray}
Hence, combining (\ref{q-1}) and (\ref{q-2}), we deduce that
\begin{eqnarray*}
\mathbb{E}\sup_{t\in [0,T]}\tilde{\bar{J}}_2(t)=\mathbb{E}\sup_{t\in [0,T]}\tilde{J}_2(t)
\leq \frac{\varepsilon}{2} D_0T\delta^{-1}+\frac{\varepsilon}{2}C D_0T\delta^{-1}
\leq\varepsilon C D_0T\delta^{-1}.
\end{eqnarray*}
}
Recall
\begin{eqnarray}\label{eee-3}
\gamma_2(\zeta,\xi)=\int^{\infty}_{\zeta}\psi_{\delta}(\xi-\zeta')d\zeta'.
\end{eqnarray}
Using the similar arguments as in the proof of Proposition \ref{prp-1}, we have
\begin{eqnarray*}
\tilde{\bar{J}}_3+\tilde{\bar{J}}_4&=&\tilde{J}_3+\tilde{J}_4\\
&=& \sum_{k\geq 1}\int^t_0\int_{(\mathbb{T}^N)^2}\int_{\mathbb{R}^2}\gamma_2(\zeta,\xi)\rho_\gamma(x-y)\Big(g_{k}(x,\xi) -g_{k}(y,\zeta)\Big)h^{\varepsilon,k}(s)d \nu^{1,\varepsilon}_{x,s}\otimes \bar{\nu}^{2,\varepsilon}_{y,s}(\xi,\zeta) dxdyds\\
&\leq& \sum_{k\geq 1}\int^t_0\int_{(\mathbb{T}^N)^2}\int_{\mathbb{R}^2}\gamma_2(\zeta,\xi)\rho_\gamma(x-y)|g_{k}(x,\xi)-g_{k}(y,\zeta)||h^{\varepsilon,k}(s)| d \nu^{1,\varepsilon}_{x,s}\otimes \bar{\nu}^{2,\varepsilon}_{y,s}(\xi,\zeta) dxdyds\\
&\leq& \int^t_0\int_{(\mathbb{T}^N)^2}\int_{\mathbb{R}^2}\gamma_2(\zeta,\xi)\rho_\gamma(x-y)\Big(\sum_{k\geq 1}|g_{k}(x,\xi)-g_{k}(y,\zeta)|^2\Big)^{\frac{1}{2}}\Big(\sum_{k\geq 1}|h^{\varepsilon,k}(s)|^2\Big)^{\frac{1}{2}} d \nu^{1,\varepsilon}_{x,s}\otimes \bar{\nu}^{2,\varepsilon}_{y,s}(\xi,\zeta) dxdyds\\
&\leq& \sqrt{D_1}\int^t_0|h^\varepsilon(s)|_{U}\int_{(\mathbb{T}^N)^2}\int_{\mathbb{R}^2}\gamma_2(\zeta,\xi)\rho_\gamma(x-y)|x-y| d \nu^{1,\varepsilon}_{x,s}\otimes \bar{\nu}^{2,\varepsilon}_{y,s}(\xi,\zeta) dxdyds\\
&&\ + \sqrt{D_1}\int^t_0|h^\varepsilon(s)|_{U}
\int_{(\mathbb{T}^N)^2}\rho_\gamma(x-y)\int_{\mathbb{R}^2}\gamma_2(\zeta,\xi)|\xi-\zeta|d \nu^{1,\varepsilon}_{x,s}\otimes \bar{\nu}^{2,\varepsilon}_{y,s}(\xi,\zeta) dxdyds\\
&=:& \tilde{J}_{3,1}+\tilde{J}_{4,1}.
\end{eqnarray*}
By
\begin{eqnarray*}
\int_{(\mathbb{T}^N)^2}\rho_\gamma(x-y)|x-y|dxdy&\leq& \gamma,\\
\mathbb{E}\int_{(\mathbb{T}^N)^2}\gamma_2(\zeta,\xi) d \nu^{1,\varepsilon}_{x,s}\otimes \bar{\nu}^{2,\varepsilon}_{y,s}(\xi,\zeta)&\leq& 1,
\end{eqnarray*}
it follows that
\begin{eqnarray*}
\mathbb{E}\sup_{t\in [0,T]}\tilde{J}_{3,1}(t)\leq \sqrt{D_1}\gamma(T+M).
\end{eqnarray*}
Using the same method as the estimate of $\tilde{K}_{2,2}$ in Theorem \ref{thm-2}, we have
\begin{eqnarray*}
\tilde{J}_{4,1}&\leq& 2\sqrt{D_1}\delta(T+M)\\
&& +\sqrt{D_1}\int^t_0|h^\varepsilon(s)|_{U}
\int_{(\mathbb{T}^N)^2}\int_{\mathbb{R}^2}\rho_\gamma (x-y)\psi_{\delta}(\xi-\zeta)(f^{\pm}_1\bar{f}^{\pm}_2+\bar{f}^{\pm}_1{f}^{\pm}_2)d\xi d\zeta dxdyds, \quad a.s..
\end{eqnarray*}
Combining all the previous estimates, it follows that
\begin{eqnarray*}
&& \int_{(\mathbb{T}^N)^2}\int_{\mathbb{R}^2}\rho_\gamma (x-y)\psi_{\delta}(\xi-\zeta)(f^{\pm}_1(x,t,\xi)\bar{f}^{\pm}_2(y,t,\zeta)+\bar{f}^{\pm}_1(x,t,\xi){f}^{\pm}_2(y,t,\zeta))d\xi d\zeta dxdy\\ \notag
&\leq& \int_{\mathbb{T}^N}\int_{\mathbb{R}}(f_{1,0}(x,\xi)\bar{f}_{2,0}(x,\xi)+\bar{f}_{1,0}(x,\xi){f}_{2,0}(x,\xi))dxd\xi +|\mathcal{E}_0(\gamma,\delta)|\\
&&\ +|\tilde{J}_1(t)|+|\tilde{\bar{J}}_1(t)|+2\tilde{\bar{J}}_2(t)+2\tilde{J}_{3,1}(t)+|\tilde{J}_7|(t)+|\tilde{\bar{J}}_7|(t)+4\sqrt{D_1}\delta(T+M)\\
&&\ +2\sqrt{D_1}\int^t_0|h^\varepsilon(s)|_{U}
\int_{(\mathbb{T}^N)^2}\int_{\mathbb{R}^2}\rho_\gamma (x-y)\psi_{\delta}(\xi-\zeta)(f^{\pm}_1\bar{f}^{\pm}_2+\bar{f}^{\pm}_1{f}^{\pm}_2)d\xi d\zeta dxdyds, \quad a.s..
\end{eqnarray*}
Applying Gronwall inequality, we get
\begin{eqnarray*}
&& \int_{(\mathbb{T}^N)^2}\int_{\mathbb{R}^2}\rho_\gamma (x-y)\psi_{\delta}(\xi-\zeta)(f^{\pm}_1(x,t,\xi)\bar{f}^{\pm}_2(y,t,\zeta)+\bar{f}^{\pm}_1(x,t,\xi){f}^{\pm}_2(y,t,\zeta))d\xi d\zeta dxdy\\ \notag
&\leq& e^{2\sqrt{D_1}(T+M)}\Big[\int_{\mathbb{T}^N}\int_{\mathbb{R}}(f_{1,0}\bar{f}_{2,0}+\bar{f}_{1,0}{f}_{2,0})dxd\xi+|\mathcal{E}_0(\gamma,\delta)|+4\sqrt{D_1}\delta(T+M)\Big]\\
&&\ +e^{2\sqrt{D_1}(T+M)}\Big[|\tilde{J}_1(t)|+|\tilde{\bar{J}}_1(t)|+2\tilde{\bar{J}}_2(t)+2\tilde{J}_{3,1}(t)+|\tilde{J}_7|(t)+|\tilde{\bar{J}}_7|(t)\Big],\quad a.s..
\end{eqnarray*}
Thus, collecting all the above estimates, we deduce that
\begin{eqnarray}\notag
&&\int_{\mathbb{T}^N}\int_{\mathbb{R}}(f^{\pm}_1(x,t,\xi)\bar{f}^{\pm}_2(x,t,\xi)+\bar{f}^{\pm}_1(x,t,\xi){f}^{\pm}_2(x,t,\xi))dxd\xi\\ \notag
&=&\int_{(\mathbb{T}^N)^2}\int_{\mathbb{R}^2}(f^{\pm}_1(x,t,\xi)\bar{f}^{\pm}_2(y,t,\zeta)+\bar{f}^{\pm}_1(x,t,\xi){f}^{\pm}_2(y,t,\zeta))\rho_{\gamma}(x-y)\psi_{\delta}(\xi-\zeta)dxdyd\xi d\zeta+\mathcal{E}_t(\gamma,\delta)\\
\notag
&\leq& e^{2\sqrt{D_1}(T+M)}\Big[\int_{\mathbb{T}^N}\int_{\mathbb{R}}(f_{1,0}\bar{f}_{2,0}+\bar{f}_{1,0}{f}_{2,0})dxd\xi+|\mathcal{E}_0(\gamma,\delta)|+4\sqrt{D_1}\delta(T+M)\Big]\\
\notag
&&\ +e^{2\sqrt{D_1}(T+M)}\Big[|\tilde{J}_1(t)|+|\tilde{\bar{J}}_1(t)|+2\tilde{\bar{J}}_2(t)+2\tilde{J}_{3,1}(t)+|\tilde{J}_7|(t)+|\tilde{\bar{J}}_7|(t)\Big]+\mathcal{E}_t(\gamma,\delta)\\
\notag
&=:& e^{2\sqrt{D_1}(T+M)}\int_{\mathbb{T}^N}\int_{\mathbb{R}}(f_{1,0}\bar{f}_{2,0}+\bar{f}_{1,0}{f}_{2,0})dxd\xi\\
\label{qeq-21}
&&
 +e^{2\sqrt{D_1}(T+M)}(|\tilde{J}_7|(t)+|\tilde{\bar{J}}_7|(t))+r(\varepsilon,\gamma,\delta,t), \quad a.s.,
\end{eqnarray}
where the remainder is given by
 \begin{eqnarray*}
r(\varepsilon,\gamma,\delta,t)=e^{2\sqrt{D_1}(T+M)}[|\tilde{J}_1(t)|+|\tilde{\bar{J}}_1(t)|+2\tilde{\bar{J}}_2(t)+4\sqrt{D_1}\delta(T+M)+|\mathcal{E}_0(\gamma,\delta)]|+\mathcal{E}_t(\gamma,\delta).
\end{eqnarray*}

Applying the Burkholder-Davis-Gundy inequality, and utilizing (\ref{eee-3}), (\ref{equ-28}) that
\begin{eqnarray*}
\mathbb{E}\sup_{t\in [0,T]}|\tilde{J}_7|(t)&\leq&\sqrt{\varepsilon}\mathbb{E}\sup_{t\in [0,T]}|\sum_{k\geq 1}\int^t_0\int_{(\mathbb{T}^N)^2}\int_{\mathbb{R}^2}\bar{f}_2(s,y,\zeta)g_{k}(x,\xi)\alpha d\zeta d{\nu}^{1,\varepsilon}_{x,s}(\xi)dxdyd\beta_k(s)|\\
&=& \sqrt{\varepsilon}\mathbb{E}\sup_{t\in [0,T]}|\sum_{k\geq 1}\int^t_0\int_{(\mathbb{T}^N)^2}\int_{\mathbb{R}^2}\bar{f}_2(s,y,\zeta)\partial_{\zeta} \gamma_2(\xi,\zeta) \rho_{\gamma}(x-y)g_{k}(x,\xi) d\zeta d{\nu}^{1,\varepsilon}_{x,s}(\xi)dxdyd\beta_k(s)|\\
&=& \sqrt{\varepsilon}\mathbb{E}\sup_{t\in [0,T]}|\sum_{k\geq 1}\int^t_0\int_{(\mathbb{T}^N)^2}\int_{\mathbb{R}^2}\gamma_2(\xi,\zeta) \rho_{\gamma}(x-y)g_{k}(x,\xi) d\nu^{1,\varepsilon}_{x,s}\otimes \bar{\nu}^{2,\varepsilon}_{y,s}(\xi,\zeta)dxdyd\beta_k(s)|\\
&\leq& \sqrt{\varepsilon}\mathbb{E}\Big[\int^T_0\int_{(\mathbb{T}^N)^2}\int_{\mathbb{R}^2}\gamma^2_2(\xi,\zeta) \rho^2_{\gamma}(x-y)\Big(\sum_{k\geq 1}g^2_{k}(x,\xi)\Big)d\nu^{1,\varepsilon}_{x,s}\otimes \bar{\nu}^{2,\varepsilon}_{y,s}(\xi,\zeta)dxdyds\Big]^{\frac{1}{2}}\\
&\leq& \sqrt{\varepsilon}\sqrt{D_0}\mathbb{E}\Big[\int^T_0\int_{(\mathbb{T}^N)^2}\int_{\mathbb{R}^2}\gamma^2_2(\xi,\zeta) \rho^2_{\gamma}(x-y)(1+|\xi|^2)d\nu^{1,\varepsilon}_{x,s}\otimes \bar{\nu}^{2,\varepsilon}_{y,s}(\xi,\zeta)dxdyds\Big]^{\frac{1}{2}}\\
&\leq& \sqrt{\varepsilon}\sqrt{D_0}\Big[\mathbb{E}\int^T_0\int_{(\mathbb{T}^N)^2}\int_{\mathbb{R}^2}\gamma^2_2(\xi,\zeta) \rho^2_{\gamma}(x-y)(1+|\xi|^2)d\nu^{1,\varepsilon}_{x,s}\otimes \bar{\nu}^{2,\varepsilon}_{y,s}(\xi,\zeta)dxdyds\Big]^{\frac{1}{2}}\\
&\leq&\sqrt{\varepsilon}\sqrt{D_0}\gamma^{-N}\Big[\mathbb{E}\int^T_0\int_{(\mathbb{T}^N)^2}\int_{\mathbb{R}^2} \gamma^2_2(\xi,\zeta)(1+|\xi|^2)d\nu^{1,\varepsilon}_{x,s}\otimes \bar{\nu}^{2,\varepsilon}_{y,s}(\xi,\zeta)dxdyds\Big]^{\frac{1}{2}}.
\end{eqnarray*}
Taking into account the following fact
\begin{eqnarray*}
&&\mathbb{E}\int^T_0\int_{(\mathbb{T}^N)^2}\int_{\mathbb{R}^2}\gamma^2_2(\xi,\zeta)(1+|\xi|^2)d\nu^{1,\varepsilon}_{x,s}\otimes \bar{\nu}^{2,\varepsilon}_{y,s}(\xi,\zeta)dxdyds\\
&\leq& \mathbb{E}\int^T_0\int_{(\mathbb{T}^N)^2}\int_{\mathbb{R}^2}(1+|\xi|^2)d\nu^{1,\varepsilon}_{x,s}\otimes \bar{\nu}^{2,\varepsilon}_{y,s}(\xi,\zeta)dxdyds\leq CT,
\end{eqnarray*}
we further deduce that
\begin{eqnarray*}
\mathbb{E}\sup_{t\in [0,T]}|\tilde{J}_7|(t)
\leq C\sqrt{\varepsilon}\sqrt{D_0 T}\gamma^{-N}.
\end{eqnarray*}
By the same method as above, we deduce that
\begin{eqnarray*}
\mathbb{E}\sup_{t\in [0,T]}|\tilde{\bar{J}}_7|(t)
\leq C\sqrt{\varepsilon}\sqrt{D_0 T}\gamma^{-N}.
\end{eqnarray*}
For the remainder, we have
\begin{eqnarray}\notag
 \mathbb{E}\underset{0\leq t\leq T}{{\rm{ess\sup}}}\ r(\varepsilon,\gamma,\delta,t)
&\leq& 2e^{2\sqrt{D_1}(T+M)}[TC_p \delta \gamma^{-1}+\varepsilon C D_0T\delta^{-1} +\sqrt{D_1}(2\delta+\gamma)(T+M)]\\
\label{qeq-23}
&&\ +(e^{2\sqrt{D_1}(T+M)}+1)\mathbb{E}\underset{0\leq t\leq T}{{\rm{ess\sup}}}\ |\mathcal{E}_t(\gamma,\delta)|.
\end{eqnarray}
{In the following, we aim to make estimates of the error term $\mathbb{E}\underset{0\leq t\leq T}{{\rm{ess\sup}}}\ |\mathcal{E}_t(\gamma, \delta)|$ by utilizing a similar method as the proof of Proposition 6.1 and Theorem 6.2 in \cite{DHV}.

For any $t\in [0,T]$, we have
\begin{eqnarray*}
\mathcal{E}_t(\gamma, \delta)&=&\int_{\mathbb{T}^N}\int_{\mathbb{R}}(f^{\pm}_1(x,t,\xi)\bar{f}^{\pm}_2(x,t,\xi)+\bar{f}^{\pm}_1(x,t,\xi){f}^{\pm}_2(x,t,\xi))d\xi dx
\\
&& -\int_{(\mathbb{T}^N)^2}\int_{\mathbb{R}^2}(f^{\pm}_1(x,t,\xi)\bar{f}^{\pm}_2(y,t,\zeta)+\bar{f}^{\pm}_1(x,t,\xi){f}^{\pm}_2(y,t,\zeta))\rho_{\gamma}(x-y)\psi_{\delta}(\xi-\zeta)dxdyd\xi d\zeta \\
&=& \Big[\int_{\mathbb{T}^N}\int_{\mathbb{R}}(f^{\pm}_1(x,t,\xi)\bar{f}^{\pm}_2(x,t,\xi)+\bar{f}^{\pm}_1(x,t,\xi){f}^{\pm}_2(x,t,\xi))d\xi dx\\
&& -\int_{(\mathbb{T}^N)^2}\int_{\mathbb{R}}\rho_{\gamma}(x-y)(f^{\pm}_1(x,t,\xi)\bar{f}^{\pm}_2(y,t,\xi)+\bar{f}^{\pm}_1(x,t,\xi){f}^{\pm}_2(y,t,\xi))d\xi dxdy\Big]\\
&& +\Big[\int_{(\mathbb{T}^N)^2}\int_{\mathbb{R}}\rho_{\gamma}(x-y)(f^{\pm}_1(x,t,\xi)\bar{f}^{\pm}_2(y,t,\xi)+\bar{f}^{\pm}_1(x,t,\xi){f}^{\pm}_2(y,t,\xi))d\xi dxdy\\
&& -\int_{(\mathbb{T}^N)^2}\int_{\mathbb{R}^2}(f^{\pm}_1(x,t,\xi)\bar{f}^{\pm}_2(y,t,\zeta)+\bar{f}^{\pm}_1(x,t,\xi){f}^{\pm}_2(y,t,\zeta))\rho_{\gamma}(x-y)\psi_{\delta}(\xi-\zeta)dxdyd\xi d\zeta\Big]\\
&=:&H_1+H_2,
\end{eqnarray*}
Applying the same method as (\ref{qeq-14}) and (\ref{qeq-14-1}), it follows that
\begin{eqnarray}\label{qeq-10}
|H_2(t)|\leq 2\delta,\quad a.s..
\end{eqnarray}
Moreover, it is easy to deduce that
\begin{eqnarray*}
|H_1(t)|&\leq& \Big|\int_{(\mathbb{T}^N)^2}\rho_{\gamma}(x-y)\int_{\mathbb{R}}I_{\bar{u}^{\varepsilon, \pm}(x,t)>\xi}(I_{v^{\varepsilon,\pm}(x,t)\leq \xi}-I_{v^{\varepsilon,\pm}(y,t)\leq \xi})d\xi dxdy\Big|\\
&& +\Big|\int_{(\mathbb{T}^N)^2}\rho_{\gamma}(x-y)\int_{\mathbb{R}}I_{\bar{u}^{\varepsilon, \pm}(x,t)\leq\xi}(I_{v^{\varepsilon,\pm}(x,t)> \xi}-I_{v^{\varepsilon,\pm}(y,t)> \xi})d\xi dxdy\Big|\\
&\leq& 2\int_{(\mathbb{T}^N)^2}\rho_{\gamma}(x-y)|v^{\varepsilon,\pm}(x,t)-v^{\varepsilon,\pm}(y,t)|dxdy.
\end{eqnarray*}
By (\ref{qeq-10}) and (\ref{qeq-18}), we have
\begin{eqnarray*}
&&\mathbb{E} \underset{0\leq t\leq T}{{\rm{ess\sup}}}\ \int_{(\mathbb{T}^N)^2}\rho_{\gamma}(x-y)|v^{\varepsilon,\pm}(x,t)-v^{\varepsilon,\pm}(y,t)|dxdy\\
&=& \mathbb{E}\underset{0\leq t\leq T}{{\rm{ess\sup}}}\ \int_{(\mathbb{T}^N)^2}\int_{\mathbb{R}}\rho_{\gamma}(x-y)(f^{\pm}_2(x,t,\xi)\bar{f}^{\pm}_2(y,t,\xi)+\bar{f}^{\pm}_2(x,t,\xi){f}^{\pm}_2(y,t,\xi))d\xi dxdy\\
&\leq& \mathbb{E}\underset{0\leq t\leq T}{{\rm{ess\sup}}}\  \int_{(\mathbb{T}^N)^2}\int_{\mathbb{R}^2}\rho_{\gamma}(x-y)\psi_{\delta}(\xi-\zeta)(f^{\pm}_2(x,t,\xi)\bar{f}^{\pm}_2(y,t,\zeta)+\bar{f}^{\pm}_2(x,t,\xi){f}^{\pm}_2(y,t,\zeta))d\xi d\zeta dxdy+2\delta\\
&\leq& e^{2\sqrt{D_1}(T+M)}\Big[\int_{\mathbb{T}^N}\int_{\mathbb{R}}(f_{2,0}\bar{f}_{2,0}+\bar{f}_{2,0}{f}_{2,0})d\xi dx+|\mathcal{E}_0(\gamma,\delta)|\Big]\\
&&\ +2e^{2\sqrt{D_1}(T+M)}[TC_p \delta \gamma^{-1}+\sqrt{D_1}(\gamma+2\delta)(T+M)]+2\delta\\
&\leq &e^{2\sqrt{D_1}(T+M)}|\mathcal{E}_0(\gamma,\delta)|+2e^{2\sqrt{D_1}(T+M)}[TC_p \delta \gamma^{-1}+\sqrt{D_1}(\gamma+2\delta)(T+M)]+2\delta,
\end{eqnarray*}
where $|\mathcal{E}_0(\gamma, \delta)|\rightarrow 0$, when $\gamma, \delta\rightarrow 0$. Then,
\begin{eqnarray*}
\mathbb{E}\underset{0\leq t\leq T}{{\rm{ess\sup}}}\ |H_1(t)|&\leq& 4\delta+2 e^{2\sqrt{D_1}(T+M)}\mathcal{E}_0(\gamma,\delta)+4e^{2\sqrt{D_1}(T+M)}[TC_p \delta \gamma^{-1}+\sqrt{D_1}(\gamma+2\delta)(T+M)].
\end{eqnarray*}
Combining all the above estimates, we conclude that
\begin{eqnarray*}
\mathbb{E}\underset{0\leq t\leq T}{{\rm{ess\sup}}}\ |\mathcal{E}_t(\gamma, \delta)|\leq 6\delta+2e^{2\sqrt{D_1}(T+M)}|\mathcal{E}_0(\gamma,\delta)|+4e^{2\sqrt{D_1}(T+M)}[TC_p \delta \gamma^{-1}+\sqrt{D_1}(\gamma+2\delta)(T+M)].
\end{eqnarray*}
Hence, we deduce from (\ref{qeq-23}) that
\begin{eqnarray*}
&&\mathbb{E}\underset{0\leq t\leq T}{{\rm{ess\sup}}}\ r(\varepsilon,\gamma,\delta,t)\\
&\leq& 2e^{2\sqrt{D_1}(T+M)}[TC_p \delta \gamma^{-1}+\varepsilon C D_0T\delta^{-1} +\sqrt{D_1}(2\delta+\gamma)(T+M)]\\
&&\ +6(e^{2\sqrt{D_1}(T+M)}+1)\delta+2e^{2\sqrt{D_1}(T+M)}(e^{\sqrt{D_1}(T+M)}+1)|\mathcal{E}_0(\gamma,\delta)|\\
&& +4e^{2\sqrt{D_1}(T+M)}(e^{\sqrt{D_1}(T+M)}+1)[TC_p \delta \gamma^{-1}+\sqrt{D_1}(\gamma+2\delta)(T+M)].
\end{eqnarray*}
}
Letting
\[
\delta=\gamma^{\frac{4}{3}},\quad \gamma=\varepsilon^{\frac{1}{2(1+N)}},
\]
then,
\begin{eqnarray}\label{e-47}
\mathbb{E}\sup_{t\in [0,T]}|\tilde{J}_7|(t)
\leq C\sqrt{D_0 T}\varepsilon^{\frac{1}{2(1+N)}}\rightarrow 0, \quad \varepsilon \rightarrow 0,\\
\label{e-48}
\mathbb{E}\sup_{t\in [0,T]}|\tilde{\bar{J}}_7|(t)
\leq C\sqrt{D_0 T}\varepsilon^{\frac{1}{2(1+N)}}\rightarrow 0, \quad \varepsilon \rightarrow 0,
\end{eqnarray}
and
\begin{eqnarray}\notag
&& \mathbb{E}\underset{0\leq t\leq T}{{\rm{ess\sup}}}\ r(\varepsilon,\gamma,\delta,t)\\ \notag
 &\leq& 2e^{2\sqrt{D_1}(T+M)}[TC_p \varepsilon^{\frac{1}{6(1+N)}}+ C D_0T\varepsilon^{\frac{1+3N}{3(1+N)}} +\sqrt{D_1}(2\varepsilon^{\frac{2}{3(1+N)}}+\varepsilon^{\frac{1}{2(1+N)}})(T+M)]\\ \notag
 &&\ +6(e^{2\sqrt{D_1}(T+M)}+1)\varepsilon^{\frac{2}{3(1+N)}}+2e^{2\sqrt{D_1}(T+M)}(e^{2\sqrt{D_1}(T+M)}+1)\mathcal{E}_0(\gamma,\delta)\\ \notag
 &&\ +4e^{2\sqrt{D_1}(T+M)}(e^{2\sqrt{D_1}(T+M)}+1)[TC_p \varepsilon^{\frac{1}{6(1+N)}}+\sqrt{D_1}(2\varepsilon^{\frac{2}{3(1+N)}}+\varepsilon^{\frac{1}{2(1+N)}})(T+M)]\\
\label{qeq-12}
 &\rightarrow&0, \quad {\rm{as}} \quad \varepsilon\rightarrow0.
\end{eqnarray}
Notice that $f_1=I_{\bar{u}^\varepsilon>\xi}$ and $f_2=I_{v^{\varepsilon}>\xi}$ with initial data $f_{1,0}=I_{u_0>\xi}$ and ${f}_{2,0}=I_{u_0>\xi}$, respectively. With the help of identity (\ref{equ-1}), we deduce from (\ref{qeq-21}) that
\begin{eqnarray*}\notag
\mathbb{E}\underset{0\leq t\leq T}{{\rm{ess\sup}}}\ \|\bar{u}^\varepsilon(t)-v^\varepsilon(t)\|_{L^1(\mathbb{T}^N)}
\leq e^{2\sqrt{D_1}(T+M)}\Big(\mathbb{E}\sup_{t\in [0,T]}|\tilde{J}_7|(t)+\mathbb{E}\sup_{t\in [0,T]}|\tilde{\bar{J}}_7|(t)\Big)+\mathbb{E}\underset{0\leq t\leq T}{{\rm{ess\sup}}}\ r(\varepsilon,\gamma,\delta,t).
\end{eqnarray*}
Hence, it follows from (\ref{e-47}), (\ref{e-48}) and (\ref{qeq-12}) that
\begin{eqnarray*}
&&\mathbb{E}\|\bar{u}^\varepsilon-v^\varepsilon\|_{L^1([0,T];L^1(\mathbb{T}^N))}\\ \notag
&\leq& T\cdot \mathbb{E} \underset{0\leq t\leq T}{{\rm{ess\sup}}}\ \|\bar{u}^\varepsilon(t)-v^\varepsilon(t)\|_{L^1(\mathbb{T}^N)}\rightarrow 0,
\end{eqnarray*}
which implies that $\|\bar{u}^\varepsilon-v^\varepsilon\|_{L^1([0,T];L^1(\mathbb{T}^N))}\rightarrow 0$ in probability, as $\varepsilon\rightarrow 0$. We complete the proof.
\end{proof}

\vskip 0.2cm
\noindent{\bf  Acknowledgements}\quad  The authors are grateful to the anonymous referees for their constructive comments and  valuable suggestions. This work is partly supported by National Natural Science Foundation of China (No. 11431014,11801032, 11671372, 11721101). Key Laboratory of Random Complex Structures and Data Science, Academy of Mathematics and Systems Science, Chinese Academy of Sciences (No. 2008DP173182). China Postdoctoral Science Foundation funded project (No. 2018M641204).

\def\refname{ References}

\end{document}